\documentclass[12pt]{article}
\usepackage[utf8]{inputenc}
\usepackage{color}
\usepackage{xcolor}
\definecolor{note_fontcolor}{rgb}{0.800781, 0.800781, 0.800781}
\usepackage{amsmath}
\usepackage{amsthm}
\usepackage{amssymb}
\usepackage[letterpaper]{geometry}
\PassOptionsToPackage{normalem}{ulem}
\usepackage{ulem}
\usepackage{enumitem}
\usepackage{soul}
\usepackage[unicode=true,pdfusetitle, bookmarks=true,bookmarksnumbered=false,bookmarksopen=false,
 breaklinks=false,pdfborder={0 0 1},backref=page,colorlinks=true]
 {hyperref}
\hypersetup{
 linkcolor=blue, citecolor=blue}
\usepackage{graphicx} 
\usepackage{dsfont} 
\usepackage[all]{xy} 
\usepackage{bbm} 

\newtheorem{theorem}{Theorem}[section]
\newtheorem{cor}[theorem]{Corollary}
\newtheorem{lemma}[theorem]{Lemma}
\newtheorem{prop}[theorem]{Proposition}
\newtheorem{remark}[theorem]{Remark}
\newtheorem{conj}[theorem]{Conjecture}

\newtheorem{defn}[theorem]{Definition}
\newtheorem{exm}[theorem]{Example}

\numberwithin{equation}{section}

\def\Z{\mathbb{Z}}
\def\R{\mathbb{R}}
\def\N{\mathbb{N}}
\def\Q{\mathbb{Q}}
\def\C{\mathbb{C}}

\def\U{\mathrm{U}}
\def\L{{\cal L}}
\def\Reg{\mathrm{Reg}}

\def\Regun{\mathrm{Reg}_{\U(n)}}
\def\std{\mathrm{std}}
\def\triv{\mathrm{triv}}
\def\supp{\mathrm{supp}}
\def\sym{\mathrm{Sym}}  

\def\defi{\stackrel{\text{def}}{=}}

\def\reg{\mathrm{Reg}}
\def\N{{\cal N}}
\def\MS{{\cal MS}}
\def\pure{{\cal PURE}} 
\def\ii{\mathbf{I}} 
\def\jj{\mathbf{J}}
\def\jjj{\mathfrak{J}} 

\def\gl{\mathrm{GL}}
\def\irr{\mathrm{Irr}}
\def\kmp{\mathcal{KMP}}

\def\End{\mathrm{End}}
\def\rkl{R_{k,m}}
\def\Sym{\mathrm{Sym}} 
\def\skm{S_{k,m}}
\def\torinv{\mathrm{TorInv}}
\def\wg{\mathrm{Wg}}
\def\sign{\mathrm{sgn}}
\def\under{\underline{\hspace{3mm}}} 
\def\gnl{\Gamma^n_\ell} 
\def\gn{\Gamma^n} 

\newcommand\ms[2]{\left(\!\left(\begin{smallmatrix} #1 \\ #2 \end{smallmatrix} \right)\!\right)}

\title{Aldous-type Spectral Gaps in Unitary Groups}
\author{Gil Alon and Doron Puder}

\begin{document}
\maketitle
\begin{abstract}


Aldous' spectral gap conjecture, proven by Caputo, Liggett and Richthammer, states the following: for any set of transpositions in the symmetric group $\sym(n)$, the spectral gap of the corresponding random walk on the group -- an $n!$-state process -- coincides with that of the corresponding random walk of a single element -- an $n$-state process.

This paper presents an analog of this conjecture in the unitary group $\U(n)$, and proves it in several non-trivial cases. The phenomenon we discover is that for some natural families of probability distributions on $\U(n)$, the spectral gap of the corresponding random walk, which has a continuous state space, is identical to that of a discrete KMP process (also known as the uniform reshuffling process) with two indistinguishable particles on a hypergraph on $n$ vertices -- a discrete Markov chain with $\binom{n+1}{2}$ states.

\end{abstract}
\tableofcontents{}    

\section{Introduction}

This paper grew out of attempts to understand the depth and width of the phenomenon known as the Aldous spectral gap conjecture. This conjecture, which was open for nearly two decades until it was settled in \cite{caputo2010proof}, states the following: For any weighted set $\Sigma$ of transpositions in the symmetric group $\sym(n)$, the spectral gap of the Cayley graph $\mathrm{Cay}(\sym(n),\Sigma)$ is identical to that of the Schreier graph $\mathrm{Sch}(\sym(n)\curvearrowright[n],\Sigma)$ -- a graph on $n$ vertices depicting the standard action of $\sym(n)$ on the numbers $[n]\defi\{1,\ldots,n\}$.  This is a remarkable result in that the spectral gap of a process with $n!$ states is completely determined by the much smaller process with only $n$ states. 

Our intuition is that such an astonishing theorem cannot be isolated and must hint at a broader phenomenon. Some attempts to generalize the phenomenon (in this perspective of group theory) were made by several researchers, e.g.,  \cite{piras2010generalizations,parzanchevski2020aldous,cesi2020spectral,Ghosh23,alon2025aldous-caputo,AlonGhosh25}, yet all these prior works were confined to settings rather close to the original conjecture, in the symmetric group or generalized symmetric groups. 
In more distant groups, some very concrete special cases of the Aldous' phenomenon can be traced in older works -- most prominently Kac's master equation, relevant in the group $\mathrm{SO}(n)$, as elaborated in $\S$\ref{subsec:related works}, yet we are not aware of any attempt to find a general statement parallel to Aldous' spectral gap conjecture. 
In this paper, we introduce a conjecture in the case of the unitary group\footnote{In $\S$\ref{sec:open problems} we briefly discuss generalizations of this conjecture to other sequences of groups.} $\U(n)$ which completely parallels Aldous' conjecture in $\sym(n)$ (Conjectures \ref{conj:main conj} and \ref{conj:main KMP style}), establish supporting evidence (e.g., Theorems \ref{thm:(k,-k) dominates non-balanced parts} and \ref{thm:kmp_k is central block of...}), and prove it in full in several non-trivial cases (Theorems \ref{thm:mean-field} and \ref{thm:n-1s}). We also show the somewhat surprising fact that the $U(n)$-spectrum of a graph (or more generally a hypergraph) contains its $\sym(n)$-spectrum (Theorem \ref{thm:evalues of Sn}).

\subsection{The Aldous phenomenon in $\sym(n)$}
\subsubsection*{Aldous' original conjecture in $\sym(n)$}
We begin with describing precisely the original conjecture of Aldous, as well as its extension to hypergraphs by Caputo. Let $Q\colon \sym(n)\to\R_{\ge0}$ be a non-negative symmetric function on $\sym(n)$ (here symmetric means that $Q(\sigma^{-1})=Q(\sigma)$ for all $\sigma\in \sym(n)$). Let $\rho\colon \sym(n)\to \gl_d(\mathbb{C})$ be a $d$-dimensional representation of $\sym(n)$. Then the Laplacian operator associated with $\rho$ and $Q$ is
\[
    \L(Q,\rho)\defi \sum_{\sigma\in \sym(n)}Q(\sigma)\left(I_d-\rho(\sigma)\right)\in \mathrm{M}_d(\mathbb{C}),
\]
where $I_d$ is the $d\times d$ identity matrix. When $Q$ is a probability measure, namely, when \linebreak$\sum_{\sigma\in\sym(n)} Q(\sigma)=1$, the definition is slightly more intuitive, as we simply choose a random element according to $Q$, and $\L(Q,\rho)=\mathbb{E}_{\sigma\sim Q}[I_d-\rho(\sigma)]$.

As $Q$ is symmetric, the resulting matrix $\L(Q,\rho)$ is hermitian and has a real spectrum.\footnote{Every finite-dimensional representation of a finite (or compact) group $G$ is equivalent to a unitary representation $\rho\colon G\to\U(d)$, and $\rho(g)+\rho(g^{-1})$ is hermitian for all $g\in G$.} We denote the smallest eigenvalue of $\L(Q,\rho)$ by 
\[
    \lambda_{\min}(Q,\rho)\defi
        \mathrm{the~smallest~eigenvalue~of}~ \L(Q,\rho).
\]
As the operator norm of $\rho(\sigma)$ is 1 for all $\sigma$, the spectrum must be non-negative, so $\lambda_{\min}(Q,\rho)\ge0$.

Every finite-dimensional complex representation of any finite (or compact) group, decomposes as a direct sum of irreducible representations (or \textbf{irreps} for short). In matrix terms, this means that after an appropriate change of basis, $\L(Q,\rho)$ becomes a block matrix, with each block defined by some irrep. Each eigenvalue of $\L(Q,\rho)$ (here we think of the spectrum as a multiset) is thus associated with some irrep of $\sym(n)$. Any summand in the decomposition of $\rho$ corresponding to the trivial representation of $\sym(n)$, gives rise to an eigenvalue of 0. We define
\[
    \lambda_{\min}^*(Q,\rho)\defi\left(
    \begin{gathered} 
        \mathrm{the~smallest~eigenvalue~of}~ \L(Q,\rho) \\
        \mathrm{not~ associated~with~the~trivial~representation}
    \end{gathered}\right).
\]
Equivalently, this is the smallest eigenvalue whose associated eigenvector is not invariant under $\rho(\sigma)$ for all $\sigma\in\sym(n)$. If $\rho$ has only trivial components, we define $\lambda_{\min}^*(Q,\rho)=\infty$. Of course, $\lambda_{\min}^*(Q,\rho)\ge0$. 

Whenever a representation $\rho$ contains precisely one copy of the trivial representation, the spectral gap of $\L(Q,\rho)$, which is the difference between its smallest two eigenvalues, is equal to $\lambda_{\min}^*(Q,\rho)$.
The theorem of Caputo, Liggett and Richthammer \cite[Thm~1.1]{caputo2010proof}, which is slightly stronger than Aldous' original conjecture, says that whenever $Q$ is supported only on  \textit{transpositions} (permutations of the form $(ij)$), the spectral gap of the \textit{regular} representation $\Reg_{\sym(n)}\colon \sym(n)\to\mathrm{GL}_{n!}(\mathbb{C})$, is identical to that of the standard $n$-dimensional permutation representation \label{std of sym}$\pi_{\std}\colon \sym(n)\to \mathrm{GL}_n(\mathbb{C})$\marginpar{$\pi_{\std}$},\footnote{We try to mark new notation on the right margin to make it easier for the reader to locate it.} namely,\footnote{Both $\Reg_{\sym(n)}$ and $\pi_{\std}$ contain exactly one copy of the trivial representation in their decomposition.}
\[
    \lambda_{\min}^*(Q,\Reg_{\sym(n)})=\lambda_{\min}^*(Q,\pi_{\std}).
\]
(The regular representation is, too, a permutation representation of $\sym(n)$, describing the action of $\sym(n)$ on itself by multiplication from the left.) 

Denote by $\irr(\sym(n))$\marginpar{$\irr$} the set of (equivalence classes) of irreducible representations of $\sym(n)$. Recall that the irreps of $\sym(n)$ are classified by partitions of $n$ or, equivalently, Young diagrams with $n$ blocks. If $\nu=(\nu_1\ldots,\nu_r)\vdash n$ is a partition of $n$, we mark by $\pi_\nu$\marginpar{$\pi_\nu$} the corresponding irrep of $\sym(n)$. In particular, the trivial representation of $\sym(n)$ is $\triv=\pi_{(n)}$.  The regular representation contains every irrep of $\sym(n)$ in its decomposition (with multiplicity equal to the dimension of the irrep), while $\pi_{\std}$ decomposes as $\triv\oplus\pi_{(n-1,1)}$. Hence, we may equivalently state the result as follows:

\begin{theorem} \cite[Thm.~1.1]{caputo2010proof} \label{thm:CLR}
    Let $Q\colon \sym(n)\to \R_{\ge0}$ be supported on transpositions. Then for every $\triv\ne\rho\in\irr(\sym(n))$,
    \[
        \lambda_{\min}(Q,\pi_{(n-1,1)}) \le \lambda_{\min}(Q,\rho).
    \]
\end{theorem}

\subsubsection*{Caputo's hypergraph conjecture in $\sym(n)$}

Caputo suggested the following conjecture which generalizes Theorem \ref{thm:CLR}. Let $\Gamma=([n],w)$ be a weighted hypergraph on $n$ vertices, given by the assignment of a non-negative weight $w_B$ to any hyperedge $B\subseteq [n]$. This hypergraph determines a non-negative measure $Q_\Gamma\colon \sym(n)\to\R_{\ge0}$ by ``distributing" the weight of every subset $B$ uniformly among all the permutations it supports. Namely, $$Q_\Gamma(\sigma)=\sum_{B\colon B\supseteq\supp(\sigma)}\frac{w_B}{|B|!}.$$ 
(The support of a permutation $\sigma\in \sym(n)$ is $\supp(\sigma)=\{i\in[n]\,\mid\,\sigma(i)\ne i\}$.) If the weights $w_B$ sum up to 1 (so we start with a probability distribution on the subsets of $[n]$), this is equivalent to picking a subset $B$ according to the given distribution, and then picking uniformly at random a permutation in $\sym(B)$. Abusing notation, we denote
\begin{equation}\label{eq:notation for lambda_min of hypergraph}
    \L(\Gamma,\rho)\defi\L(Q_\Gamma,\rho),~~~~~\lambda_{\min}(\Gamma,\rho)\defi \lambda_{\min}(Q_\Gamma,\rho)~~~~~\mathrm{and}~~~~~\lambda^*_{\min}(\Gamma,\rho)\defi \lambda^*_{\min}(Q_\Gamma,\rho).
\end{equation}
    
The following conjecture of Caputo appears in \cite[Conj.~3]{piras2010generalizations}, \cite[p.~301]{cesi2016few} and \cite[p.~78]{aldous2020life}:

\begin{conj}[Caputo's hypergraph conjecture] \label{conj:Caputo}
For any weighted hypergraph $\Gamma=\left([n],w\right)$ with non-negative weights and every non-trivial irrep $\rho$ of $\sym(n)$,
    \[
        \lambda_{\min}(\Gamma,\pi_{(n-1,1)}) \le \lambda_{\min}(\Gamma,\rho).
    \]    
\end{conj}

As the value of $Q$ at the identity element of $\sym(n)$ does not change $\L(Q,\rho)$, the case where $w$ is supported on subsets of $[n]$ of size $\le2$ is precisely Theorem \ref{thm:CLR}. Conjecture \ref{conj:Caputo} is still open in general. It is known in the mean-field case (where $w_B$ depends only on the size of $B$) \cite[Thm.~1.8]{bristiel2024entropy}, and in some additional non-trivial cases \cite{alon2025aldous-caputo}.

We remark that Theorem \ref{thm:CLR} and Conjecture \ref{conj:Caputo} can be equivalently stated in terms of certain discrete processes with particles on the hypergraph $\Gamma$. One is the Interchange Process, where $n$ distinct particles are placed one on each vertex of the hypergraph $\Gamma$. Every hyperedge $B\subseteq[n]$ rings with Poisson rate $w_B$, and when it does, the $|B|$ particles it touches are permuted uniformly at random. This is analogous to the regular representation. The other process is the Random Walk, where we only have one particle, which sits on some vertex of $\Gamma$. When a hyperedge $B$ rings (with the same rate as before), if the particle sits at a vertex contained in $B$ it is moved uniformly at random to one of the $|B|$ vertices in $B$, and otherwise it does not budge. This process is equivalent to the standard representation. Conjecture \ref{conj:Caputo} states that the spectral gaps of the Laplacians of both processes are equal.

\subsection{Hypergraph measures on $\U(n)$} \label{subsec: hypergraphs measures on U(n)}
We now introduce an analog of Caputo's hypergraph measures in the unitary group $\U(n)$ of
complex unitary $n\times n$ matrices. For every subset $B\subseteq[n]$, let $\U_B\le \U(n)$\marginpar{$U_B$} be the subgroup of unitary matrices which are identical to the identity matrix outside the $|B|\times|B|$-minor determined by $B$. So $\U_B\cong \U(|B|)$. For example, inside $\U(5)$, 
\[
    \U_{\{2,3,5\}}=\left\{\begin{pmatrix}
            1&  &  &  &  \\
             &* &* &  &* \\
             &* &* &  &* \\
             &  &  &1 &  \\ 
             &*  &* &  &* 
        \end{pmatrix}\right\}\le\U(5).
\]

\begin{defn}
     Let $\Gamma=([n],w)$ be a weighted hypergraph. The measure induced by $\Gamma$ on $\U(n)$ is \marginpar{$\mu_\Gamma$}
    \[
        \mu_{\Gamma}=\sum_{B\subseteq[n]}w_B\mu_B,
    \]
    where $\mu_B$ \marginpar{$\mu_B$}is the Haar probability measure on the subgroup $\U_B\cong \U(|B|)$.
\end{defn}

If $\sum_Bw_B=1$, then $\mu_\Gamma$ is a probability measure, where a random element is chosen by first picking $B\subseteq[n]$ at random according to $w$, and then sampling $A\in \U_B$ by the Haar measure on $\U_B$. Analogously to our notation above, for any $d$-dimensional representation $\rho\colon\U(n)\to \gl_d(\C)$, we let\footnote{We use the same notation $\L(\Gamma,\rho)$ when $\rho$ is a representation of $\sym(n)$ and when it is a representation of $\U(n)$, and likewise with $\lambda_{\min}(\Gamma,\rho)$ or $\lambda_{\min}^*(\Gamma,\rho)$. This should not cause any confusion as $\rho$ is always defined to be a representation of one group or the other.}
\begin{equation} \label{eq:def-laplacian-of-Gamma-and-rho}
    \L(\Gamma,\rho)\defi \int_{A\in\U(n)}\left(I_d-\rho(A)\right)\mu_\Gamma=\sum_{B\subseteq[n]}w_B\left[I_d-\int_{A\in\U_B}\rho(A)\mu_B\right]\in \mathrm{M}_d(\mathbb{C}).
\end{equation}
In this case, too, the Laplacian has a non-negative real spectrum (see Lemma \ref{lem:L of reps - basic properties of spectrum}), and $\rho$ decomposes as a direct sum of irreducible irreps of $\U(n)$. Similarly to \eqref{eq:notation for lambda_min of hypergraph}, denote
\begin{eqnarray}
    \lambda_{\min}(\Gamma,\rho)&\defi&
    \mathrm{the~smallest~eigenvalue~of}~ \L(\Gamma,\rho).    \\
    \lambda_{\min}^*(\Gamma,\rho)&\defi&\left(
    \begin{gathered} 
        \mathrm{the~smallest~eigenvalue~of}~ \L(\Gamma,\rho) \\
        \mathrm{not~ associated~with~the~trivial~representation}
    \end{gathered}\right).    \label{eq:lambda_min of irrep of U}
\end{eqnarray}

By \textbf{the $\U(n)$-spectrum of $\Gamma$} we mean the set of eigenvalues of $\L(\Gamma,\rho)$ across all finite dimensional (irreducible) representations $\rho$ of $\U(n)$. The closure of this set is the spectrum of $\mu_\Gamma$ in the regular representation of $\U(n)$ -- see $\S$\ref{subsec:representations theory of U(n)}.


\subsection{An Aldous phenomenon in $\U(n)$} \label{subsec:Aldous phenomenon in U}

\subsubsection*{The mean-field case}
Our first result concerns weighted hypergraphs in which the weight of a hyperedge depends only on its size. This is sometimes called the ``mean-field'' case (e.g., in \cite{bristiel2024entropy}). The irreps $\irr(\U(n))$ of $\U(n)$ are all finite-dimensional and are parametrized by their ``highest weight vectors": non-increasing length-$n$ sequences of integers, namely,  $\mu=(\mu_1,\mu_2,\ldots,\mu_n)\in \Z^n$ with $\mu_1\ge\mu_2\ge\ldots\ge\mu_n$ (e.g., \cite[\S36]{bump2013lie}, and see $\S$\ref{subsec:representations theory of U(n)}). The trivial irrep corresponds to $\mu=(0,\ldots,0)$. We denote by $\rho_\mu\in\irr(\U(n))$\marginpar{$\rho_\mu$} the irrep corresponding to $\mu\in\mathbb{Z}^n$. 

\begin{theorem}\label{thm:mean-field}
Let $n\ge2$ and consider the hypergraph $\Gamma=([n],w)$ where $w_B=c_{|B|}$ and $c_\ell\ge0$ for $\ell=0,\ldots,n$. 
Then\footnote{We use the convention that $\binom{a}{b}=0$ when $b<0$ or $a<b$.}
\begin{equation} \label{eq:mean-field}
    \min_{\triv\ne\rho\in\irr(\U(n))} \lambda_{\min}(\Gamma,\rho) = \lambda_{\min}(\Gamma,\rho_{(2,0,
    \ldots,0,-2)}) = 
    \sum_{\ell=0}^n c_\ell \frac{n+1}{\ell+1}\binom{n-2}{\ell-2}.
\end{equation}
\end{theorem}
Namely, in the mean-field case, the smallest non-trivial eigenvalue in the $\U(n)$-spectrum of $\Gamma$ is always obtained in the irrep with highest weight vector $(2,0,\ldots,0,-2)$. Interestingly, in the symmetric group, the mean-field special case of Conjecture \ref{conj:Caputo} is known \cite[Thm.~1.8]{bristiel2024entropy}, and the smallest eigenvalue, obtained in the irrep $\pi_{(n-1,1)}$, is the slightly larger
\begin{equation}\label{eq:mean-field-S_n-evalue}
    \sum_{\ell=0}^n c_\ell \frac{n}{\ell}\binom{n-2}{\ell-2}.
\end{equation}
As we explain below, this eigenvalue \eqref{eq:mean-field-S_n-evalue} is also obtained in the $\U(n)$-spectrum of $\Gamma$, as the minimal eigenvalue of  $\L(\Gamma,\rho_{(1,0,\ldots,0,-1)})$.

\subsubsection*{The codimension-1 case}

A similar result holds in the following large family of hypergraphs:

\begin{theorem} \label{thm:n-1s}
    Let $\Gamma=([n],w)$ be a hypergraph with non-negative weights supported on subsets of size $\ge n-1$. Then the smallest non-trivial eigenvalue of the Laplacian is attained in one of the irreps $(1,0,\ldots,0,-1)$ or $(2,0,\ldots,0,-2)$. 
\end{theorem}

That is, for every $\Gamma$ as in the theorem, 
\[
    \min_{\triv\ne\rho\in\irr(\U(n))} \lambda_{\min}(\Gamma,\rho) = \lambda_{\min}\left(\Gamma,\rho_{(1,0,\ldots,0,-1)}\oplus \rho_{(2,0,\ldots,0,-2)}\right) .
\]
We stress that in the setting of Theorem \ref{thm:n-1s}, both irreps mentioned are necessary: each of them captures the spectral gap for certain values of the weights: see Corollary \ref{cor:smallest-evalue for n-1 mean-field} and Example \ref{exm:n-1s with smallest evalue in (1,...,-1)}.

For every $n\ge2$, the dimension of $\rho_{(1,0,\ldots,0,-1)}$ is $n^2-1$, and that of $\rho_{(2,0,\ldots,0,-2)}$ is $\frac{n^2(n-1)(n+3)}{4}$. Along the paper we provide a rather concrete description of these two irreps from Theorems \ref{thm:mean-field} and \ref{thm:n-1s}. First,  let $\rho_\std\colon\U(n)\to\gl(V)$\marginpar{$(\rho_\std,V)$} with $V=\mathbb{C}^n$ be the standard, $n$-dimensional representation of $\U(n)$ given by $A\mapsto A$. For $k,m\in\mathbb{Z}_{\ge0}$, define the $\U(n)$-representation $\skm$\marginpar{$\skm$} by
\begin{equation} \label{eq:def of Sk,m}
    \skm\defi \Sym^k(\rho_\std)\otimes \Sym^m(\rho_\std^*) = \Sym^k(V)\otimes \Sym^m(V^*),
\end{equation}
where $\rho_\std^*$ is the representation dual to $\rho_\std$. A special case of Corollary \ref{cor:decomp of sym-k tensor sym-m-dual} below is that
\[
    S_{2,2}=\rho_{(2,0,\ldots,0,-2)}\oplus\rho_{(1,0,\ldots,0,-1)} \oplus \triv.
\]
Hence, Theorem \ref{thm:n-1s} is equivalent to that for every $n\ge2$ and every hypergraph $\Gamma$ as in the theorem (so weights are supported on subsets of size $\ge n-1$),
\[
    \min_{\triv\ne\rho\in\irr(\U(n))} \lambda_{\min}(\Gamma,\rho) = \lambda^*_{\min}(\Gamma,S_{2,2}).
\]  
Of course, a slightly weaker form of Theorem \ref{thm:mean-field} can use the same formulation.

Moreover, as explained in $\S$\ref{subsec:intro-KMP} below, $\lambda_{\min}(\Gamma,\rho_{(2,0,\ldots,0,-2)}\oplus\rho_{(1,0,\ldots,0,-1)})=\lambda_{\min}^*(\Gamma,S_{2,2})$ is also the smallest non-trivial eigenvalue of a very concrete process on $\binom{n+1}{2}$ states, known as the discrete KMP process on $\Gamma$ with 2 indistinguishable particles. 

Finally, one can also consider the two irreps (or $S_{2,2}$) as a subrepresentation of the easy-to-construct representation $R_{2,2}$: For $k,m\in\mathbb{Z}_{\ge0}$, define the $n^{k+m}$-dimensional representation $\rkl$\marginpar{$\rkl$}
\begin{equation} \label{eq:def of Rkl}
    \rkl\defi (\rho_\std)^{\otimes k} \otimes (\rho_\std^*)^{\otimes m}=V^{\otimes k}\otimes (V^*)^{\otimes m}~~~\mathrm{defined~by}~~~A\mapsto A^{\otimes k}\otimes \overline{A}^{\otimes m}.
\end{equation}
We cite relevant information about the decomposition of $\rkl$ to irreps in $\S$\ref{subsec:representations theory of U(n)}. As $S_{2,2}$ is a sub-representation of $R_{2,2}$ we obtain that a slightly weaker form of Theorems \ref{thm:mean-field} and \ref{thm:n-1s} is that for every $n\ge2$ and every hypergraph $\Gamma$ as in the theorems (so mean-field or supported on subsets of size $\ge n-1$), 
\[
    \min_{\triv\ne\rho\in\irr(\U(n))} \lambda_{\min}(\Gamma,\rho) =\lambda^*_{\min}(\Gamma,R_{2,2}).
\]  

\begin{remark}
    Theorems \ref{thm:mean-field} and \ref{thm:n-1s} can be equivalently stated in terms of the infinite-dimensional regular representation of $\U(n)$, as we explain in $\S$\ref{subsec:representations theory of U(n)}. In both cases, the spectral gap of $S_{2,2}$ is equal to the spectral gap of the Laplacian corresponding to $\Gamma$ in the regular representation of $\U(n)$. 
\end{remark}


\subsubsection*{An Aldous-Caputo-type conjecture for $U(n)$}
The following conjecture arises naturally from the theorems above. It is also supported by additional results we present below, as well as by computer simulations. It says that the statement of Theorem \ref{thm:n-1s} should hold for all weighted hypergraphs (with non-negative weights).

\begin{conj} \label{conj:main conj}
    Let $\Gamma=([n],w)$ be an arbitrary hypergraph with non-negative weights. Then the smallest non-trivial eigenvalue of the $\U(n)$-spectrum of $\Gamma$ is attained in one of the irreps $(1,0,\ldots,0,-1)$ or $(2,0,\ldots,0,-2)$. 
\end{conj}

Equivalently, the spectral gap of the $\U(n)$-spectrum of $\Gamma$ coincides with that of $\L(\Gamma,S_{2,2})$.
Note that there are usually infinitely many distinct eigenvalues in the $\U(n)$-spectrum of $\Gamma$. So it is not at all obvious that the infimum over all these eigenvalues is attained. The conjecture states not only that it is attained, but also always at one of two specific irreps.
Conjecture \ref{conj:main KMP style} below gives an equivalent statement in terms of a discrete KMP process. 

\subsubsection*{Results holding for arbitrary hypergraphs}

The following result, holding for every hypergraph with non-negative weights, guarantees that the infimum of the non-trivial spectrum is obtained only in a certain subset of the irreps, and only in certain invariant subspaces of the latter.

\begin{defn}
    Let $\rho\in\irr(\U(n))$ be associated with the highest weight vector $(\mu_1,\ldots,\mu_n)$. We say that $\rho$ is \textbf{balanced} if and only if $\sum \mu_i=0$. Otherwise, $\rho$ is called \textbf{unbalanced}.
\end{defn}

(The irrep $\rho$ is balanced if and only if the center of $\U(n)$ acts trivially.) Consider the subgroup $T_n\le\U(n)$\marginpar{$T_n$} consisting of all diagonal matrices, also called the \textbf{torus} of $\U(n)$. Namely,  
\[
    T_n\defi \left\{\begin{pmatrix}
            *&  &  &   \\
             &* & &   \\
             & &\ddots   & \\
             &  &  &*
        \end{pmatrix}\right\}\le\U(n).
\]

\begin{defn} \label{def:torus-invariant subspace}
    Let $\rho\colon\U(n)\to\gl(V)$ be a representation of $\U(n)$ over a finite dimensional complex vector space $V$.\footnote{We restrict to finite dimensional representations for simplicity. The definition makes sense and its properties remain true for any unitary representation of $\U(n)$.} The \textbf{torus-invariant subspace} of $\rho$ is\marginpar{$\scriptstyle{\torinv(\rho)}$}
    \[
        \torinv(\rho)\defi V^{T_n}=\left\{ v\in V\,\mid\,\rho(A)v=v \mathrm{~for~all~} A\in T_n\right\}.
    \]
\end{defn}

As stated in Proposition \ref{prop:torus-inv subspace is invariant for every Gamma} below, $\torinv(\rho)$ is invariant under the Laplacian $\L(\Gamma,\rho)$ corresponding to any weighted hypergraph $\Gamma$. In fact, it also has a complement subspace invariant under $\L(\Gamma,\rho)$. We stress that this is relevant even when $\rho$ is irreducible: while $\rho$ does not have any proper non-zero subspace which is invariant under all $A\in\U(N)$, it may have, and usually does have, non-trivial subspaces invariant under $\L(\Gamma,\rho)$ for all $\Gamma$. See $\S$\ref{sec:torus-invariant subspace} for a discussion of additional subspaces of $\rho$ which are invariant under $\L(\Gamma,\rho)$ for all $\Gamma$. 

Denote by  $\L_\torinv(\Gamma,\rho)$\marginpar{$\scriptscriptstyle{\L_\torinv(\Gamma,\rho)}$}\label{page:L_torinv} the restriction of $\L(\Gamma,\rho)$ to $\torinv(\rho)$. In particular, the spectrum of $\L_\torinv(\Gamma,\rho)$ is a sub-multiset of the spectrum of $\L(\Gamma,\rho)$. By Corollary \ref{cor:torus-inv subspace empty for unbalanced irreps} below, the torus-invariant subspace of an \textit{unbalanced} irrep is trivial (namely, the zero subspace). If $\rho$ is a balanced irrep, its torus-invariant subspace is of dimension roughly $\sqrt{\dim(\rho)}$ -- see Corollary \ref{cor:torus-inv of balanced is of dim sqrt}. 




\begin{theorem} \label{thm:(k,-k) dominates non-balanced parts}
Let $k\in\mathbb{Z}_{\ge1}$ and denote $\rho_k=\rho_{(k,0,\ldots,0,-k)}$. Let $\rho\in\irr(\U(n))$ and let $\alpha$ be any eigenvalue of $\L(\Gamma,\rho)$ which is \textbf{not} associated with $\L_\torinv(\Gamma,\rho)$.
Then $\lambda_{\min}(\Gamma,\rho_k)\le\alpha$.

In particular, if $\rho$ is unbalanced, then for all $k$, $\rho_k$ ``spectrally dominates $\rho$":
\[
    \lambda_{\min}(\Gamma,\rho_k)\le\lambda_{\min}(\Gamma,\rho).
\]    
\end{theorem}

\begin{cor} \label{cor:spectral gap obtained in non-balanced parts}
Let $\Gamma$ be a hypergraph with non-negative weights. Then
\[
    \inf_{\triv\ne\rho\in\irr(\U(n))} \lambda_{\min}(\Gamma,\rho) = \inf_{\substack{\triv\ne\rho\in\irr(\U(n)):\\\rho~\mathrm{is~balanced}}} \lambda_{\min}(\L_\torinv(\Gamma,\rho)).
\]
\end{cor}
Namely, the infimum of the non-trivial $\U(n)$-spectrum of a weighted hypergraph $\Gamma$ (or the smallest eigenvalue if it exists), is always obtained by considering only the linear action defined by $\Gamma$ on the torus-invariant subspaces of balanced representations.

\subsection{A discrete KMP (or multiset) process}\label{subsec:intro-KMP}

In \cite{kipnis1982heat}, Kipnis, Marchioro and Presutti describe a continuous-time process on a finite graph, a process which is usually referred to as the KMP process after these three authors. In this process, each vertex of the graph contains a certain finite mass. Each edge is equipped with a Poisson clock with some rate. Whenever an edge ``rings'', the combined mass at the two incident vertices $u$ and $v$ is redistributed among $u$ and $v$, where $u$ gets a $p$-fraction of the total mass and $v$ a $(1-p)$-fraction, $p\in[0,1]$ taken uniformly at random.\footnote{We remark that in the original setting in \cite{kipnis1982heat}, the graph is a finite path, all the edges have the same rate, and there are additionally two ``absorbing vertices'' at the two endpoints of the path.} 

In \cite[$\S$2]{kipnis1982heat}, the authors describe a discrete version of their process, which is nowadays known as the ``discrete KMP process'' or the ``uniform reshuffling process''. We extend the definition to weighted hypergraphs, and show that the spectrum of the discrete KMP process on the weighted hypergraph $\Gamma$ is contained in the $\U(n)$-spectrum of $\Gamma$. Moreover, the KMP spectrum is precisely the part that conjecturally, and at times provably, dominates the spectral gap.
\medskip

Given a set $S$ and an integer $k\geq 0$, consider the set $\ms{S}{k}$\marginpar{$\ms{S}{k}$} of multisets of size $k$ with elements from $S$. Equivalently, these are functions $f\colon S\to\mathbb{Z}_{\ge0}$ with $\sum_{s\in S} f(s)=k$. When $S=[n]$ is the set of vertices of a hypergraph $\Gamma$, we think of the multisets $\ms{[n]}{k}$ as the possible configurations of $k$ indistinguishable particles located at the vertices (with no limitation on the number of particles at the same vertex). Note that \marginpar{$\ms{n}{k}$}
\[
    \left|\ms{[n]}{k}\right|=\ms{n}{k}\defi\binom{n+k-1}{k}.
\]
Denote by $\MS(n,k)$\marginpar{$\scriptstyle{\MS(n,k)}$} the formal $\mathbb{C}$-span of these configurations. 

Every subset $B\subseteq[n]$ defines a linear action $\N_B$\marginpar{$\N_B$} on $\MS(n,k)$ as follows. Its action on some configuration $f\in\ms{[n]}{k}$ is given by collecting all the particles at the vertices belonging to $B$, and redistributing them among the same vertices of $B$ in a uniformly random way among all possibilities. Equivalently, $\N_B$ maps $f\colon[n]\to\mathbb{Z}_{\ge0}$ uniformly at random to one of the functions $g\colon[n]\to\mathbb{Z}_{\ge0}$ satisfying $\sum_i g(i)=\sum_i f(i)$ and $g(j)=f(j)$ for all $j\notin B$. If there are $\ell$ particles in total in the vertices contained in $B$, then there are $\ms{|B|}{\ell}=\binom{|B|+\ell-1}{\ell}$ such possibilities. 

 
\begin{defn}[Discrete KMP on hypergraphs]
    Let $\Gamma=([n],w)$ be a hypergraph with non-negative weights and fix $k\in\mathbb{Z}_{\ge1}$. The discrete KMP process with $k$ particles on $\Gamma$, denoted $\kmp_k(\Gamma)$\marginpar{$\scriptstyle{\kmp_k(\Gamma)}$}, is a process with states $\ms{[n]}{k}$, in which every hyperedge $B\subseteq[n]$ ``rings'' and acts by $\N_B$ with rate $w_B$.
    The corresponding Laplacian is given by
    \begin{equation} \label{eq:kmp-k generator}
        \L(\Gamma,\kmp_k)\defi \sum_{B\subseteq[n]}w_B(I-\N_B)\in\End(\MS(n,k)).
    \end{equation}
\end{defn}

The Laplacian $\L(\Gamma,\kmp_k)$ has non-negative real spectrum (see Corollary \ref{cor:spectrum of Laplacian on KMP in [0,sum of weights]}). There is always a single trivial zero eigenvalue corresponding to the uniform distribution among all particle configurations. In analogy with \eqref{eq:lambda_min of irrep of U}, we denote  
\begin{equation} \label{eq:lambda_min of KMP}
    \lambda_{\min}^*(\Gamma,\kmp_k) \defi
    \mathrm{the~smallest~non\text{-}trivial~eigenvalue~of}~ \L(\Gamma,\kmp_k).    
\end{equation}

\medskip
We can identify the spectrum of $\L(\Gamma,\kmp_k)$ inside the $\U(n)$-spectrum of $\Gamma$. By Corollary \ref{cor:decomp of sym-k tensor sym-m-dual} below,
\[
    S_{k,k}\cong\bigoplus_{j=0}^k\rho_{(j,0,\ldots,0,-j)}.
\]
Recall the torus-invariant subspace of a representation from Definition \ref{def:torus-invariant subspace} and the notation $\L_\torinv(\Gamma,\rho)$ for every (finite-dimensional) representation $\rho$ of $\U(n)$. 

\begin{theorem} \label{thm:kmp_k is central block of...}
 For any $k\in\mathbb{Z}_{\ge1}$ and weighted hypergraph $\Gamma=([n],w)$, 
\begin{equation} \label{eq:KMP is central block of..}
    \L(\Gamma,\kmp_k)\cong \L_\torinv\left(\Gamma,S_{k,k}\right) \cong \L_\torinv\left(\Gamma,\bigoplus_{j=0}^k\rho_{(j,0,\ldots,0,-j)}\right).
\end{equation}    
\end{theorem}

Namely, all Laplacians in \eqref{eq:KMP is central block of..} are the same operators, up to conjugation. In particular, $\kmp_2(\Gamma)$, which is a discrete process of dimension $\binom{n+1}{2}$, satisfies
\[
    \L(\Gamma,\kmp_2)\cong \L_\torinv\left(\Gamma,S_{2,2}\right) \cong \L_\torinv(\Gamma,\rho_{(2,0,\ldots,0,-2)}\oplus\rho_{(1,0,\ldots,0,-1)}\oplus \triv).
\]    
By Theorem \ref{thm:(k,-k) dominates non-balanced parts}, the smallest eigenvalue of $\L(\Gamma,S_{k,k})$ is associated with its torus-invariant subspace. Hence, by Theorem \ref{thm:kmp_k is central block of...},
\begin{equation} \label{eq:lam-min-star of KPMk same as Skk}
    \lambda_{\min}^*(\Gamma,\kmp_k) = \lambda_{\min}^*(\Gamma,S_{k,k}).
\end{equation}    
This shows that Conjecture \ref{conj:main conj} is equivalent to the following one:

\begin{conj} \label{conj:main KMP style}
    Let $\Gamma=([n],w)$ be an arbitrary hypergraph with non-negative weights. Then $\lambda_{\min}^*(\Gamma,\kmp_2)$ is the smallest non-trivial eigenvalue in the $\U(n)$-spectrum of $\Gamma$. Namely,
    \[
        \lambda_{\min}^*(\Gamma,\kmp_2) = \inf_{\triv\ne\rho\in\irr(\U(n))}\lambda_{\min}(\Gamma,\rho) = \min_{\triv\ne\rho\in\irr(\U(n))}\lambda_{\min}(\Gamma,\rho).
    \]
\end{conj}

As explained in $\S$\ref{subsec:representations theory of U(n)}, this is equivalent to that the spectral gap of the Laplacian corresponding to $\Gamma$ on the regular representation of $\U(n)$ coincides with that of $\kmp_2(\Gamma)$. 

By \eqref{eq:lam-min-star of KPMk same as Skk}, the following weaker conjecture, that can be stated completely in terms of discrete KMP processes without mentioning of any representations of $\U(n)$, follows from Conjecture \ref{conj:main KMP style}.

\begin{conj} \label{conj:pure KMP}
    For any weighted hypergraph $\Gamma$ with non-negative weights and $k\in\mathbb{Z}_{\ge2}$ we have
    \[
        \lambda_{\min}^*(\Gamma,\kmp_k)=\lambda_{\min}^*(\Gamma,\kmp_2).
    \]
\end{conj}

In particular, this conjecture is true when $\Gamma$ is as in Theorems \ref{thm:mean-field} or \ref{thm:n-1s}. Note that as an immediate corollary from Theorem \ref{thm:kmp_k is central block of...}, the spectrum of $\L(\Gamma,\kmp_\ell)$ is contained in the spectrum of $\L(\Gamma,\kmp_k)$ whenever $\ell\le k$ (there is also a  direct simple proof of this fact -- see Proposition \ref{prop:embedding MS(n,k) in MS(n,k+1) properties}). So for every $k\ge2$, 
\[
    \lambda_{\min}^*(\Gamma,\kmp_k)\le\lambda_{\min}^*(\Gamma,\kmp_2)\le\lambda_{\min}^*(\Gamma,\kmp_1).
\]
We also suggest a more concrete conjecture about the smallest eigenvalue of $\L(\Gamma,\kmp_k)$ not appearing in $\L(\Gamma,\kmp_{k-1})$: see Conjecture \ref{conj:order of omega_k} below.

\begin{remark}
    Consider a similar process, where there are $k\le n$ indistinguishable particles at the vertices, but \textit{at most one} at every vertex, and where every $B\subseteq[n]$ acts by randomly moving the particles at its vertices among all $\binom{|B|}{K}$ possibilities. Since this process corresponds to a representation of $\sym(n)$ that contains $\rho_{(n-1,1)}$, Caputo's Conjecture \ref{conj:Caputo} yields that the spectral gap with a single particle ($k=1$) is equal to the spectral gap with arbitrary $k$. In particular, this conjecture is known when $\Gamma$ is a graph. In contrast, to the best of our knowledge, Conjecture \ref{conj:pure KMP} is not known even when $\Gamma$ is a graph.
\end{remark}

\subsection{The $\U(n)$-spectrum of a hypergraph contains its $\sym(n)$-spectrum}

Let $\Gamma$ be an arbitrary weighted hypergraph. It is far from a priori clear that the spectrum corresponding to the Laplacian of $\Gamma$ in $\U(n)$ should have anything to do with the corresponding spectrum in $\sym(n)$. However, consider $\kmp_1(\Gamma)$, the discrete KMP process on $\Gamma$ with a single particle, and notice that this is precisely the random walk of a single particle on $\Gamma$. The Laplacian spectrum of this random walk is precisely $\L(\Gamma,\pi_\std)$, where $\pi_\std\cong\pi_{(n-1,1)}\oplus\triv$ is the standard $n$-dimensional permutation representation of $\sym(n)$ mentioned on page \pageref{std of sym}. In particular, Theorem \ref{thm:kmp_k is central block of...} yields that the spectrum of $\L(\Gamma,\pi_\std)$ (this is in $\sym(n)$) is contained in the spectrum of $\L(\Gamma,\rho_{(1,0,\ldots,0,-1)}\oplus\triv)$ in $\U(n)$, and more precisely, $\L(\Gamma,\pi_{(n-1,1)})$ is equivalent to $\L_\torinv(\Gamma,\rho_{(1,0,\ldots,0,-1)})$. 

We show that this fact is not an outlier but the law: 

\begin{theorem} \label{thm:evalues of Sn}
    Let $\Gamma=([n],w)$ be an arbitrary hypergraph with non-negative weights. Then the $\sym(n)$-spectrum of $\Gamma$ is contained in the $\U(n)$-spectrum of $\Gamma$.
   
    More precisely, if $k\in[n]$ and $\nu=(\nu_1,\ldots,\nu_r)\vdash n$ is a partition of $n$ with at most $k$ boxes outside the first row (namely,  $k\le n-\nu_1$), and $\pi_\nu$ the corresponding irreducible $\sym(n)$-representation, then the spectrum of $\L(\Gamma,\pi_\nu)$ is contained in that of $\L(\Gamma,R_{k,k})$.
\end{theorem}

In particular, the entire $\sym(n)$-spectrum of $\Gamma$ is contained in $\L(\Gamma,R_{n,n})$ (in fact, even in $\L(\Gamma,R_{n-1,n-1})$).
We give a more precise version of these results in $\S$\ref{sec:sym(n) spectrum contained}. 

This connection between $U(n)$ and $\sym(n)$ shows that Caputo's Conjecture \ref{conj:Caputo} in $\sym(n)$ and our Conjecture \ref{conj:main conj} in $\U(n)$ are related not only is spirit but in very concrete way: indeed, together with Theorem \ref{thm:evalues of Sn}, Conjecture \ref{conj:main conj} yields that the minimal non-trivial eigenvalue in the $\sym(n)$-spectrum of $\Gamma$ is at least the smallest one in the spectrum corresponding to $\rho_{(1,0,\ldots,0,-1)}$ (which is precisely the one in Caputo's conjecture), or the smallest one corresponding to $\rho_{(2,0,\ldots,0,-2)}$. In fact, a slightly stronger version of Conjecture \ref{conj:main conj}, a version that we suggest in $\S$\ref{sec:open problems}, yields Caputo's Conjecture \ref{conj:Caputo} for $\sym(n)$ as a special case.

\subsection{Outline and notation}
\subsubsection*{Outline of the paper}
After reviewing some related works in $\S$\ref{subsec:related works}, we collect some preliminary results and facts in $\S$\ref{sec:prelimininaries}: basic facts we will use about the representation theory of $\U(n)$, a quick introduction to Weingarten calculus -- a tool for integrating over Haar-random unitaries, and some basic properties of the spectrum of the Laplacians we study. Section \ref{sec:torus-invariant subspace} discusses the torus-invariant subspaces of a $\U(n)$-representation $\rho$, proves it is invariant under $\L(\Gamma,\rho)$, identifies the torus-invariant subspace of $S_{k,k}$ as KMP with $k$ particles (thus proving Theorem \ref{thm:kmp_k is central block of...}), and shows that torus-invariant subspaces control the spectral gap of $\Gamma$, establishing Theorem \ref{thm:(k,-k) dominates non-balanced parts}. Subsection $\S$\ref{subsec:embed KMP_k in KMP_k+1} also discusses directly the embedding of $\kmp_k(\Gamma)$ in $\kmp_{k+1}(\Gamma)$.

In $\S$\ref{sec:n-1} we analyze hypergraphs supported on subsets of size $\ge n-1$. First, after pointing to useful special properties of this case, we establish in Corollary \ref{cor:smallest-evalue for n-1 mean-field} the precise spectral gap when all weights of subsets of size $n-1$ are equal. This is used, later, in the proof of the general mean-field case in $\S$\ref{sec:mean-field}. Then, in $\S$\ref{subsec:proof of n-1s general case} we carry a detailed analysis of certain real-rooted polynomials leading to a full proof of Theorem \ref{thm:n-1s}. 

Section \ref{sec:mean-field} first proves the mean-field case (Theorem \ref{thm:mean-field}) in full, and then shows that any connected hypergraph admits a positive spectral gap in its $\U(n)$-spectrum. In $\S$\ref{sec:sym(n) spectrum contained} we prove Theorem \ref{thm:evalues of Sn} that the $\sym(n)$-spectrum of $\Gamma$ is contained in its $\U(n)$-spectrum. Section \ref{sec:Possible and impossible extensions} mentions three natural extensions of our conjectures: two of which are probably impossible as they are ruled out by simulations, and one is more plausible. We end in $\S$\ref{sec:open problems} with some intriguing open problems that naturally arise from this work.

\subsubsection*{Notation}
To facilitate the detection of new notation, we usually mark new notation in the right margin. We also list here notation that is used across different sections. Throughout the paper, an ``irrep'' means an irreducible (complex) representation, and ``the $\U(n)$-spectrum of a hypergraph $\Gamma$'' is the union over all finite dimensional representations $\rho$ of $\U(n)$ of the spectrum of $\L(\Gamma,\rho)$. The closure of this set is the spectrum of the Laplacian $\L(\Gamma)=\L(\Gamma,\Regun)$ of the regular representation -- see $\S$\ref{subsec:representations theory of U(n)}.

\paragraph{Representations:} The set of (equivalence classes of) irreps of a group $G$ is denoted $\irr(G)$. We try to denote representations of the symmetric group $\sym(n)$ with the letter $\pi$, and those of $\U(n)$ with the letter $\rho$. In particular, for a partition $\nu\vdash n$, the corresponding irrep of $\sym(n)$ is $\pi_\nu$, and for a non-increasing $\mu\in\Z^n$, the corresponding irrep of $\U(n)$ is $\rho_\mu$. The standard $n$-dimensional representations are marked $\pi_\std$ (for $\sym(n)$) and $\rho_\std$ (for $\U(n)$). The $\U(n)$-representations $S_{k,m}$ and $R_{k,m}$ are defined in $\S$\ref{subsec:Aldous phenomenon in U}. Any two partitions $\nu^+,\nu^-$ give rise to a $\U(n)$-irrep $\rho_{\nu^+,\nu^-}^{(n)}$ for every $n\ge|\nu^+|+|\nu^-|$, introduced in $\S$\ref{subsec:representations theory of U(n)}.
The notation $\wg$ for the Weingarten function is introduced in \ref{subsec:Weingarten}.

\paragraph{Hypergraph measures and Laplacians:} For $B\subseteq[n]$ we denote by $\U_B$ the corresponding subgroup of $\U(n)$ defined in \ref{subsec: hypergraphs measures on U(n)}, by $\mu_B$ the Haar measure on $\U_B$, and by $\mu_\Gamma$ the hypergraph measure corresponding to the hypergraph $\Gamma$. The Laplacian corresponding to $\Gamma$ and a representation $\rho$ is denoted $\L(\Gamma,\rho)$. The smallest eigenvalue of its spectrum is $\lambda_{\min}(\Gamma,\rho)$, and the smallest eigenvalue not associated with the trivial component of $\rho$ is $\lambda^*_{\min}(\Gamma,\rho)$. We denote by $P_B$ the orthogonal projection onto the $\U_B$-invariant subspace in some $\U(n)$-representation $\rho$, and by $\L_B=I-P_B$ the associated Laplacian. 

\paragraph{Torus-invariant subspace:} We mark the torus subgroup of $\U(n)$ by $T_n$, the torus-invariant subspace of a $\U(n)$-representation $\rho$ by $\torinv(\rho)$, and the restriction of $\L(\Gamma,\rho)$ to this invariant subspace by $\L_\torinv(\Gamma,\rho)$. 

\paragraph{KMP:} The basic notation for the KMP processes is introduced in $\S$\ref{subsec:intro-KMP}. This includes the family $\ms{S}{k}$ of size-$k$ multisets with elements from the set $S$, the size $\ms{|S|}{k}$ of this family, the complex vector space $\MS(n,k)$ formally spanned by $\ms{[n]}{k}$, and the linear action $\N_B$ of $B\subseteq[n]$ on this space. This leads to $\kmp_k(\Gamma)$ -- the discrete KMP process with $k$ indistinguishable particles defined on $\Gamma$. The linear embedding of $\MS(n,k)$ into $\MS(n,k+1)$ is denoted $\Psi_k$ and introduced in $\S$\ref{subsec:embed KMP_k in KMP_k+1}. We also introduce there the subspace $\pure(n,k)\le\MS(n,k)$ and the notation $\delta_\ii$ and $\#_x(\ii)$ for $\ii\in\ms{[n]}{k}$.

\paragraph{Additional notation:} We mark by $I_d,J_d\in\gl_d(\C)$ the  identity matrix and the all-one matrix, respectively. If $T_1,T_2$ are two operators on the regular representation of $\U(n)$, we write $T_1\le T_2$ if $T_2-T_1$ is positive semi-definite -- see $\S$\ref{subsec:spectrum is real non-negative and bounded}.

\subsection{Related works} \label{subsec:related works}

\paragraph{Aldous-type phenomena in groups.} 
Several attempts to find extensions within the symmetric group of the original Aldous spectral gap conjecture are discussed in \cite{piras2010generalizations,aldousorder,cesi2016few,parzanchevski2020aldous,alon2025aldous-caputo}. Generalizations to wreath products of the form $G\wr\sym(n)$ include \cite{cesi2020spectral,Ghosh23,AlonGhosh25,levhari2026wreath}. 

In a rather different circle of ideas, Kac \cite{kac1956foundations} introduced a model of random energy-preserving ``molecular collisions'', which translates to a random walk on the group $\mathrm{SO}(n)$, where 
first a pair of coordinates $i,j\in[n]$, $i\ne j$ is chosen uniformly at random, and then a random element in $\mathrm{SO}_{\{i,j\}}\le\mathrm{SO}(n)$ is chosen according to some fixed, symmetric distribution on $\mathrm{SO}(2)\cong S^1$. When this distribution is Haar, we get an instance of our hypergraph measures model (corresponding to the complete graph with constant weight). This model was studied extensively. In particular, the papers \cite{maslen2003eigenvalues,carlen2003determination} prove that under mild assumptions on the fixed probability distribution on $\mathrm{SO}(2)$, the spectral gap is obtained in a certain irreducible representation of $\mathrm{SO}(n)$ of dimension of order $n^4$ (e.g., \cite[Thm.~2.1]{maslen2003eigenvalues}): a representation involving polynomials in $n$ variables of total degree 4. This is very much in the same spirit of our Theorem \ref{thm:mean-field}. 
It is highly plausible that a suitable version of Conjecture \ref{conj:main conj} holds for general hypergraph measures in $\mathrm{SO}(n)$ (or in $\mathrm{O}(n)$) -- see also $\S$\ref{sec:open problems}.

\paragraph{Random walks on $\U(n)$.} 
Quite a few papers study the spectral gaps of various random walks on $\U(n)$, $\mathrm{SU}(n)$ or in more general compact Lie groups. We mention some highlights (we certainly do not aim to present an exhaustive list). The works \cite{bourgain2008spectral,bourgain2012spectral,benoist2016spectral} establish a spectral gap for random walks supported on a finite set of elements with algebraic entries generating dense subgroups. (It is a major open problem whether the same result holds for Haar-random finite tuples of elements.) In \cite{bourgain2017random}, Bourgain considers a measure on $\mathrm{SU}(n)$ (or similarly on $\mathrm{SO}(n)$) which resembles the Kac's model from above only the underlying graph is not the complete graph but rather a cycle on $n$ vertices, with a fixed measure on $\mathrm{SU}(2)$ (or on $\mathrm{SO}(2)$, respectively), and analyses a type of mixing time for this random walk.

\paragraph{KMP and related processes.}
Several works analyze processes that are related or similar to the discrete KMP process we describe in $\S$\ref{subsec:intro-KMP} above. We mentioned above the original paper \cite{kipnis1982heat} where both the continuous and discrete versions of KMP were introduced. The recent paper \cite{kim2025spectral} analyses the spectral gap of the continuous KMP process on graphs. In \cite{quattropani2023mixing}, the authors study a process where $k$ \textit{distinguishable} particles lie on the vertices of a graph, and similarly to the discrete KMP process, when an edge rings the adjacent particles are redistributed uniformly among all possible options (this is at least one prominent special case of what they call the ``Binomial splitting process''). The spectral gap of this processes is obtained already with one particle ($k=1$) \cite[Thm.~2.1]{quattropani2023mixing}, in contrast to our results in the current paper. Two other related processes, although a bit more remote, are the ``symmetric inclusion process'' discussed in \cite{kim2024spectral}, and the ``generalized exclusion process'' discussed in \cite{Kanegae04032026}.
    
\paragraph{Quantum circuits.} Finally, we also mention the recent line of works about spectral gaps of certain random quantum circuits, e.g., \cite{haah2025efficient,chen2025incompressibility}. In these models, one considers the group $\mathrm{SU}(2^n)$ and certain diagonal embedding of Haar-random unitary matrices from $\mathrm{SU}(2^k)$. These papers give bounds for the spectral gaps of these measures on $\mathrm{SU}(2^n)$. In particular, in page 12 of the arXiv version of \cite{chen2025incompressibility}, it is pointed out that simulation carried out by Nick Hunter-Jones suggest a phenomenon similar to our Conjecture \ref{conj:main conj}: that the spectral gap is obtained in $R_{2,2}$. See also $\S$\ref{sec:open problems}.

\subsection*{Acknowledgments}
We are in debt to Gady Kozma for sharing our obsession with the Aldous spectral gap conjecture and for taking part in a crucial discussion in the summer of 2023 which was the starting point of this paper. Had he not gone on sabbatical a week later, he would have certainly been an integral part of this work. We also thank Elon Lindenstrauss, Omri Solan, Thomas Spencer and Avi Wigderson for beneficial discussions. We thank the Institute for Advanced Study in Princeton, NJ for hosting the second author long-term, and allowing the first author to come for a long visit: much of the progress towards this paper was obtained during this visit. D.P.~was supported by the European Research Council (ERC) under the European Union’s Horizon 2020 research and innovation programme (grant agreement No 850956), by the Israel Science Foundation, ISF grants 1140/23, the National Science Foundation under Grant No.~DMS-1926686, as well as by the Kovner Member Fund at the IAS, Princeton.


\section{Preliminaries} \label{sec:prelimininaries}

\subsection{The representation theory of $U(n)$} \label{subsec:representations theory of U(n)}

We mention some facts about the representation theory of $\U(n)$ which will be useful below. 

\subsubsection*{Representations of compact groups}
The most basic principles of the representation theory of $\U(n)$ come from its being compact. Let $G$ be an arbitrary compact group. A (complex) representation of $G$ is a continuous homomorphism $G\to\mathrm{GL}(V)$ where $V$ is a complex Banach space. It is irreducible if it has no proper nonzero invariant subspaces. All irreps (complex irreducible representations) of $G$ are finite-dimensional, and every finite $m$-dimensional representation is equivalent to a unitary representation (a continuous homomorphism $G\to\U(m)$).\footnote{These and other general facts about the representation theory of compact groups can be found in, e.g., \cite[Part I]{bump2013lie} or \cite[$\S$5]{folland2016course}.} Each finite-dimensional representation has a unique decomposition to irreps. (More precisely, it has a unique decomposition into isotypic spaces, and the multiplicity of every irrep is unique.) This is true, more generally, to every possibly-infinite dimensional \textit{unitary} representation of $G$ (a not-necessarily-finite-dimensional unitary representation is a homomorphism $\rho\colon G\to\mathrm{GL}(H)$ where $H$ is a Hilbert space and $\rho$ preserves the inner product in $H$).

The (left) regular representation $\Reg_G$ of $G$ is the linear action by $G$ on $L^2(G)$ by $(g.f)(x)=f(g^{-1}x)$. The Peter-Weyl theorem states that 
\[
    \reg_G\cong \widehat{\bigoplus}_{\rho\in\irr(G)}\rho^{\oplus \dim(\rho)},
\]
where $\widehat{\bigoplus}$ denotes the closure of the direct sum.

\subsubsection*{The regular representation of $\U(n)$ and the existence of a spectral gap}

Recall that in $\S$\ref{subsec: hypergraphs measures on U(n)} we defined the measure $\mu_B$ on $\U(n)$ for every $B\subseteq[n]$, the measure $\mu_\Gamma=\sum w_B\mu_B$ for any weighted hypergraph $\Gamma=([n],w)$, and the Laplacian operator $\L(\Gamma,\rho)$ for any finite dimensional representation $\rho$ of $\U(n)$. The definition of the Laplacian naturally extends to infinite dimensional representations. In particular, denote by $\L(\Gamma)=\L(\Gamma,\Regun)$\marginpar{$\L(\Gamma)$} the Laplacian operator on $\Regun$\marginpar{$\Regun$}, the regular representation of $\U(n)$ associated with $\Gamma$. It is given by
\begin{equation} \label{eq:Laplacian on Reg of U(n)}
    \left(\L(\Gamma).f\right)(x) = \int_{A\in \U(n)} \left(f(x)-f(A^{-1}x)\right)\mu_\Gamma = \sum_{B\subseteq[n]} w_B\int_{A\in U_B} \left(f(x)-f(A^{-1}x)\right)\mu_B.
\end{equation}
for any $f\in L^2(\U(n))$. By the Peter-Weyl theorem, the spectrum of $\L(\Gamma)$ is the closure of the $\U(n)$-spectrum of $\Gamma$, namely,
\[ 
    \overline{\bigcup_{\rho\in\irr(\U(n))}\mathrm{Spec}\left(\L(\Gamma,\rho)\right)}.
\] 
There is one trivial, zero eigenvalue in the spectrum of $\L(\Gamma,\Regun)$, corresponding to the constant function in $L^2(\U(n))$ (equivalently, to the trivial representation). 
The non-trivial spectrum of $\L(\Gamma)$ is the closure of the same union without the trivial irrep, and its infimum is the spectral gap of this Laplacian operator. Hence, Theorems \ref{thm:mean-field} and \ref{thm:n-1s}, and Conjectures \ref{conj:main conj} and \ref{conj:main KMP style}, can be equivalently stated as saying that for the appropriate weighted hypergraphs, the spectral gap of $\L(\Gamma)$ (which could be zero) is equal to $\lambda_{\min}\left(\Gamma, \rho_{(1,0,\ldots,0,-1)}\oplus \rho_{(2,0,\ldots,0,-2)}\right)$, or, equivalently, to the spectral gap $\lambda_{\min}^*(\Gamma,\kmp_2)$ of $\L(\Gamma,\kmp_2)$.

\medskip

One may wonder if our conjectures pass the following, much easier, ``test": is it true that both operators -- $\L(\Gamma,\kmp_2)$ and $\L(\Gamma)$ -- admit positive spectral gaps\footnote{For most authors, the word `positive' here is redundant, as by saying that some operator has a spectral gap, they mean explicitly it has a positive one.} for the same hypergraphs $\Gamma$?

Indeed, in both cases we know precisely when the spectral gap is positive: when the hypergraph $\Gamma=([n],w)$ is \textit{connected}\label{page:connected definition}. We say that the hypergraph $\Gamma=([n],w)$ is \textit{connected} if the underlying hypergraph given by the support of $w$ is connected. Namely, if there is no non-trivial partition of the vertices $[n]=E_1\sqcup E_2$ with $E_1,E_2\ne\emptyset$, such that for all $B\subseteq[n]$, if $w_B>0$ then $B\subseteq E_1$ or $B\subseteq E_2$.

Assume first that $\Gamma$ is not connected. It is easy to see it has zero spectral gap in both operators: In $L^2(\U(n))$, the indicator function of $\U_{E_1}$ is another eigenfunction with eigenvalue zero, which is not spanned by the constant function. Likewise, the uniform distribution on $\ms{E_1}{2}$ has eigenvalue zero in $\kmp_2(\Gamma)$.

Conversely, assume that $\Gamma$ is connected. A vector $v\in\MS(n,2)$ is an eigenvector of $\L(\Gamma,\kmp_2)$ with eigenvalue zero if and only if it is invariant under $N_B$ for all $B$ in the support of $\Gamma$. It is a simple exercise to see that this condition yields that $v$ must be a constant function on $\ms{[n]}{2}$.

As for the regular representation, the existence of a spectral gap when $\Gamma$ is connected basically follows from the criterion \cite[Thm.~1.1]{benoist2016spectral}.\footnote{We thank Elon Lindenstrauss for pointing this out to us.} We give an alternative proof of this fact in $\S$\ref{subsec:spectral gap for connected hypergraphs}.

\subsubsection*{The branching rule in $\U(n)$}
Recall that the irreps of $\U(n)$ are classified by their highest weight vectors which are given by non-increasing integer sequences of length $n$. The entries of the vectors are called weights. (These vectors are based on eigenvectors and eigenvalues of the restriction of the irrep to the torus $T_n$.)

Let $\mu=(\mu_1,\ldots,\mu_n)\in\mathbb{Z}^n$ and $\pi=(\pi_1,\ldots,\pi_{n-1})\in\mathbb{Z}^{n-1}$ be non-increasing vectors. We say that $\mu$ and $\pi$ \textbf{interlace} if
\[
    \mu_1 \ge \pi_1 \ge \mu_2 \ge \pi_2 \ge \ldots \ge \pi_{n-1} \ge \mu_n.
\]
The branching rule of $\U(n)$, which describes the decomposition of an irrep of $\U(n)$ when restricted to $U(n-1)\cong\U_{[n-1]}$, will be useful for us. It appears, for example, as \cite[Thm.~41.1]{bump2013lie}. 

\begin{theorem}[The branching rule in $\U(n)$] \label{thm:branching rule}
Let $\mu=(\mu_1,\ldots,\mu_n)\in\mathbb{Z}^n$ be a non-increasing vector. Then the restriction of $\rho_\mu$ to $\U_{[n-1]}\cong\U(n-1)$ is given by 
\[
    \rho_\mu|_{\U_{[n-1]}} \cong \bigoplus_{\pi\,\colon\,\pi,\mu~\mathrm{interlace}} \rho_\pi,
\]
where the sum is over the non-increasing vectors $\pi\in\mathbb{Z}^{n-1}$ such that $\pi$ and $\mu$ \textit{interlace}. In particular, the decomposition of $\rho_\mu|_{\U_{[n-1]}}$ is multiplicity-free.
\end{theorem}

\subsubsection*{The decomposition of $\rkl$ and of $\skm$}

Any non-increasing sequence $\mu\in\mathbb{Z}^n$ can also be described in terms of a pair of partitions $(\nu^+,\nu^-)$ as follows. Denote by $r^+$ (respectively, $r^-$) the number of positive (respectively, negative) weights in $\mu$. Let $\nu^+=(\nu^+_1,\ldots,\nu^+_{r^+})$ with $\nu^+_1\ge\ldots\ge\nu^+_{r^+}\ge1$ consist of the positive weights in $\mu$, and let $\nu^-=(\nu^-_1,\ldots,\nu^-_{r^-})$ with $\nu^-_1\ge\ldots\ge\nu^-_{r^-}\ge1$ consist of the opposites of the negative weights in $\mu$. Obviously, $r^++r^-\le n$, and we have 
\[
    \mu=(\nu^+_1,\ldots,\nu^+_{r^+},\underbrace{0,\ldots,0}_{\scriptstyle{n-r^+-r^-~\mathrm{times}}},-\nu^-_{r^-},\ldots,-\nu^-_1).
\]
We also write the irrep $\rho_\mu$ as $\rho_{\nu^+,\nu^-}^{(n)}$.\marginpar{$\rho_{\nu^+,\nu^-}^{(n)}$}

Recall from \eqref{eq:def of Rkl} the representation $\rkl$ of $\U(n)$ given by $A\mapsto A^{\otimes k}\otimes \overline{A}^{\otimes m}$ for any $k,m\in\mathbb{Z}_{\ge0}$. This notation applies to arbitrary $n$. When we need to specify the particular value of $n$ considered, we denote $\rkl^{(n)}$\marginpar{$\rkl^{(n)}$}.
It follows from the Schur-Weyl duality that the decomposition of $R_{k,0}^{(n)}$ consists of all the representations $\rho_{\nu,\emptyset}^{(n)}$ with $\nu\vdash k$ and $\ell(\nu)\le n$,\footnote{For a partition $\nu=(\nu_1,\ldots,\nu_r)\vdash k$ denote by $|\nu|$ the sum of numbers, namely, $|\nu|=k=\sum\nu_i$, and by $\ell(\nu):=r$ the number of parts. We denote by $\emptyset$ the empty partition, so $\emptyset\vdash0$.} where if $\rho_{\nu,\emptyset}^{(n)}$ appears, it does so with multiplicity given by the dimension of the $\sym(k)$-representation $\pi_\nu$. This result generalizes to all $\rkl$, even when $m\ge1$, as long as $n$ is large enough.

\begin{theorem} \cite[Thm.~2.12]{benkart1994tensor} \label{thm:decomposition of Rkl} Let $k,m\in\mathbb{Z}_{\ge0}$. For all $n\ge k+m$, 
\begin{equation} \label{eq:decomp of Rkm for large n}
    \rkl^{(n)}\cong \sum_{j=0}^{\min(k,m)} 
    \sum_{\substack{ \nu^+\vdash k-j \\ \nu^-\vdash m-j}}
    j!\binom{k}{j}\binom{m}{j}\dim(\pi_{\nu^+})\dim(\pi_{\nu^-}) \rho^{(n)}_{\nu^+,\nu^-}
\end{equation}
where $\pi_{\nu^+}$ and $\pi_{\nu^-}$ are the corresponding representations of~ $\Sym(k-j)$ and~ $\sym(m-j)$, respectively.
\end{theorem}

For example, for every $n\ge4$ we have \label{page:R_22 decomp}
\begin{equation*}
    R_{2,2}  \cong  \rho_{(2,0,\ldots,0,-2)} \oplus \rho_{(2,0,\ldots,0,-1,-1)} \oplus \rho_{(1,1,0,\ldots,0,-2)} \oplus \rho_{(1,1,0,\ldots,0,-1,-1)} \oplus 4\cdot\rho_{(1,0,\ldots,0,-1)} \oplus 2\cdot\triv.
\end{equation*}
(In contrast, in $\U(3)$ the decomposition is $R_{2,2} \cong \rho_{(2,0,-2)} \oplus \rho_{(2,-1,-1)} \oplus \rho_{(1,1,-2)} \oplus 4\cdot\rho_{(1,0,-1)} \oplus 2\cdot\triv$, and in $\U(2)$ it is  $R_{2,2} \cong \rho_{(2,-2)} \oplus 3\cdot\rho_{(1,-1)} \oplus 2\cdot\triv$.)

When $n$ is small, namely, when $n< k+m$, Koike \cite{koike1989decomposition} gives a precise recipe for computing the decomposition of $\rkl^{(n)}$ to irreps. An important property for us is that for all $n$, $\rkl^{(n)}$ contains a sub-representation he denotes by $T^\bullet_{k,m}$, and its decomposition, given in \cite[Thm.~1.1]{koike1989decomposition}, is precisely
\[
    T^\bullet_{k,m}\cong \bigoplus_{\substack{\nu^+\vdash k,\nu^-\vdash m\\\ell(\nu^+)+\ell(\nu^-)\le n}} \dim(\pi_{\nu^+})\dim(\pi_{\nu^-})  \rho^{(n)}_{\nu^+,\nu^-}.
\]
Note that when $n\ge k+m$, the sub-representation $T^\bullet_{k,m}$ corresponds exactly to the summand $j=0$ in the main sum in \eqref{eq:decomp of Rkm for large n}.

\begin{cor}\cite{koike1989decomposition} \label{cor:koike nu+nu- appears in R_km}
    Let $\nu^+\vdash k$ and $\nu^-\vdash m$ be integer partition. Then $\rho_{\nu^+,
    \nu^-}^{(n)}$ appears in the decomposition of $\rkl^{(n)}$ for every $n\ge\ell(\nu^+)+\ell(\nu^-)$.
\end{cor}


Koike also provides in \cite[Cor.~2.3.1]{koike1989decomposition} a precise formula for the decomposition of $\rho_{\nu^+,\emptyset}^{(n)}\otimes\rho_{\emptyset,\nu^-}^{(n)}$:
\begin{equation}\label{eq:Koike formula for tensor of poly and anti-poly}
    \rho_{\nu^+,\emptyset}^{(n)}\otimes\rho_{\emptyset,\nu^-}^{(n)} = \sum_{\tau,\eta,\xi}c^{\nu^+}_{\tau,\eta}~ c^{\nu^-}_{\tau,\xi} ~\left[\rho_{\eta,\xi}\right]_n,
\end{equation}
where the sum is over partitions $\tau,\eta,\xi$, the numbers $c^{\nu^+}_{\tau,\eta}$ and $c^{\nu^-}_{\tau,\xi}$ are the Littlewood-Richardson coefficients (e.g., \cite[$\S$I.9]{macdonald1998symmetric}), and $[\rho_{\eta,\xi}]_n$ is given in \cite[Prop.~2.2]{koike1989decomposition}: if $n\ge\ell(\eta)+\ell(\xi)$, it is simply $\rho_{\eta,\xi}^{(n)}$, and otherwise it is zero or $\pm$(some irreducible representation). Note that the summation is finite as, for example, we must have that in terms of Young diagrams, $\tau$ is a subdiagram of $\nu^+$ and of $\nu^-$, and also $|\tau|+|\eta|=|\nu^+|$ and $|\tau|+|\xi| = |\nu^-|$ for the Littlewood-Richardson coefficients to not vanish.

Let us apply this formula of Koike when $\nu^+=(k)$ and $\nu^-=(m)$. Recall our notation $\skm$ from \eqref{eq:def of Sk,m}. Note that $\rho_{(k),\emptyset}^{(n)}$ is the $k^\mathrm{th}$-symmetric power $S_{k,0}=\mathrm{Sym}^k(\rho_{\std})$, where $\rho_\std=\rho_\std^{(n)}=R_{1,0}^{(n)}=\rho_{(1),\emptyset}^{(n)}=\rho_{(1,0,\ldots,0)}$ is the standard representation of $\U(n)$ mapping $A\mapsto A$. Likewise, $\rho_{\emptyset,(m)}^{(n)}=S_{0,m}=\Sym^m(\rho_{\std}^{~*})=\Sym^m(\rho_{\std})^*$ is the \textit{dual} representation of the $m^\mathrm{th}$-symmetric power. In this case, the summand corresponding to some $\tau$, $\eta$ and $\xi$ in \eqref{eq:Koike formula for tensor of poly and anti-poly} does not vanish if and only if for some $0\le j\le \min(k,m)$ we have $\tau=(j)\vdash j$, $\eta=(k-j)\vdash k-j$ and $\xi=(m-j)\vdash m-j$. In each of these cases, the corresponding Littlewood-Richardson coefficients are 1. We conclude:

\begin{cor} \label{cor:decomp of sym-k tensor sym-m-dual}
    Let $k,m\in\mathbb{Z}_{\ge0}$. Then for all $n\ge2$
    \[
        \skm=\mathrm{Sym}^k(\rho_{\std})\otimes \mathrm{Sym}^m(\rho_{\std}^{~*})=
        \rho_{(k),\emptyset}^{(n)}\otimes\rho_{\emptyset,(m)}^{(n)}=\bigoplus_{j=0}^{\min(k,m)}\rho_{\left(k-j,0,\ldots,0,-(m-j)\right)}.
    \]
\end{cor}

\subsection{Integration in $\U(n)$ and Weingarten calculus} \label{subsec:Weingarten}

The Weingarten calculus is a method to compute integrals over Haar-random unitary matrices (as well as Haar-random matrices in other natural families of groups). Originally developed in \cite{samuel1980integral,collins2003moments,collins2006integration}, this method gives concrete, finite combinatorial formulas for the average value of polynomial expressions in matrix entries and their complex-conjugates. A first observation is that ``non-balanced" integrals vanish. We repeat the simple argument here for the benefit of the reader. Below we apply the results of this subsection to the various subgroups $\U_B$ of $\U(n)$, so we choose to denote the dimension of the unitary group here by $d$. Recall that $\mu_{[d]}$ is the Haar measure on $\U(d)$.

\begin{lemma} \label{lem:unbalanced integrals vanish}
    The integral
    \[
        \int_{A\in\U(d)}A_{i_1,j_1}\cdots A_{i_k,j_k}\cdot \overline{A_{i'_1,j'_1}}\cdots \overline{A_{i'_m,j'_m}} d\mu_{[d]}
    \]
    vanishes unless $k=m$ and there are equalities of multisets $\{i_1,\ldots,i_k\}=\{i'_1,\ldots,i'_m\}$ and\linebreak $\{j_1,\ldots,j_k\}=\{j'_1,\ldots,j'_m\}$.
\end{lemma}
\begin{proof}
    Denote the integral by $I$. First assume that $\{i_1,\ldots,i_k\}\ne\{i'_1,\ldots,i'_m\}$ as multisets (this is certainly the case if $k\ne m$). Then there is some index $t\in[d]$ that appears $a_t$ times in the first multiset and $b_t$ times in the second, with $a_t\ne b_t$. The Haar measure is left-invariant, so in the integral we may replace $A$ with $g_{t,\theta} A$, where $g_{t,\theta}\in\U(d)$ is the diagonal matrix with $\theta\in S^1$ in the $(t,t)$-entry, and 1 in every other diagonal entry. The result is
    \[
        I=\int_{A\in\U(d)}(g_{t,\theta}A)_{i_1,j_1}\cdots (g_{t,\theta}A)_{i_k,j_k}\cdot \overline{(g_{t,\theta}A)_{i'_1,j'_1}}\cdots \overline{(g_{t,\theta}A)_{i'_m,j'_m}} = \theta^{a_t-b_t}I.
    \]
    But $\theta\in S^1$ is arbitrary, so if $a_t\ne b_t$ we must have $I=0$. A similar argument, using the right-invariance of the Haar measure, proves the necessity of the equality of the other two multisets.
\end{proof}

The Weingarten calculus allows one to approach these integrals in the cases where they do not vanish, in order to compute them explicitly as well as to derive asymptotic properties as $d$ grows. The formulas use the ``Weingarten function'': for every $k,d\in\Z_{\ge1}$, this is a function $\wg_{k,d}\colon\sym(k)\to\mathbb{Q}$.\marginpar{$\wg$} When $d\ge k$, it can be defined by 
\begin{equation} \label{eq:wg definition when d ge k}
    \wg_{k,d}(\sigma)\defi\int_{A\in\U(d)}A_{1,1}A_{2,2}\cdots A_{k,k}\overline{A_{1,\sigma(1)}A_{2,\sigma(2)}\cdots A_{k,\sigma(k)}}.
\end{equation}
The definition \eqref{eq:wg definition when d ge k} coincides with the following one, appearing in \cite[Prop.~2.3]{collins2006integration}, which applies to all values $k,d\in\Z_{\ge1}$. It expresses $\wg_{k,d}$, which is a class function on $\sym(k)$, as a linear combination of the irreducible characters of $\sym(k)$.

\begin{equation} \label{eq:wein function by CS06}
    \wg_{k,d} \defi \frac{1}{k!^2}\sum_{\substack{\nu\vdash k\\\ell(\nu)\le d}}\frac{\dim(\pi_\nu)^2}{\dim\left(\rho_{\nu,\emptyset}^{(d)}\right)} \chi_\nu,
\end{equation}
where $\pi_\nu$ is the irrep of $\sym(k)$ corresponding to $\nu$ (as above), and \marginpar{$\chi_\nu$}$\chi_v=\mathrm{tr}\circ\pi_\nu\colon \sym(k)\to\Z$ its character. Note that for any fixed partition $\nu$, the denominator $\dim(\rho_{\nu,\emptyset}^{(d)})$ in \eqref{eq:wein function by CS06} is a polynomial in $d$ (it is also equal to the Schur polynomial $s_\nu$ on $d$ variables, evaluated at $1,\ldots,1$). Hence, for fixed $k$ and $\sigma\in\sym(k)$, the function $\wg_{k,d}(\sigma)$ coincides with some rational function in $\Q(d)$ for all $d\ge k$. For example, for all $d\ge2$, $\wg_{2,d}(\mathrm{Id})=\frac{1}{d^2-1}$ and $\wg_{2,d}((1~2))=\frac{-1}{d^3-d}$. 

The main significance of the Weingarten function is the following result. In this form, it appeared first in $\cite{collins2006integration}$.

\begin{theorem}\cite[Cor.~2.4]{collins2006integration} \label{thm:weingarten}
    Let $k,d\ge1$, and $i_1,\ldots,i_k,j_1,\ldots,j_k,i'_1,\ldots,i'_k,j'_1,\ldots,j'_k\in[d]$. Then
    \begin{eqnarray} \label{eq:Weingarten main result}
        \int_{A\in\U(d)}A_{i_1,j_1}\cdots A_{i_k,j_k}\cdot \overline{A_{i'_1,j'_1}}\cdots \overline{A_{i'_k,j'_k}} ~d\mu_{[d]} = ~~~~~~~~~~~~~~~~~~~~~~~~~~~~~~~~~~~~\nonumber \\
        \sum_{\sigma,\tau\in\sym(k)} \delta_{i_1,i'_{\sigma(1)}}\cdots\delta_{i_k,i'_{\sigma(k)}} \delta_{j_1,j'_{\tau(1)}}\cdots\delta_{j_k,j'_{\tau(k)}}\wg_{k,d}(\tau\sigma^{-1}),
    \end{eqnarray}
    where $\delta_{i,j}$ is the Kronecker delta.
\end{theorem}

We will need the following lemmas about the sum, and signed sum, of the values of the Weingarten function for a fixed $k$ and $d$.

\begin{lemma} \label{lem:sum and signed sum of weingarten}
    For every $k,d\in\Z_{\ge1}$, 
    \begin{enumerate}
        \item \label{enu:sum of wg} $$\sum_{\sigma\in\sym(k)}\wg_{k,d}(\sigma)=\frac{1}{k!\cdot\ms{d}{k}}=\frac{1}{d(d+1)(d+2)\cdots(d+k-1)}.$$
        \item \label{enu:signed sum of wg} $$\sum_{\sigma\in\sym(k)}\sign(\sigma)\cdot\wg_{k,d}(\sigma)=
        \begin{cases}
            \frac{1}{k!\cdot\binom{d}{k}}=\frac{1}{d(d-1)(d-2)\cdots(d-k+1)} & d\ge k\\
            0 & d<k.
        \end{cases}
        $$
    \end{enumerate}    
\end{lemma}

\begin{proof}
    Recall that the ordinary inner product in a finite group $G$ is defined by $\langle f_1,f_2\rangle_G\defi \frac{1}{|G|}\sum_{g\in G}f_1(g)\overline{f_2(g)}$ for all $f_1,f_2\colon G\to\C$. We have,
    \[
        \frac{1}{k!}\sum_{\sigma\in\sym(k)}\wg_{k,d}(\sigma) = \langle \wg_{k,d},\chi_{(k)}\rangle_{\sym(k)} = \frac{1}{k!^2}\frac{\dim(\pi_{(k)})^2}{\dim\left(\rho_{(k),\emptyset}^{(d)}\right)} = \frac{1}{k!^2\cdot \ms{d}{k}},
    \]
    where the first equality holds as $\pi_{(k)}$ is the trivial representation of $\sym(k)$,  the second equality by \eqref{eq:wein function by CS06} and the orthogonality of characters, and the third by recalling that $\rho_{(k),\emptyset}^{(d)}$ is the representation $S_{k,0}$ of $\U(d)$, so its dimension is $\ms{d}{k}$. This proves item \ref{enu:sum of wg}.

    Item \ref{enu:signed sum of wg} is proven similarly, recalling that the sign character of $\sym(k)$ corresponds to the partition $(1^k)=(1,\ldots,1)$ and that $\rho_{(1^k),\emptyset}^{(d)}$ is the representation $\wedge^k(\rho_\std)$ whose dimension is $\binom{d}{k}$.
\end{proof}

\subsection{Basic properties of the spectrum and operator inequalities} \label{subsec:spectrum is real non-negative and bounded}
We collect here some basic facts about the spectra of $\L(\Gamma,\rho)$ and $\L(\Gamma,\kmp_k)$, and mention some easy operator inequalities.

\begin{lemma} \label{lem:L as sum of projections}
    Let $\rho\colon\U(n)\to\gl(V)$ be a $d$-dimensional unitary representation of $\U(n)$ and $\Gamma=([n],w)$ a weighted hypergraph. Then $\L(\Gamma,\rho)=\sum_{B\subseteq[n]}w_B(I_d-P_B)\in\End(V)$, where $P_B\in\End(V)$\marginpar{$P_B$} is the orthogonal projection on the $\U_B$-invariant subspace.
\end{lemma}
\begin{proof}
    This is based on the following general fact: if $H$ is a closed subgroup of a compact group $G$ and $\mu_H$ is the normalized Haar measure of $H$, then for any finite-dimensional unitary representation $\pi\colon G\to\gl(V)$, the operator 
    \[
        P_H\defi \int_{h\in H}\pi(h)d\mu_H(h)\in\gl(V)
    \]
    is the orthogonal projection onto the $H$-invariant subspace $V^H=\{v\in V\,\mid\,\pi(h)v=v ~~\forall h\in H\}$.
\end{proof}

\begin{lemma} \label{lem:L of reps - basic properties of spectrum}
    Let $\rho$ be a $d$-dimensional representation of $\U(n)$ and $\Gamma=([n],w)$ a hypergraph with non-negative weights. Then the spectrum of $\L(\Gamma,\rho)$ is real and contained in the interval $[0,\sum_{B\in[n]}w_B]$.
\end{lemma}

\begin{proof}
    Every finite dimensional representation of a compact group admits an inner product which makes it unitary, so we assume without loss of generality that $\rho$ is unitary. As orthogonal projections are self-adjoint, so it $\L(\Gamma,\rho)$, by Lemma \ref{lem:L as sum of projections}. Hence the spectrum of $\L(\Gamma,\rho)$ is real. As all the eigenvalues of a projection are $0$ or $1$, the spectrum of $\L(\Gamma,\rho) = \sum_{B\subseteq[n]}w_B(I_d-P_B)$ lies in $[0,\sum_{B\in[n]}w_B]$.
\end{proof}

\begin{lemma} \label{lem:evalues of N_B are 0 or 1}
    Let $B\subseteq[n]$. Then for all $k\in\Z_{\ge1}$, the operator $\N_B\in\End(\MS(n,k))$ is self-adjoint and its eigenvalues are all 0 or 1. 
\end{lemma}
Here, we mean self-adjoint with respect to the standard inner product defined with the basis elements corresponding to $\ms{[n]}{k}$.
\begin{proof}
    Let $J_d$\marginpar{$J_d$} be the $d\times d$ all-1 matrix. By definition, the matrix corresponding to $\N_B$ with respect to the basis $\ms{[n]}{k}$ of $\MS(n,k)$ is block-diagonal, with blocks corresponding to the possible values of multisets $f\in\ms{[n]}{k}$ outside $B$. For a given $f\in\ms{[n]}{k}$, denote by $\ell=\sum_{i\in B}f(i)$ the number of particles from $B$ in $f$, and by $c_f=\ms{|B|}{\ell}$ the number of multisets $g\in\ms{[n]}{k}$ agreeing with $f$ outside $B$. Then the block corresponding to $f$ is of size $c_f\times c_f$, and its value is precisely $\frac{1}{c_f}J_{c_f}$, whose spectrum consists of a single one and $c_f-1$ zeros. This also shows that this matrix is symmetric, hence $\N_B$ is self-adjoint.
\end{proof}

\begin{cor} \label{cor:spectrum of Laplacian on KMP in [0,sum of weights]}
    For every hypergraph $L=([n],w)$ with non-negative weights and every $k\in\Z_{\ge1}$, the spectrum of $\L(\Gamma,\kmp_k)$ is real and contained in $[0,\sum_B w_B]$.
\end{cor}

Let $T_1,T_2$ be two operators on the regular representation $L^2(\U(n))$. We write $T_1\le T_2$\marginpar{$T_1\le T_2$} to mean that the difference $T_2-T_1$ is a positive semi-definite operator on $L^2(\U(n))$: $\langle(T_2-T_1).f,f\rangle\ge0$ for all $f\in L^2(\U(n))$. Whenever we use this notation, the operators preserve sub-representations, and then the inequality means, equivalently, that the corresponding spectrum in every finite dimensional representation $\rho$ of $\U(n)$ is real and non-negative.

\begin{lemma} \label{lem:inequality between subset}
    For $B\subseteq[n]$ denote by $\L_B=I-P_B$ the operator in the regular representation $L^2(\U(n))$.
    If $B_1\subseteq B_2\subseteq[n]$ then $\L_{B_1}\le\L_{B_2}$.
\end{lemma}
\begin{proof}
    Recall that $P_B$ is the orthogonal projection on the trivial isotypic component in the restriction of the regular representation to $\U_B$. Hence $I-P_B$ is the orthogonal projection on its orthogonal complement. The lemma follows as the trivial component in the restriction to $\U_{B_2}$ is a subspace of the trivial component in the restriction to $\U_{B_1}$.
\end{proof}

\begin{lemma} \label{lem:inequalities remain true with branching}
    Let $S\subseteq [n]$, and let $\Gamma=(S,w)$ and $\Gamma'=(S,w')$ be two weighted hypergraphs defined on the set of vertices $S$. If $\L(\Gamma)\le\L(\Gamma')$ in $L^2(\U_S)$, then the inequality holds also in $L^2(\U(n))$.
\end{lemma}
\begin{proof}
    For any representation $\rho$ of $\U(n)$, the inequality $\L(\Gamma,\rho)\le\L(\Gamma',\rho)$ when $\Gamma,\Gamma'$ are thought of as hypergraphs on the vertex set $[n]$, is equivalent to the inequality $\L(\Gamma,\rho|_{\U_S})\le\L(\Gamma',\rho|_{\U_S})$ when $\Gamma,\Gamma'$ are thought of as hypergraphs on the vertex set $S$.
\end{proof}

\section{The torus-invariant subspace of a representation of $\U(n)$} \label{sec:torus-invariant subspace}

Recall the definition of the torus subgroup $T_n\le \U(n)$ and of the torus-invariant subspace $\torinv(\rho)$ of a representation $\rho$ (Definition \ref{def:torus-invariant subspace}). In this section we prove that $\torinv(\rho)$ is indeed invariant under $\L(\Gamma,\rho)$ for all weighted hypergraphs $\Gamma$ (Proposition \ref{prop:torus-inv subspace is invariant for every Gamma}), and that it is trivial for non-balanced representations (Corollary \ref{cor:torus-inv subspace empty for unbalanced irreps}). We also prove Theorem \ref{thm:kmp_k is central block of...}, which identifies the discrete KMP processes on $\Gamma$ as torus-invariant subspaces of certain representations, and Theorem \ref{thm:(k,-k) dominates non-balanced parts}, which shows that the spectral gap of the $\U(n)$-spectrum of $\Gamma$ is obtained by torus-invariant subspaces.

\begin{prop} \label{prop:torus-inv subspace is invariant for every Gamma}
    The torus-invariant subspace $\torinv(\rho)$ of any finite-dimensional representation $\rho$ of $\U(n)$ is invariant under $\L(\Gamma,\rho)$ for any weighted hypergraph $\Gamma$.
\end{prop}

We use the following lemmas. Recall that for every $B\subseteq[n]$ we defined the measure $\mu_B$ on $\U(n)$ as the Haar measure on the subgroup $\U_B$.

\begin{lemma} \label{lem:B_1 contained in B_2}
    Let $B_1\subseteq B_2\subseteq[n]$. Then $\mu_{B_1}*\mu_{B_2}=\mu_{B_2}*\mu_{B_1}=\mu_{B_2}$. 
\end{lemma}

\begin{proof}
    We have
    \begin{eqnarray*}
       (\mu_{B_1}*\mu_{B_2})(E) &=& \int_{x\in\U_{B_1}}\left(\int_{y\in\U_{B_2}}\mathds{1}_E(xy)d\mu_{B_2}(y)\right)d\mu_{B_1}(x) \\
       &=& \int_{x\in\U_{B_1}}\left(\int_{y\in\U_{B_2}}\mathds{1}_E(y)d\mu_{B_2}(y)\right)d\mu_{B_1}(x) \\
       &=&  \int_{x\in\U_{B_1}}\mu_{B_2}(E)d\mu_{B_1}(x) = \mu_{B_2}(E),
    \end{eqnarray*}
    where the second equality is by the left-invariance of the Haar measure on $\U_{B_2}$ and the fact that $x\in\U_{B_1}\le\U_{B_2}$. Hence $\mu_{B_1}*\mu_{B_2}=\mu_{B_2}$. A parallel argument, using the right-invariance of Haar measures, shows that $\mu_{B_2}*\mu_{B_1}=\mu_{B_2}$.
\end{proof}

\begin{lemma} \label{lem:convolution of disjoint mu_Bs}
    Let $B_1,B_2\subseteq[n]$ satisfy $B_1\cap B_2=\emptyset$. Then $\mu_{B_1}*\mu_{B_2}=\mu_{B_2}*\mu_{B_1}$.
\end{lemma}

\begin{proof}
    If $B_1\cap B_2=\emptyset$, then $U_{B_1}$ and $\U_{B_2}$ are commuting subgroups with trivial intersection, and the conclusion follows from Fubini's theorem.
\end{proof}

\begin{proof}[Proof of Proposition \ref{prop:torus-inv subspace is invariant for every Gamma}]
    Let $\rho\colon\U(n)\to\gl(V)$ be a $d$-dimensional representation and $\Gamma=([n],w)$ a weighted hypergraph. By Lemma \ref{lem:L as sum of projections}, $\L(\Gamma,\rho)=\sum_{B\subseteq[n]}w_B(I_d-P_B)\in\mathrm{End}(V)$. By Lemmas \ref{lem:B_1 contained in B_2} and \ref{lem:convolution of disjoint mu_Bs}, for every $i\in[n]$, $P_{\{i\}}$ commutes with $P_B$ for all $B$. Hence all the operators
    \[
        \L(\Gamma,\rho),P_{\{1\}},\ldots,P_{\{n\}}\in\mathrm{End}(V)
    \]
    commute with each other and thus admit a simultaneous diagonalization. In particular, the intersection $V^{\{1\}}\cap\ldots\cap V^{\{n\}}$, which is precisely the torus-invariant subspace $\torinv(\rho)$, is invariant under $\L(\Gamma,\rho)$.    
\end{proof}

The proof of Proposition \ref{prop:torus-inv subspace is invariant for every Gamma} actually yields a decomposition of $V$ as a direct sum of $2^n$ subspaces, each of which is invariant under $\L(\Gamma,\rho)$. (Of course, each of these subspaces may be trivial for specific $\rho$.)

\begin{cor} \label{cor:2 to the n invariant subspaces}
    Let $\rho\colon\U(n)\to\gl(V)$ be a finite dimensional representation. Denote
    \[
        V^{\{i\},1}\defi V^{\{i\}}=P_{\{i\}}(V) \mathrm{~~~and~~~} V^{\{i\},0}\defi\ker\left(P_{\{i\}}\right).
    \] 
    Then
    \[
        V=\bigoplus_{(\varepsilon_1,\ldots,\varepsilon_n)\in\{0,1\}^n}V^{\{1\},\varepsilon_1}\cap\ldots\cap V^{\{n\},\varepsilon_n},
    \]
    and each of these (possibly trivial) subspaces is invariant under $\L(\Gamma,\rho)$ for every weighted hypergraph $\Gamma$.
\end{cor}

In the notation of Corollary \ref{cor:2 to the n invariant subspaces}, $\torinv(\rho)=V^{\{1\},1}\cap\ldots\cap V^{\{n\},1}$.


\subsection{Discrete KMP as a torus-invariant subspace} \label{subsec:discrete KMP as tor-inv subspace}

Recall the notation $\L_\torinv(\Gamma,\rho)$ from page \pageref{page:L_torinv} for the restriction of $\L(\Gamma,\rho)$ to the invariant subspace $\torinv(\rho)$. We now find explicitly the torus-invariant subspaces of the representations $\rkl$ and $\skm$, and show that  $\L_\torinv(\Gamma,S_{k,k})$ is precisely the discrete KMP process on $\Gamma$ with $k$ particles. 
\medskip

We begin with $\rkl$. Consider the standard basis $e_1,\ldots,e_n$ of $V=\mathbb{C}^n$, and mark the dual basis of the dual space $V^*$ by $e^1,\ldots,e^n$ (so $\langle e_i,e^j\rangle=\delta_{ij}$) . The canonical basis for $V^{\otimes k}\otimes (V^*)^{\otimes m}$ is
\[
    \left\{ e_{i_1}\otimes\ldots\otimes e_{i_k}\otimes e^{j_1}\otimes\ldots\otimes e^{j_m} \,\mid\, i_1,\ldots,i_k,j_1,\ldots,j_m\in[n] \right\}.
\]
Note that the action of $A\in\U(n)$ on $\varphi\in V^*$ is given by $\varphi\mapsto\overline{A}\varphi$, so its action on $V^{\otimes k}\otimes (V^*)^{\otimes m}$ is given by the matrix $A^{\otimes k}\otimes \overline{A}^{\otimes m}$.

\begin{prop} \label{prop:the torus-invariant subspace of Rkm}
    The torus-invariant subspace of $\rkl$ is spanned by the basis elements 
    \[
    \left\{ e_{i_1}\otimes\ldots\otimes e_{i_k}\otimes e^{j_1}\otimes\ldots\otimes e^{j_m} \,\mid\, \{i_1,\ldots,i_k\}\stackrel{\mathrm{as~multisets}}{=}\{j_1,\ldots,j_m\} \right\}.
\]
    In particular, if $k\ne m$, the torus-invariant subspace of $\rkl$ is $\{0\}$.
\end{prop}

\begin{proof}
    The Haar measure on $T_n$ is the Lebesgue measure on $S^1\times\ldots\times S^1$. Let $A\in T_n$ be Haar-random and consider $\mathbb{E}_{A\in T_n}[A^{\otimes k}\otimes \overline{A}^{\otimes m}]$ -- the matrix giving the projection onto $\torinv(R_{k,m})$. The entry in row $(i_1,\ldots,i_k,j_1,\ldots,j_m)$ and column $(i'_1,\ldots,i'_k,j'_1,\ldots,j'_m)$ is 
    \[
        \mathbb{E}_{A\in T_n}\left[A_{i_1,i'_1}\cdots A_{i_k,i'_k}\cdot\overline{A_{j_1,j'_1}}\cdots\overline{A_{j_m,j'_m}}\right],
    \]
    which is either 1 or 0. It is 1 if and only if only diagonal entries of $A$ are in play, and each of these appears in a ``balanced'' manner. This is the case precisely if $(i)$ $i_t=i'_t$ for all $t$, $(ii)$ $j_t=j'_t$ for all $t$, and $(iii)$ $\{i_1,\ldots,i_k\}=\{j_1,\ldots,j_m\}$ as multisets.
\end{proof}

\begin{cor} \label{cor:torus-inv subspace empty for unbalanced irreps}
    The torus-invariant subspace of an unbalanced irreducible representation of $\U(n)$ is $\{0\}$. In particular, whenever $k\ne m$, we also have $\torinv(\skm)=\{0\}$.
\end{cor}
\begin{proof}
    The first statement is immediate from Corollary \ref{cor:koike nu+nu- appears in R_km} combined with Proposition \ref{prop:the torus-invariant subspace of Rkm}, noticing that if $\rho'$ is a sub-representation of $\rho$, then the torus-invariant subspace of $\rho'$ is contained in that of $\rho$. For the second statement, recall that $S_{k,m}$ decomposes to unbalanced irreps by Corollary \ref{cor:decomp of sym-k tensor sym-m-dual}.
\end{proof}

Fix two partitions $\nu^+$ and $\nu^-$. It follows from \cite[Eq.~(0.3)]{koike1989decomposition} that the dimension of $\rho_{\nu^+,\nu^-}^{(n)}$ is given by a polynomial in $n$ of degree $|\nu^+|+|\nu^-|$ (the precise derivation of this result from \cite{koike1989decomposition} is elaborated, e.g., in \cite[$\S$3.2]{puder2023stable}). In particular, if the partition is balanced, namely, if $|\nu^+|=|\nu^-|=k$ for some $k$, the dimension is of order $n^{2k}$.
The following corollary shows that even in this case, where the torus-invariant subspace of $\rho$ is not empty, it must be of dimension at most of order square root of $\dim(\rho)$.

\begin{cor} \label{cor:torus-inv of balanced is of dim sqrt}
    Let $\nu^+,\nu^-\vdash k$ be two partitions. Then 
    \[        \dim\left(\torinv\left(\rho_{\nu^+,\nu^-}^{(n)}\right)\right)=O(n^k).
    \]
\end{cor}
\begin{proof}
    By Corollary \ref{cor:koike nu+nu- appears in R_km}, $\rho_{\nu^+,\nu^-}^{(n)}$ is a sub-representation of $R_{k,k}$, so $\torinv(\rho_{\nu^+,\nu^-}^{(n)})$ can be realized as a subspace of $\torinv(R_{k,k}^{(n)})$. By Proposition \ref{prop:the torus-invariant subspace of Rkm}, the dimension of $\torinv(R_{k,k}^{(n)})$ is given by a polynomial in $n$ of degree $k$.
\end{proof}

\begin{remark}
    It is likely that the order $n^k$ is precise, and not only an upper bound. In fact, it is likely that for all $\nu^+,\nu^-\vdash k$, the dimension of $\torinv(\rho_{\nu^+,\nu^-}^{(n)})$ is given by a polynomial in $n$ of degree $k$ for all $n\ge2k$. 
\end{remark}

Our next goal is to prove Theorem \ref{thm:kmp_k is central block of...}, which states that 
\[
    \L_\torinv\left(\Gamma,S_{k,k}\right) \cong \L(\Gamma,\kmp_k).
\]

We begin by identifying the subspace $\torinv(S_{k,k})$. A convenient way to think about $\Sym^k(\rho_\std)$, the $k^\mathrm{th}$-symmetric power of the standard representation of $\U(n)$, is by $k$-homogeneous polynomials in $n$ commuting variables, namely, as the space
\[
    \mathbb{C}[x_1,\ldots,x_n]_k\defi \mathrm{Span}_{\mathbb{C}}\left\{ x_1^{\alpha_1}\cdots x_n^{\alpha_n}\,\mid\, \alpha_1,\ldots,\alpha_n\in\mathbb{Z}_{\ge0}, \sum\alpha_i=k \right\},
\]
where the action of $A\in\U(n)$ is defined by its action on a single variable, given by 
\begin{equation} \label{eq:action of A on variable}
    A.x_i\defi\sum_{j=1}^n A_{j,i}x_j.
\end{equation}
Similarly, the representation $S_{k,k}=\Sym^k(\rho_\std)\otimes \Sym^k(\rho_\std^{~*})$ can be realized as the action on the space of polynomials in $2n$ commuting variables $x_1,\ldots,x_n,y_1,\ldots,y_n$, where $A\in\U(n)$ acts on every $x_i$ as in \eqref{eq:action of A on variable}, and on every $y_i$ as $A.y_i=\sum_{j=1}^n\overline{A_{j,i}}y_j$.

\begin{prop} \label{prop:the tor-inv subspace of sym-k otimes sym-k-dual}
    The torus-invariant subspace of $S_{k,k}$ has a basis
    \begin{equation} \label{eq:basis of tor-inv of Skk}
        \left\{ x_{i_1}\cdots x_{i_k} y_{i_1}\cdots y_{i_k} \,\mid\, 1\le i_1\le\ldots\le i_k \le n\right\}.
    \end{equation}
\end{prop}
\begin{proof}
    We will prove that the projection matrix of $P_{T_n}$ in the standard basis of $S_{k,k}$ is $0,1$-diagonal matrix with $1$ precisely in the basis elements of the form \eqref{eq:basis of tor-inv of Skk}.    
    Consider the action of $P_{T_n}$, the projection on the torus-invariant subspace, on the monomial $x_{i_1}\cdots x_{i_k} y_{j_1}\cdots y_{j_k}$ with $i_1\le i_2\le\ldots\le i_k$ and $j_1\le j_2\le\ldots\le j_k$. In the resulting polynomial, the coefficient of $x_{i'_1} \cdots x_{i'_k} y_{j'_1} \cdots y_{j'_k}$ (here the indices are sorted as well) is
    \[
        \mathbb{E}_{A\in T_n}\left[\left(\sum_{\substack{i''_1,\ldots,i''_k\in[n]~\mathrm{s.t.}\\ \{i''_1,\ldots,i''_k\}\stackrel{\mathrm{as~multisets}}{=}\{i'_1,\ldots,i'_k\}}} A_{i''_1,i_1}\cdots A_{i''_k,i_k} \right)
        \left(\sum_{\substack{j''_1,\ldots,j''_k\in[n]~\mathrm{s.t.}\\ \{j''_1,\ldots,j''_k\}\stackrel{\mathrm{as~multisets}}{=}\{j'_1,\ldots,j'_k\}}} \overline{A_{j''_1,j_1}}\cdots \overline{A_{j''_k,j_k}} \right) \right].        
    \]
    By the linearity of the expectation, this is equal to a sum over 
    \[
        \mathbb{E}_{A\in T_n}\left[A_{i''_1,i_1}\cdots A_{i''_k,i_k} \cdot
        \overline{A_{j''_1,j_1}}\cdots \overline{A_{j''_k,j_k}} \right]
    \]
    with $i''_1,\ldots,i''_k,j''_1,\ldots,j''_k$ as in the original expression. But for this term to not vanish, we must have $i''_t=i_t$ and $j''_t=j_t$ for all $t\in[k]$ (as $A\in T_n$ is diagonal). In particular, this means that the $i''_t$'s are sorted, and hence $i_t=i''_t=i'_t$ and $j_t=j''_t=j'_t$ for all $t$. So the coefficient computed above does not vanish only if it is a diagonal coefficient: $i'_t=i_t$ and $j'_t=j_t$ for all $t$. In this case, it is equal to 
    \[
        \mathbb{E}_{A\in T_n}\left[\left(A_{i_1,i_1}\cdots A_{i_k,i_k} \right)
        \left( \overline{A_{j_1,j_1}}\cdots \overline{A_{j_k,j_k}} \right) \right].        
    \]
    Finally, this diagonal entry of $P_{T_n}$ does not vanish if and only if each diagonal entry of $A$ appears a balanced number of times. This translates exactly to that $(i_1,\ldots,i_k)=(j_1,\ldots,j_k)$. We conclude that $P_{T_n}$ is the projection onto the subspace spanned by the polynomials in \eqref{eq:basis of tor-inv of Skk}.
\end{proof}

So far we have established that the torus-invariant subspace of $S_{k,k}$ has a basis corresponding to $\ms{[n]}{k}$: the multisets of size $k$ with elements from $[n]$ are given by $\{i_1,\ldots,i_k\}$ from \eqref{eq:basis of tor-inv of Skk}. We now show that the linear action of $\L(\Gamma,S_{k,k})$ on this subspace is identical to that of $\L(\Gamma,\kmp_k)$. 

In fact, even though the basis \eqref{eq:basis of tor-inv of Skk} of $\torinv(S_{k,k})$ has a natural one-to-one correspondence with the basis $\ms{[n]}{k}$ of the space $\MS(k,n)$ on which $\kmp_k(\Gamma)$ acts, it is not properly normalized. It turns out that the right normalization, which makes the two actions of $\Gamma$ identical, is the one turning the basis \eqref{eq:basis of tor-inv of Skk} into an orthonormal basis for $\torinv(S_{k,k})$, when the $S_{k,k}$ is realized as a subrepresentation of $R_{k,k}$. Let us explain what this means.

Consider first $R_{k,0}=V^{\otimes k}$ where $V=\C^n$ is the standard representation. The elements $e_{i_1}\otimes\ldots\otimes e_{i_k}$ form an orthonormal basis. To realize $S_{k,0}=\Sym^k(V)$ as a subspace of $R_{k,0}$, one can take as an orthonormal basis the elements 
\[
    \left\{ \frac{1}{\sqrt{C_{\ii}}}\sum_{\substack{i'_1,\ldots,i'_k\in[n]\\ \{i'_1,\ldots,i'_k\}\stackrel{\mathrm{as~multisets}}{=}\ii}}  e_{i'_1}\otimes\ldots\otimes e_{i'_k}\,\middle|\,\ii=\{i_1,\ldots,i_k\},~1\le i_1\le\ldots\le i_k\le n\right\},
\]
where $c_\ii$\marginpar{$c_\ii$} is the number of elements in the sum, namely, the number of $(i'_1,\ldots,i'_k)$ which are identical to $\ii$ as multisets. When $S_{k,0}$ is constructed as an action on commuting homogeneous, degree-$k$, polynomials on $n$ variables $x_1,\ldots,x_n$, all $c_\ii$ summands become identical, and this orthonormal basis translates to\marginpar{$x_\ii$}
\[
    \left\{ \sqrt{C_{\ii}}\cdot x_\ii \defi \sqrt{C_{\ii}}\cdot x_{i_1}\cdots x_{i_k}
    \,\middle|\,\ii=\{i_1,\ldots,i_k\},~1\le i_1\le\ldots\le i_k\le n\right\}.
\]
Finally, in $S_{k,k}=S_{k,0}\otimes S_{0,k}$, this leads to the orthonormal basis with elements $\sqrt{c_\ii c_\jj}\cdot x_\ii y_\jj$. We obtain the following orthonormal basis for the torus-invariant subspace $\torinv(S_{k,k})$:
\begin{equation} \label{eq:ortho basis of tor-inv of Skk}
    \left\{ c_\ii \cdot x_\ii y_\ii = c_\ii\cdot x_{i_1}\cdots x_{i_k} y_{i_1}\cdots y_{i_k} \,\middle|\, \ii=\{i_1,\ldots,i_k\},~1\le i_1\le\ldots\le i_k \le n\right\}.
\end{equation}

\begin{proof}[Proof of Theorem \ref{thm:kmp_k is central block of...}]
    Let $\Gamma=([n],w)$ be a weighted hypergraph. Recall that $\L(\Gamma,\kmp_k)=\sum_{B\subseteq[n]}w_B(I-\N_B)$, and that $\L(\Gamma,S_{k,k})=\sum_{B\subseteq[n]}w_B(I-P_B)$, where $P_B$ is the projection onto the $\U_B$-invariant subspace of $S_{k,k}$. 
    It is thus enough to show that with suitable bases, the linear operators $P_B|_{\torinv(S_{k,k})}$ and $\N_B$ are identical for all $B\subseteq[n]$. We show this using the bases \eqref{eq:ortho basis of tor-inv of Skk} for $\torinv(S_{k,k})$ and $\ms{[n]}{k}$ for $\MS(n,k)$ (with the natural one-to-one correspondence between them).

    So let $B\subseteq[n]$ and $\ii=f\in\ms{[n]}{k}$. We think of $\ii$ as a multiset of size $k$ of elements from $[n]$, and of $f$ as a function $[k]\to\Z_{\ge0}$ with $\sum f(i)=k$. Write $\ii=\ii_1\sqcup\ii_2$ with $\ii_1$ a multiset containing only elements from $B$, and $\ii_2$ a multiset containing only elements from $[n]\setminus B$. Denote $b=|B|$ and $\ell=|\ii_1|$. We have
    \[
        \N_B.f = \frac{1}{\ms{b}{\ell}}\sum_{\substack{g\in\ms{[n]}{k}\\g|_{[n]\setminus B}=f|_{[n]\setminus B}}} g.
    \]
    Now consider the basis element $c_\ii x_\ii y_\ii=c_\ii x_{\ii_1} x_{\ii_2}y_{\ii_1} y_{\ii_2}$. Note that $c_\ii=\binom{k}{\ell}c_{\ii_1}c_{\ii_2}$.     
    Every $A\in\U_B$ acts trivially on $x_{\ii_2}y_{\ii_2}$, and maps $x_{\ii_1}$ to a linear combination of $x_{\jj_1}$ with $\jj_1$ a multiset of size $\ell$ of elements from $B$. Hence $P_B.(c_\ii x_\ii y_\ii)$ is a linear combination of the elements $x_\jj y_\jj$ with $\jj=\jj_1\sqcup\ii_2$ and $\jj_1$ a multiset of size $\ell$ of elements from $B$. Let us compute the coefficient of  $x_\jj y_\jj$ for some fixed $\jj_1$. Denote $\ii_1=\{\alpha_1,\ldots,\alpha_\ell\}$
    \[
        [x_\jj y_\jj]P_B.(c_\ii x_\ii y_\ii) = \binom{k}{\ell}c_{\ii_1}c_{\ii_2}\int_{A\in\U_B} \sum_{\substack{\jjj'=(j'_1,\ldots,j'_\ell),\jjj''=(j''_1,\ldots,j''_\ell)\\ \jjj'=\jjj''=\jj_1~\mathrm{as~multisets}}}  A_{j'_1,\alpha_1}\cdots A_{j'_\ell,\alpha_\ell} \overline{A_{j''_1,\alpha_1}}\cdots \overline{A_{j''_\ell,\alpha_\ell}} d\mu_B 
    \]
    (We use here the notation $\jjj'$ rather than $\jj'$ to stress these are tuples and not multisets.) Now apply Theorem \ref{thm:weingarten} to write each summand as a summation over a pair of permutations $\sigma,\tau\in \sym(\ell)$. Change the order of summation to 
    \[
        \sum_{\jjj'}\sum_{\sigma,\tau\in\sym(\ell)}\sum_{\jjj''}.
    \]
    The number of possible $\jjj'$ is $c_{\jj_1}$. No matter what $\jjj'$ is, the permutation $\tau$ needs to satisfy $\alpha_t=\alpha_{\tau(t)}$ for all $t$. The number of such $\tau$'s is precisely $\frac{\ell!}{c_{\ii_1}}$: indeed, this is the size of the stabilizer of $(\alpha_1,\ldots,\alpha_\ell)$ in the action of $\sym(\ell)$ to $\ell$-tuples, and $c_{\ii_1}$ is the number of elements in its orbit. Given $\jjj'$, any $\sigma\in\sym(\ell)$ agrees with a \textit{unique} $\jjj''$: the one equal to $\sigma.\jjj'$. So for a fixed $\jjj'$, when scanning all valid permutations $(\sigma,\tau)$ (valid in the sense that the corresponding summand in \eqref{eq:Weingarten main result} does not vanish), the values of $\theta=\tau\sigma^{-1}$ we get are all the permutations in $\sym(\ell)$, each one exactly $\frac{\ell!}{c_{\ii_1}}$ times. 
    
    We obtain that the coefficient is
    \begin{eqnarray*}
        &=& \binom{k}{\ell}c_{\ii_1}c_{\ii_2}\cdot c_{\jj_1}\frac{\ell!}{c_{\ii_1}}\sum_{\theta\in\sym(\ell)}\wg_{\ell,b}(\theta) \\
        &\stackrel{\mathrm{Lemma}~\ref{lem:sum and signed sum of weingarten}}{=}& 
        \binom{k}{\ell}c_{\ii_1}c_{\ii_2}\cdot c_{\jj_1}\frac{\ell!}{c_{\ii_1}} \frac{1}{\ell!\ms{b}{\ell}} \\
        &=& \binom{k}{\ell}c_{\jj_1}c_{\ii_2} \frac{1}{\ms{b}{\ell}} = c_\jj\frac{1}{\ms{b}{\ell}}.
    \end{eqnarray*}
    Hence 
    \[
        P_B.(c_\ii x_\ii y_\ii)=\frac{1}{\ms{b}{\ell}} \sum_{\substack{\jj\in\ms{[n]}{k}\\ \jj|_{[n]\setminus B}=\ii|_{[n]\setminus B}}} c_\jj x_\jj y_\jj,
    \]
    which is identical to the action of $\N_B$ above.
\end{proof}

\subsection{Torus-invariant subspaces control the spectral gap}

We now prove a weaker version of Theorem \ref{thm:(k,-k) dominates non-balanced parts} which suffices to yield Corollary \ref{cor:spectral gap obtained in non-balanced parts}, that the spectral gap of the $\U(n)$-spectrum of any weighted hypergraph $\Gamma$ is obtained in torus-invariant subspaces. Given $\Gamma=([n],w)$, the proof goes by considering the values\marginpar{$\scriptstyle{\phi_i(\Gamma),\phi(\Gamma)}$}
\begin{equation} \label{eq:phi_i and phi}
    \phi=\phi(\Gamma)\defi\min_{i\in[n]}\{\phi_i\}~~~~\mathrm{where}~~~~\phi_i=\phi_i(\Gamma)\defi\sum_{B\ni i}w_B.
\end{equation}
On the one hand, we show that all eigenvalues of $\Gamma$ outside the torus-invariant subspaces are at least $\phi$. On the other hand, certain torus-invariant subspaces admit eigenvalues \textit{at most} $\phi$. 

We begin with the first assertion. Recall the decomposition of any irrep $\rho\colon\U(n)\to\gl(V)$ to $\L(\Gamma,\rho)$-invariant subspaces from Corollary \ref{cor:2 to the n invariant subspaces}.

\begin{lemma} \label{lem:blocks outside tor-int have large evalues}
    Let $\Gamma=([n],w)$ be a hypergraph with non-negative weights, $\rho\colon\U(n)\to\gl(V)$ a finite-dimensional representation, and
    $V=\torinv(\rho)\oplus V'$ the decomposition to two $\L(\Gamma,\rho)$-invariant subspaces, which exists by Corollary \ref{cor:2 to the n invariant subspaces}. Then every eigenvalue of $\L(\Gamma,\rho)|_{V'}$ is at least $\phi(\Gamma)$, defined in \eqref{eq:phi_i and phi}.
    
\end{lemma}

\begin{proof}
    By Corollary \ref{cor:2 to the n invariant subspaces}, 
    \[
        V'=\bigoplus_{\substack{\varepsilon_1,\ldots,\varepsilon_n\in\{0,1\} \\ \mathrm{not~all~ones}}} V_{\varepsilon_1,\ldots,\varepsilon_n},
    \]
    where $V_{\varepsilon_1,\ldots,\varepsilon_n}\defi V^{\{1\},\varepsilon_1}\cap\ldots\cap V^{\{n\},\varepsilon_n}$, and all are subspaces invariant under $\L(\Gamma,\rho)$. Consider one of this subspaces $V_{\varepsilon_1,\ldots,\varepsilon_n}$, and let $j\in[n]$ satisfy $\varepsilon_j=0$. Recall that $V^{\{j\},0}=\ker(P_{\{j\}})$, so $V_{\varepsilon_1,\ldots,\varepsilon_n}\subseteq\ker(P_{\{j\}})$. 
    
    By Lemma \ref{lem:B_1 contained in B_2}, whenever $B\ni j$, we have $\mu_{B}=\mu_{B}*\mu_{\{j\}}$. So in the subspace $V_{\varepsilon_1,\ldots,\varepsilon_n}$ we have 
    \[
        P_B=P_B\cdot P_{\{j\}}=0,
    \] 
    and
    \[
        \L(\Gamma,\rho)=\sum_{B\subseteq[n]} w_B(I-P_B)=\left(\sum_{B\ni j} w_B\right) + \left(\sum_{B\not\ni j} w_B(I-P_B)\right)=\phi_j+\left(\sum_{B\not\ni j} w_B(I-P_B)\right).
    \]
    We are done as $\phi_j\ge\phi$ and the sum in the right hand side is over non-negative operators.    
\end{proof}

Recall that Theorem \ref{thm:(k,-k) dominates non-balanced parts} states that all the irreps of the form $\rho_{(k,0\ldots,0,-k)}$ "spectrally dominate" all invariant subspaces which are not the torus-invariant ones. We will show that in shortly, in $\S$\ref{subsec:embed KMP_k in KMP_k+1}. As a warm-up, we show that $\rho_{(1,0,\ldots,0,-1)}$ dominates these subspaces. This is already enough in order to conclude that the spectral gap of the $\U(n)$-spectrum of $\Gamma$ lies in torus-invariant subspaces. 

\begin{lemma} \label{lem:R11 dominates non-central blocks}
    For any hypergraph $\Gamma=([n],w)$ with non-negative weights,
    \[
        \lambda_{\min}\left(\Gamma,\rho_{(1,0,\ldots,0,-1)}\right) \le \phi(\Gamma).
    \]
    Moreover, this smallest eigenvalue lies in the restriction of $\L(\Gamma,\rho_{(1,0,\ldots,0,-1)})$ to the torus-invariant subspace.
\end{lemma}

\begin{proof}
    We show that $\L_\torinv(\Gamma,\rho_{(1,0,\ldots,0,-1)})$ admits an eigenvalue which is at most $\phi(\Gamma)$. Combined with Lemma \ref{lem:blocks outside tor-int have large evalues}, this would prove the statement.
    
    By Theorem \ref{thm:kmp_k is central block of...}, $$\L(\Gamma,\kmp_1)\cong\L_\torinv(\Gamma,\rho_{(1,0,\ldots,0,-1)}\oplus\triv)=\L_\torinv(\Gamma,\rho_{(1,0,\ldots,0,-1)})\oplus\L_\torinv(\Gamma,\triv).$$
    Of course, $\torinv(\triv)$ is the entire one-dimensional trivial representation, and thus the smallest eigenvalue of $\L_\torinv(\Gamma,\rho_{(1,0,\ldots,0,-1)})$ is $\lambda^*_{\min}(\Gamma,\kmp_1(\Gamma))$.
    
    Inside $\MS(n,1)$, denote by $\delta_i$ the characteristic vector of the multiset of size 1 $\{i\}\in\ms{[n]}{1}$ (so $\delta_1,\ldots,\delta_n$ are the standard basis vectors of $\MS(n,1)$). The subspace of $\MS(n,1)$ corresponding to the trivial representation is the constant vectors $\alpha\sum_{i=1}^n\delta_i$ for arbitrary $\alpha\in\C$. Hence the subspace of $\MS(n,1)$ corresponding to $\torinv(\rho_{(1,0,\ldots,0,-1)})$ is 
    \[
        \left\{ \sum_{i=1}^n\alpha_i\delta_i \middle|\sum\alpha_i=0 \right\}.
    \]
    Fix $j\in[n]$ and consider the vector $v=\left(\sum_{i\ne j}\delta_i\right)-(n-1)\delta_j\in\MS(n,1)$. Clearly, if $B\not\ni j$, then $\N_B.v=v$ and $(I-\N_B).v=0$. Hence, the Rayleigh quotient satisfies
    \[
        \frac{\left\langle \L(\Gamma,\kmp_1).v,v\right\rangle}{\langle v,v\rangle} = \sum_{B\subseteq[n]}w_B\frac{\left\langle (I-\N_B).v,v\right\rangle}{\langle v,v\rangle} = \sum_{B\ni j}w_B\frac{ \left\langle I-\N_B.v,v\right\rangle}{\langle v,v\rangle}\le \sum_{B\ni j} w_B=\phi_j(\Gamma),
    \]
    where in the inequality we relied on that the spectrum of $I-\N_B$ is contained in $\{0,1\}$ (by Lemma \ref{lem:evalues of N_B are 0 or 1}), so the Rayleigh quotient satisfies $\frac{ \left\langle I-\N_B.v,v\right\rangle}{\langle v,v\rangle}\in[0,1]$. We conclude that $$\lambda^*_{\min}(\Gamma,\kmp_1(\Gamma))\le\min_j\phi_j(\Gamma)=\phi(\Gamma).$$    
\end{proof}

Lemmas \ref{lem:blocks outside tor-int have large evalues} and \ref{lem:R11 dominates non-central blocks} yield Corollary \ref{cor:spectral gap obtained in non-balanced parts}: the spectral gap of the $\U(n)$-spectrum of $\Gamma$ is obtained in torus-invariant subspaces.

\subsection{Embedding $\kmp_k(\Gamma)$ in $\kmp_{k+1}(\Gamma)$} \label{subsec:embed KMP_k in KMP_k+1}

Next, we generalize the statement and proof of Lemma \ref{lem:R11 dominates non-central blocks} to $\rho_{(k,0,\ldots,0,-k)}$ for arbitrary $k\in\Z_{\ge1}$, thus proving Theorem \ref{thm:(k,-k) dominates non-balanced parts} in full. The analysis of the KMP process we elaborate towards this proof will be useful also in the following sections.

Recall that by Theorem \ref{thm:kmp_k is central block of...}
\[
    \L(\Gamma,\kmp_k)\cong  \bigoplus_{j=0}^k\L_\torinv\left(\Gamma,\rho_{(j,0,\ldots,0,-j)}\right).
\]
In particular, this means that $\L(\Gamma,\kmp_{k})$ is a direct summand of $\L(\Gamma,\kmp_{k+1})$. We begin by identifying, inside $\MS(n,k+1)$, the embedding of $\MS(n,k)$, as well as its orthogonal complement which is isomorphic to $\torinv(\rho_{(k+1,0,\ldots,0,-k-1)})$.

\begin{defn} \label{def:embed MS(n,k) in MS(n,k+1)}
    Define a linear map $\Psi_k \colon \MS(n,k) \to \MS(n,k+1)$\marginpar{$\Psi_k$} by
    \[
        \MS(n,k)\ni\sum_{\ii\in\ms{[n]}{{k}}}\alpha_\ii\delta_\ii ~~~~\stackrel{\Psi_k}{\longmapsto}~~~~\sum_{\jj\in\ms{[n]}{{k+1}}}\beta_\jj\delta_\jj\in\MS(n,k+1),
    \]
    where $\delta_\ii\in\MS(n,k)$\marginpar{$\delta_\ii$} is the indicator vector of $\ii\in\ms{[n]}{k}$, and for $\jj=\{j_1,\ldots,j_{k+1}\}$ we have 
    \[
        \beta_\jj=\sum_{t=1}^{k+1} \alpha_{\jj\setminus\{j_t\}}.
    \]
    Denote by $\pure(n,k)$\marginpar{$\scriptstyle{\pure(n,k)}$} the orthogonal complement of $\Psi_{k-1}(\MS(n,k-1))$ in $\MS(n,k)$.\footnote{We use the standard inner product on $\MS(n,k)$ with respect to the basis given by $\{\delta_\ii\mid\ii\in\ms{[n]}{k}\}$.}
\end{defn}
For example, if $v=\sum_{1\le a\le b\le n} v_{a,b}\delta_{\{a,b\}}\in\MS(n,2)$, then the coefficient of $\delta_{\{c,c,d\}}$ in $\Psi_2(v)\in\MS(n,3)$ is $2v_{c,d}+v_{c,c}$. Equivalently, $\Psi_k$ can be defined by 
\begin{equation} \label{eq:second definition of Psi}
    \Psi_k\left(\delta_\ii\right) = \sum_{x\in[n]}\left(\#_x(\ii)+1\right)\delta_{\ii\sqcup\{x\}},
\end{equation}
where $\#_x(\ii)$\marginpar{$\#_x(\ii)$} is the multiplicity of $x$ in $\ii$.

\begin{prop} \label{prop:embedding MS(n,k) in MS(n,k+1) properties}
    The multiset space $\MS(n,k)$ satisfies the following properties.
    \begin{enumerate}
        \item \label{enu:embedding} The map $\Psi_k\colon\MS(n,k)\to\MS(n,k+1)$ is an embedding.
        \item \label{enu:equivariant} The map $\Psi_k$ is equivariant under $\N_B$ for all $B\subseteq[n]$, namely, the following diagram commutes:
        \[
            \xymatrix{
                \MS(n,k)\ar[rr]^{\N_B}\ar[d]^{\Psi_k} && \MS(n,k)\ar[d]^{\Psi_{k}} \\
                \MS(n,k+1)\ar[rr]^{\N_B} && \MS(n,k+1) 
            }
        \]
        \item \label{enu:orth complement} The orthogonal complement $\pure(n,k)\le \MS(n,k)$ is invariant under $\L(\Gamma,\kmp_k)$ for all $\Gamma$, and $\L(\Gamma,\kmp_k)|_{\pure(n,k)}=\L_\torinv(\Gamma,\rho_{(k,0,\ldots,0,-k)})$.
        \item \label{enu:representatives for pure} For any $x\in[n]$, and any $$g_1=\sum_{\ii\in\ms{[n]\setminus\{x\}}{k}}\alpha_\ii\delta_\ii\in MS(n,k)$$
        (with arbitrary coefficients $\alpha_\ii\in\C$), there exists a unique element $$g_2=\sum_{\ii\in\ms{[n]}{k}\setminus\ms{[n]\setminus\{x\}}{k}}\in\MS(n,k)$$ such that $g_1+g_2\in\pure(n,k)$.
    \end{enumerate}    
\end{prop} 

To illustrate Item \ref{enu:representatives for pure}, consider the element $g_1=\alpha_{\{1,1\}}\delta_{\{1,1\}}+\alpha_{\{1,3\}}\delta_{\{1,3\}}+\alpha_{\{3,3\}}\delta_{\{3,3\}}\in\MS(3,2)$ (note that it avoids the element $x=2\in[3]$). It can be completed uniquely by some element $g_2=\alpha_{\{1,2\}}\delta_{\{1,2\}}+\alpha_{\{2,2\}}\delta_{\{2,2\}}+\alpha_{\{2,3\}}\delta_{\{2,3\}}\in\MS(3,2)$ so that $g_1+g_2\in\pure(3,2)$. Note that Item \ref{enu:representatives for pure} agrees with the dimension of $\pure(n,k)$, which is $\ms{n}{k}-\ms{n-1}{k}=\ms{n}{k-1}$.

\begin{proof}
    Fix $x\in[n]$. For every $\ii\in\ms{[n]}{k}$ there is a unique element in the support of $\Psi_k(\delta_\ii)$ with maximal multiplicity of $x$: the element $\delta_{\ii\sqcup\{x\}}$. Moreover, for every $\jj\in\ms{[n]}{k+1}$ with $\#_x(\jj)\ge1$, there is a unique $\ii\in\ms{[n]}{k}$ with $\jj=\ii\sqcup\{x\}$. This shows that if the basis elements $\ms{[n]}{k+1}$ are ordered so that $\#_x$ weakly decreases, then the vectors $$\left\{\Psi_k(\delta_\ii)\right\}_{\ii\in\ms{[n]}{k}}$$ are in Echelon form. This proves both Items \ref{enu:embedding} and \ref{enu:representatives for pure}.

    For Item \ref{enu:equivariant}, it is enough to prove that $\N_B(\Psi_k(\delta_\ii))=\Psi_k(\N_B(\delta_\ii))$ for every $\ii\in\ms{[n]}{k}$. Write $\ii=\ii_1\sqcup\ii_2$ with $\ii_1$ containing elements from $B$ and $\ii_2$ from $[n]\setminus B$, and denote $\ell=|\ii_1|$. Then 
    \begin{eqnarray*}
            \N_B(\Psi_k(\delta_\ii)) &=& \N_B\left( \sum_{x\in [n]}(\#_x(\ii)+1)\delta_{\ii\sqcup\{x\}} \right)\\
            &=& \left(\sum_{x\in B} (\#_x(\ii_1)+1)\frac{1}{\ms{|B|}{\ell+1}}\sum_{\jj'\in\ms{B}{\ell+1}}\delta_{\jj'\sqcup\ii_2}\right) \\
            &&+ \left(\sum_{x\in [n]\setminus B} (\#_x(\ii_2)+1)\frac{1}{\ms{|B|}{\ell}}\sum_{\jj'\in\ms{B}{\ell}}\delta_{\jj'\sqcup\ii_2\sqcup\{x\}}\right)\\
            &=& \frac{|B|+\ell}{\ms{|B|}{\ell+1}}\sum_{\jj'\in\ms{B}{\ell+1}}\delta_{\jj'\sqcup\ii_2} + \sum_{x\in [n]\setminus B} (\#_x(\ii_2)+1)\frac{1}{\ms{|B|}{\ell}}\sum_{\jj'\in\ms{B}{\ell}}\delta_{\jj'\sqcup\ii_2\sqcup\{x\}}.
    \end{eqnarray*}
    On the other hand,
    \begin{eqnarray*}
            \Psi_k(\N_B(\delta_\ii)) &=& \Psi_k\left( \frac{1}{\ms{|B|}{\ell}}\sum_{\jj'\in \ms{B}{\ell}}\delta_{\jj'\sqcup\ii_2} \right)\\
            &=& \frac{1}{\ms{|B|}{\ell}}\sum_{\jj'\in \ms{B}{\ell}}\left( \sum_{x\in B}(\#_x(\jj')+1)\delta_{\jj'\sqcup\ii_2\sqcup\{x\}} + \sum_{x\in [n]\setminus B}(\#_x(\ii_2)+1)  \delta_{\jj'\sqcup\ii_2\sqcup\{x\}} \right) \\
            &=& \frac{\ell+1}{\ms{|B|}{\ell}}\sum_{\jj''\in \ms{B}{\ell+1}} \delta_{\jj''\sqcup\ii_2} + \sum_{x\in [n]\setminus B}(\#_x(\ii_2)+1) \frac{1}{\ms{|B|}{\ell}}\sum_{\jj'\in \ms{B}{\ell}} \delta_{\jj'\sqcup\ii_2\sqcup\{x\}},
    \end{eqnarray*}
    where in the last equality we used the observation that for every $\jj''\in\ms{B}{\ell+1}$, the element $\delta_{\jj'\sqcup\ii_2}$ was obtained in the second row with total coefficient
    \[
        \frac{1}{\ms{|B|}{\ell}}\sum_{x\colon \#_x(\jj'')\ge1} \#_x\left(\jj''\setminus\{x\}\right)+1= \frac{1}{\ms{|B|}{\ell}}\sum_{x\in[n]} \#_x(\jj'') =  \frac{|\jj''|}{\ms{|B|}{\ell}}.
    \]
    As both expressions are equal, we established Item \ref{enu:equivariant}.

    Finally, for Item \ref{enu:orth complement}, assume that $v\in\pure(n,k)$, namely, that $\langle v, \Psi_{k-1}(\delta_\ii)\rangle=0$ for every $\ii\in\ms{[n]}{k-1}$, where the inner product is the standard one with respect to the basis corresponding to $\ms{[n]}{k}$. We need to show that $\N_B.v\in\pure(n,k)$. But $\N_B$ is self-adjoint (Lemma \ref{lem:evalues of N_B are 0 or 1}), so 
    \[
        \langle \N_B.v, \Psi_{k-1}(\delta_\ii)\rangle = \langle v, \N_B(\Psi_{k-1}(\delta_\ii))\rangle = \langle v, \Psi_{k-1}(\N_B(\delta_\ii))\rangle = 0.
    \]
    Hence $\pure(n,k)$ is invariant under $\N_B$ for all $B\subseteq[n]$ and thus also under $\L(\Gamma,\kmp_k)$ for every weighted hypergraph $\Gamma$. Being the orthogonal complement of $$\Psi_{k-1}(\MS(n,k-1))=\torinv\left(\bigoplus_{j=0}^{k-1}\rho_{(j,0,\ldots,0,-j)}\right),$$ it must coincide with $\torinv(\rho_{(k,0,\ldots,0,-k)})$.    
\end{proof}

We cannot complete the proof of Theorem \ref{thm:(k,-k) dominates non-balanced parts}. 

\begin{proof}[Proof of Theorem \ref{thm:(k,-k) dominates non-balanced parts}]
    As in the statement of the theorem, denote $\rho_k=\rho_{(k,0,\ldots,0,-k)}$. Given Lemma \ref{lem:blocks outside tor-int have large evalues}, it is enough to show that $\lambda_{\min}(\Gamma,\rho_k)\le\phi(\Gamma)$. By Proposition \ref{prop:embedding MS(n,k) in MS(n,k+1) properties}\eqref{enu:orth complement}, $\L_\torinv(\Gamma,\rho_k)=\L(\Gamma,\kmp_k)|_{\pure(n,k)}$. 
    
    We imitate the proof of Lemma \ref{lem:R11 dominates non-central blocks}. Fix $x\in[n]$ and consider the vector $v_x\in\pure(n,k)$ with coefficient 1 for every $\delta_\ii$ with $\ii\in\ms{[n]\setminus\{x\}}{k}$: the existence (and uniqueness) of $v_x$ is guaranteed by Proposition \ref{prop:embedding MS(n,k) in MS(n,k+1) properties}\eqref{enu:representatives for pure}. Clearly, if $B\not\ni x$, then $\N_B.v_x=v_x$ and $(I-\N_B).v_x=0$. Hence, 
    \[
        \frac{\left\langle \L(\Gamma,\kmp_k).v_x,v_x\right\rangle}{\langle v_x,v_x\rangle} = \sum_{B\subseteq[n]}w_B\frac{\left\langle (I-\N_B).v_x,v_x\right\rangle}{\langle v_x,v_x\rangle} = \sum_{B\ni x}w_B\frac{ \left\langle I-\N_B.v_x,v_x\right\rangle}{\langle v_x,v_x\rangle}\le \sum_{B\ni x} w_B=\phi_x(\Gamma),
    \]
    where in the inequality we relied again on Lemma \ref{lem:evalues of N_B are 0 or 1}. We conclude that $\lambda_{\min}(\Gamma,\rho_k)\le\min_x\phi_x(\Gamma)=\phi(\Gamma)$.  
\end{proof}

To end this section, we remark that the $n$-dimensional standard representation $\rho_\std=\rho_{(1,0,\ldots,0)}$ of $\U(n)$ (defined by $A\mapsto A$), satisfies that the $\L(\Gamma,\std)$ is diagonal with diagonal entries $\phi_1(\Gamma),\ldots,\phi_n(\Gamma)$. Indeed, this is an immediate application of Lemma \ref{lem:unbalanced integrals vanish}. Hence $\lambda_{\min}(\Gamma\rho_\std)=\phi(\Gamma)$. We obtain the following slight strengthening of Theorem \ref{thm:(k,-k) dominates non-balanced parts}.

\begin{cor}
    Let $\Gamma=([n],w)$ be a hypergraph with non-negative weights, and let $\alpha$ be an eigenvalue in the $\U(n)$-spectrum of $\Gamma$ associated with an eigenvector lying \textit{outside} the torus-invariant subspace. Then for all $k\ge1$
    \[
        \lambda_{\min}(\Gamma,\rho_{(k,0,\ldots,o,-k)}) \le \lambda_{\min}(\Gamma,\rho_\std) \le \alpha.
    \]
\end{cor}

\section{Hypergraphs supported on sets of size $n-1$} \label{sec:n-1}

In this Section we prove Theorem \ref{thm:n-1s}, which states the special case of Conjecture \ref{conj:main conj} when the weights of the hypergraph $\Gamma=([n],w)$ are supported on subsets of size $\ge n-1$. Note that we may assume the weights are supported on subsets of size precisely $n-1$: Indeed, for every $d$-dimensional representation $\rho$ of $\U(n)$, $\L(\Gamma,\rho)=\sum_{B\subseteq[n]}w_B(I_d-P_B)$ with $P_B$ the projection onto the $\U_B$-invariant subspace. When $B=[n]$, $P_{[n]}$ is the projection onto the trivial component of $\rho$, so $w_{[n]}(I_d-P_{[n]})$ only shifts the spectrum of \textit{all} non-trivial irreps by $w_{[n]}$. Hence, we may ignore this subset and assume without loss of generality that $w_{[n]}=0$. 

In addition, if $n\le2$, subsets of size $n-1$ are trivial: If $n=1$, everything is trivial. When $n=2$, subsets of size $n-1$ are singletons, but $P_{\{x\}}$ acts as the identity on all torus-invariant subspaces, so $\lambda^*_{\min}(\Gamma,\kmp_k)=0$ for all $k$ and Theorem \ref{thm:n-1s} is trivially true. \textbf{Hence, 
throughout the rest of this section, we assume that $n\ge3$} (unless stated explicitly otherwise) and that \textbf{$\Gamma$ is supported on subsets of size $n-1$.} We denote the weight $w_{[n]\setminus\{x\}}$ of $[n]\setminus\{x\}$ by $c_x$\marginpar{$c_x$}.

\begin{lemma} \label{lem:n-1s enough to consider k,0,...,0,l}
    When the weights of $\Gamma$ are supported on subsets of size $n-1$, the spectral gap of the $\U(n)$-spectrum of $\Gamma$ is obtained in the irreps of the form $\rho_{(k,0,\ldots,0,-m)}$, corresponding to highest weight vectors with at most one positive weight and at most one negative weight.
\end{lemma}

\begin{proof}
    By Lemma \ref{lem:L of reps - basic properties of spectrum}, the $\U(n)$-spectrum of $\Gamma$ is contained in $[0,\sum w_B]$. Let $\rho\colon\U(n)\to\gl(V)$ be an irrep of $\U(n)$. By Lemma \ref{lem:L as sum of projections}, $\L(\Gamma,\rho)=\sum_{B\in[n]}w_B(I-P_B)$ with $P_B\in\End(V)$ the orthogonal projection onto the $\U_B$-invariant subspace. Namely, $P_B$ is the projection onto the subspace of $V$ corresponding to the trivial isotypic component of $\mathrm{Res}^{\U(n)}_{\U_B}(\rho)$. If $|B|=n-1$, the branching rule (Theorem \ref{thm:branching rule}) yields that $P_B=0$ unless $\rho$ is of the form $\rho_{(k,0,\ldots,0,-m)}$ (these are the only non-increasing vectors of length $n$ in $\Z^n$ which interlace the vector of $n-1$ zeros, which corresponds to the trivial irrep of $\U(n-1)$). Hence, for $\Gamma$ as in the statement of the lemma, for any irrep $\rho$ not of the form $\rho_{(k,0,\ldots,0,-m)}$, we have $\L(\Gamma,\rho)=(\sum w_B)I$, and all its eigenvalues are the largest possible.
\end{proof}

\begin{remark}
    The argument of Lemma \ref{lem:n-1s enough to consider k,0,...,0,l} applies to the same type of hypergraphs measures for any sequence of groups with a well-understood branching rule. In particular, for $\sym(n)$, the only non-trivial irrep which has a non-zero trivial component when restricted to $\sym(n-1)$ is the standard irrep $\pi_{(n-1,1)}$. Hence Caputo's Conjecture $\ref{conj:Caputo}$ is immediate in the case the hypergraph is supported on subsets of size $\ge n-1$. (This was observed already in \cite[$\S$7]{alon2025aldous-caputo}.)
\end{remark}

\begin{cor} \label{cor:n-1s enough to consider KMP}
    Let $\Gamma=([n],w)$ be a hypergraph with non-negative weights supported on subsets of size $n-1$. Then the spectral gap of the $\U(n)$-spectrum of $\Gamma$ is obtained in the KMP processes $\L(\Gamma,\kmp_k)$. 
\end{cor}
\begin{proof}
    This follows from Lemma \ref{lem:n-1s enough to consider k,0,...,0,l}, together with Theorems \ref{thm:(k,-k) dominates non-balanced parts} and \ref{thm:kmp_k is central block of...}.
\end{proof}

So we have restricted the statement of Theorem \ref{thm:n-1s} to the statement of Conjecture \ref{conj:pure KMP} for hypergraphs supported on subsets of size $n-1$. Namely, we ought to show that for every $k\in\mathbb{Z}_{\ge2}$ we have
\[
    \lambda_{\min}^*(\Gamma,\kmp_k)=\lambda_{\min}^*(\Gamma,\kmp_2).
\]

The branching rule yields more: it shows that if $x\in[n]$ and $B=[n]\setminus\{x\}$, then the $\U_B$-invariant subspace of $\rho_{(k,0,\ldots,0,-k)}$ is one-dimensional (the trivial irrep, as all other irreps in the decomposition of $\mathrm{Res}^{\U(n)}_{U_B}\rho_{(k,0,\ldots,0,-k)}$, has multiplicity one). But Proposition \ref{prop:embedding MS(n,k) in MS(n,k+1) properties}\eqref{enu:representatives for pure} identifies this one-dimensional space as the span of \marginpar{$g_{k,x}$}$g_{k,x}\in\pure(n,k)=\torinv(\rho_{(k,0,\ldots,0,-k)})$ which is the unique vector in $\pure(n,k)$ with coefficient $1$ for $\delta_\ii$ for every $\ii\in\ms{[n]\setminus\{x\}}{k}$. So $\N_{[n]\setminus\{x\}}$ is the projection onto the one-dimensional subspace $\C g_{k,x}\le\pure(n,k)$. We obtain the following corollary.

\begin{cor} \label{cor:n-1 non-zero on span(g1,...,gn)}
    Let\marginpar{$W_k$} $$W_k\defi\mathrm{Span}_\C \{g_{k,1},\ldots, g_{k,n} \}.$$ Then $W_k$ is an invariant subspace of $\L(\Gamma,\kmp_k)$ and the smallest eigenvalue of $\L(\Gamma,\rho_{(k,0,\ldots,0,-k)})$ is obtained there. The orthogonal complement $W_k^\perp$ in the representation $\rho_{(k,0,\ldots,0,-k)}$ is also an invariant subspace, on which $\L(\Gamma,\rho_{(k,0,\ldots,0,-k)})$ acts as the constant $\sum w_B$.
\end{cor}

\begin{lemma} \label{lem:values of g_x}
    We have $$g_{k,x}=\sum_{\ii\in\ms{[n]}{k}}(-1)^{\#_x(\ii)}\binom{n+k-2}{\#_x(\ii)}\delta_\ii.$$
\end{lemma}

\begin{proof}
    Denote the coefficient of $\delta_\ii$ in $g_{k,x}$ by $\alpha_\ii$. We prove that $\alpha_\ii=(-1)^{\#_x(\ii)}\binom{n+k-2}{\#_x(\ii)}$ by induction on $\#_x(\ii)$. When $\#_x(\ii)=0$, $\alpha_\ii=1$ by definition. Now let $\ii\in \ms{[n]}{k}$ satisfy $\#_x(\ii)\ge1$ and assume the claim is true for any $\ii'$ with $\#_x(\ii')<\#_x(\ii)$. Denote $\jj=\ii\setminus\{x\}$. By definition, $g_{k,x}\in\pure(n,k)$, so it must satisfy $\langle \Psi_{k-1}(\delta_{\jj}),g_{k,x}\rangle=0$. By \eqref{eq:second definition of Psi},
    \[
        0=\langle \Psi_{k-1}(\delta_{\jj}),g_{k,x}\rangle = \sum_{y\in[n]}(\#_y(\jj)+1)\alpha_{\jj\sqcup\{y\}}.
    \]
    Using the induction hypothesis, we get
    \begin{eqnarray*}
        \alpha_\ii &=& -\frac{1}{\#_x(\ii)} \sum_{y\in[n]\setminus\{x\}} (\#_y(I)+1)\alpha_{\jj\sqcup\{y\}} \\
        &=& \frac{-1}{\#_x(\ii)}\left(n+k-\#_x(\ii)-1\right)\cdot (-1)^{\#_x(\ii)-1}\binom{n+k-2}{\#_x(\ii)-1} \\
        &=& (-1)^{\#_x(\ii)}\binom{n+k-2}{\#_x(\ii)}.
    \end{eqnarray*}
\end{proof}

\begin{lemma} \label{lem:N_without x on g_y}
    Let $x,y\in[n]$. Then,
    \[
        \N_{[n]\setminus\{x\}}(g_{k,y}) = \begin{cases}
            g_{k,x} & \mathrm{if~}y=x \\
            \frac{(-1)^k}{\binom{n+k-2}{k}}g_{k,x} & \mathrm{if~}y\ne x.
        \end{cases} 
    \]
\end{lemma}

\begin{proof}
    As mentioned above, $\N_{[n]\setminus\{x\}}$ is the projection onto $\C g_{k,x}$, hence $\N_{[n]\setminus\{x\}}(g_{k,x})=g_{k,x}$ and $\N_{[n]\setminus\{x\}}(g_{k,y})=\alpha\cdot g_{k,x}$ for some $\alpha\in\R$ when $y\ne x$ (by symmetry, $\alpha$ is independent of $y$ as long as $y\ne x$). To compute the value of $\alpha$, it is enough to consider the coefficient in $\N_{[n]\setminus\{x\}}(g_{k,y})$ of $\delta_\ii$ for one arbitrary multiset $\ii$ and compare it to its coefficient in $g_{k,x}$. Let $\ii=\{x,\ldots,x\}\in \ms{[n]}{k}$. As $\N_{[n]\setminus\{x\}}$ fixes the coefficient of $\delta_\ii$, the coefficient of $\delta_\ii$ in $\N_{[n]\setminus\{x\}}(g_{k,y})$ is the same as its coefficient in $g_{k,y}$, which is $1$. By Lemma \ref{lem:values of g_x}, the coefficient of $\delta_\ii$ in $g_{k,x}$ is $(-1)^k\binom{n+k-2}{k}$. The value of $\alpha$ is the quotient of the two, which is precisely the value in the statemnet of the lemma.
\end{proof}

\begin{lemma} \label{lem:dim of W_k}
    For any $k\ge2$, $\dim(W_k)=n$. In contrast, $\dim(W_1)=n-1$, with the only linear dependence of the generators (up to multiplication by a scalar) being $g_{1,1}+\ldots+g_{1,n}=0$.
\end{lemma}
\begin{proof}
    Assume that $\sum_{i=1}^n\alpha_ig_{k,i}=0$. Without loss of generality, assume that $|\alpha_n|\ge|\alpha_i|$ for all $i\in[n-1]$. Acting by $\N_{[n-1]}$ we get
    \begin{equation} \label{eq:equation on alphas}
        0=\left[(-1)^k\frac{\alpha_1+\ldots+\alpha_{n-1}}{\binom{n+k-2}{k}}+\alpha_n\right]g_{k,n},
    \end{equation}
    so 
    \begin{equation} \label{eq:bound for alpha_n}
        |\alpha_n|\le\frac{1}{\binom{n+k-2}{k}}\sum_{i=1}^{n-1}|\alpha_i|.
    \end{equation}
    Recall our assumption throughout this section that $n\ge3$. Whenever $k\ge2$, we have $\binom{n+k-2}{k}> n-1$. Hence \eqref{eq:bound for alpha_n} is only possible if either $\alpha_i=0$ for all $i\in[n]$, or $k=1$. In the latter case, \eqref{eq:equation on alphas} shows that $\alpha_n$ is the mean of the others. But the same can be equally shown to hold for $\alpha_i$ for all $i$. Hence, the only linear dependence is precisely when the $\alpha_i$'s are all equal. This proves the claim.
\end{proof}

Recall that $J_n$ denotes the all-one $n\times n$ matrix and that $c_x=w_{[n]\setminus\{x\}}$.

\begin{prop} \label{prop:matrix for evalues in n-1s}
    Let $n\ge3$ and let $\Gamma=([n],w)$ be a hypergraph with non-negative weights supported on subsets of size $n-1$. For every $k\in\Z_{\ge1}$ consider the $n\times n$ matrix\marginpar{$M_k$}
    \[
        M_k\defi \left(\sum_{x\in[n]}c_x\right)I_n-\begin{pmatrix}
            c_1& &&\\
            &c_2&&\\
            &&\ddots&\\
            &&&c_n
        \end{pmatrix} 
        \left(I_n+\frac{(-1)^k}{\binom{n+k-2}{k}}\left(J_n-I_n\right)\right).
    \]
    Then the spectrum of $\L(\Gamma,\rho_{(k,0,\ldots,0,-k)})$ consists of the eigenvalues of $M_k$ together with copies of $\sum c_x$.
\end{prop}

\begin{proof}
    By Corollary \ref{cor:n-1 non-zero on span(g1,...,gn)}, the eigenvalues of $\L(\Gamma,\rho_{(k,0,\ldots,0,-k)})$ restricted to the invariant subspace $W_k^\perp$ are all $\sum_{x\in[n]}c_x$. By Lemma \ref{lem:dim of W_k}, when $k\ge2$, the invariant subspace $W_k=\mathrm{Span}_\C\{g_{k,1},\ldots,g_{k,n}\}$ has dimension $n$, so $g_{k,1},\ldots,g_{k,n}$ is a basis, and the claim follows as $M_k$ precisely describes the linear action of $\L(\Gamma,\rho_{(k,0,\ldots,0,-k)})$ in this basis, using Lemma \ref{lem:N_without x on g_y}.

    Finally, let $k=1$. By Lemma \ref{lem:dim of W_k}, the dimension of $W_1$ is $n-1$ and the linear dependencies of $g_{1,1},\ldots,g_{1,n}$ are precisely the span of their sum. Let $\mathbbm{1}$ denote the all-ones vector. Any eigenvector of $M_1$ which is not in $\C\mathbbm{1}$ represents a non-zero eigenvector of $\L(\Gamma,\rho_{(1,0,\ldots,0,-1)})$ in $W_1$. And the vector $\mathbbm{1}$ itself is an eigenvector of $M_1$ with eigenvalue $\sum c_x$: Indeed, this follows from $J_n\mathbbm{1}=n\mathbbm{1}$ and 
    \[
        \left(I_n+\frac{(-1)^1}{\binom{n+1-2}{1}}\left(J_n-I_n\right)\right)\mathbbm{1}=0.
    \]    
\end{proof}

Recall that our goal is to show that the smallest eigenvalue in the $\U(n)$-spectrum of $\Gamma$ is obtained in one of $\L(\Gamma,\rho_{(1,0,\ldots,0,-1)})$ or $\L(\Gamma,\rho_{(2,0,\ldots,0,-2)})$. By Corollary \ref{cor:n-1s enough to consider KMP} and Proposition \ref{prop:matrix for evalues in n-1s}, this translates to showing that for any $k\ge3$,
\[
    \min\left\{\lambda_{\min}(M_1),\lambda_{\min}(M_2)\right\} \le \lambda_{\min}(M_k).
\]
What we really show is that 
\[
    \lambda_{\min}(M_1)\le \lambda_{\min}(M_3) \le \lambda_{\min}(M_5)\le \lambda_{\min}(M_7)\le\ldots
\]
and
\[
    \lambda_{\min}(M_2)\le \lambda_{\min}(M_4) \le \lambda_{\min}(M_6)\le \lambda_{\min}(M_8)\le\ldots
\]
For $t\in\R$, write\marginpar{$A_t$}
\[
    A_t\defi A_t(c_1,\ldots,c_n)\defi\left(\sum_{x\in[n]}c_x\right)I_n-\mathrm{diag}(c_1,\ldots,c_n)\cdot
    \left(I_n+t\left(J_n-I_n\right)\right),
\]
with $\mathrm{diag}(c_1,\ldots,c_n)$ being the diagonal $n\times n$ matrix with $c_1,\ldots,c_n$ in the diagonal. Note that $M_k=A_{t_k}$ with $t_k=\frac{(-1)^k}{\binom{n+k-2}{k}}$\marginpar{$t_k$}. 
Recall our assumption that $n\ge3$, and notice that 
\[
    \frac{-1}{n-1}=t_1<t_3<t_5<\ldots <0<\ldots<t_6<t_4<t_2=\frac{2}{n(n-1)}<1.
\]
Is thus enough to show that as $t\in\R_{<1}$ approaches $0$ from either side, the smallest eigenvalue of $A_t$ (weakly) increases.\footnote{The particular interval $\R_{<1}$ is the one our proof below works for. It is not a priori clear that $A_t$ is real-rooted throughout this interval, but we prove this in Lemma 
\ref{lem:interlacing and real-rootedness} below.}

Before we prove this result for general non-negative $c_1,\ldots,c_n$, we prove the much easier mean-field case, where all the weights are equal. This special case of mean-field with subsets of size $n-1$ only is the basis for the proof of Theorem \ref{thm:mean-field} in $\S$\ref{sec:mean-field}.

\begin{cor} \label{cor:smallest-evalue for n-1 mean-field}
    Let $n\ge2$ and $\Gamma=([n],w)$ with $w_B=1$ if $|B|=n-1$ and $w_B=0$ otherwise. Then the smallest eigenvalue of the $\U(n)$-spectrum of $\Gamma$ is obtained in $\rho_{(2,0,\ldots,0,-2)}$ and is equal to $\frac{(n+1)(n-2)}{n}$.    
\end{cor}

\begin{proof}
    This is the case $c_1=\ldots=c_n=1$, and $$A_t=A_t(1,\ldots,1)=nI_n-(I_n+t(J_n-I_n))=(n-1+t)I_n-tJ_n.$$ The eigenvalues of $J_n$ are $n$ (once) and zero ($n-1$ times), so the eigenvalues of $A_t$ are $(n-1)(1-t)$ and $n-1+t$. 

    If $t>0$, then $(n-1)(1-t)<n-1<n-1+t$, so $\lambda_{\min}(A_t)=(n-1)(1-t)$, which clearly grows as $t$ approaches zero. If $t<0$, then $n-1+t<n-1<(n-1)(1-t)$, so $\lambda_{\min}(A_t)=n-1+t$, which clearly grows, too, as $t<0$ approaches zero. This shows that the smallest eigenvalue of the $\U(n)$-spectrum of $\Gamma$ is obtained in $M_1=A_{t_1}$ or in $M_2=A_{t_2}$.

    Finally, $$\lambda_{\min}(M_1)=n-1+t_1=n-1-\frac{1}{n-1}=\frac{n(n-2)}{n-1},$$ and $$\lambda_{\min}(M_2)=(n-1)(1-t_2)=(n-1)\left(1-\frac{1}{\binom{n}{2}}\right)=\frac{(n+1)(n-2)}{n},$$ so $\lambda_{\min}(M_2)\le\lambda_{\min}(M_1)$ (with strict inequality if $n\ge3$).
\end{proof}
 
\subsection{The proof of Theorem \ref{thm:n-1s}} \label{subsec:proof of n-1s general case}

Let $c_1=w_{[n]\setminus\{1\}},\ldots,c_n=w_{[n]\setminus\{n\}}\in\R_{\ge0}$ be non-negative numbers. By symmetry, we may assume without lost of generality that $c_1\ge c_2 \ge\ldots\ge c_n\ge0$. If $c_2=0$, namely, if there is at most one positive weight, the hypergraph is disconnected, and its $\U(n)$-spectrum contains the eigenvalue 0 with infinite multiplicity, including in $\L(\Gamma,\rho_{(k,0,\ldots,0,-k)})$ for all $k\ge0$. Theorem \ref{thm:n-1s} is thus trivial in this case. We assume from now on that $c_2>0$.

Instead of analyzing the matrices $A_t=A_t(c_1,\ldots,c_n)$ defined above, it is slightly more convenient to analyze the matrices\marginpar{$D_t$}
\[
    D_t\defi D_t(c_1,\ldots,c_n)\defi\left(\sum_{x\in[n]}c_x\right)I_n-A_t(c_1,\ldots,c_n)=\mathrm{diag}(c_1,\ldots,c_n)\cdot
    \left(I_n+t\left(J_n-I_n\right)\right).
\]
As explained above, it is enough to show that as $t\in\R_{<1}$ approaches $0$ from either side, the smallest eigenvalue of $A_t$ (strictly) increases or, equivalently, the largest eigenvalue of $D_t$ (strictly) decreases.\footnote{When restricting to the non-degenerate case where $c_1\ge c_2>0$, we can show strict increasing/decreasing.} In particular, we ought to show that these matrices have real eigenvalues, which is not apriori given when $t$ is not of the form $t_k$.

Denote by $P(x,t)=\det(xI-D_t) \in \R[x,t]$\marginpar{$P(x,t)$} the characteristic polynomial of $D_t$. Also denote $P(\under,t)\in\R[x]$ and $P(x,\under)\in\R[t]$ the resulting univariate polynomial when $t$ (repsectively, $x$) is fixed. 
The proof is along the following steps. 
\begin{enumerate}
    \item We find a real-rooted polynomial which interlaces $P(x,t)$, concluding that $P(\under,t)$ is real-rooted for all $t<1$.
    \item We prove the theorem when $c_1>c_2>\ldots>c_n>0$.
    \item We deduce the general case.
\end{enumerate}

\subsubsection*{Step I: The polynomial $Q(x,t)$, interlacing and real-rootedness}

Denote by $\overline{c}$ the column-vector with entries $(c_1,\ldots,c_n)$, and by $\mathbbm{1}$ the all-ones column vector of length $n$. Note that 
\[
    D_t = (1-t)\mathrm{diag}(c_1,\ldots,c_n) + t\overline{c}\cdot \mathbbm{1}^T.
\]
Hence, $D_t$ differs from $(1-t)\mathrm{diag}(c_1,\ldots,c_n)$ by the rank-1 matrix $t\overline{c}\cdot \mathbbm{1}^T$. Let $Q(x,t)$ be the characteristic polynomial of $(1-t)\mathrm{diag}(c_1,\ldots,c_n)$, namely,\marginpar{$Q(x,t)$}
\[
    Q(x,t)=\det\left(xI-(1-t)\mathrm{diag}(c_1,\ldots,c_n)\right)=\prod_{i=1}^n\left(x-(1-t)c_i\right).
\]

\begin{lemma} \label{lem:two formulas}
    We have the following two formulas for $P$ in terms of $Q$:
    \begin{enumerate}
        \item \label{enu:formula with dt} $P = Q - t\frac{\partial Q}{\partial t}$.
        \item \label{enu:formula with dx} For $t\in\R_{<1}$, $P = \left(1+\frac{nt}{1-t}\right)Q - \frac{t}{1-t}x\frac{\partial Q}{\partial x}$.
    \end{enumerate}
\end{lemma}

\begin{proof}
    \eqref{enu:formula with dt} The matrix determinant lemma (e.g., \cite{Harville2008MatrixAlgebra}) states that whenever $A$ is an $n\times n$ invertible matrix and $u$ and $v$ column vectors, then $\det(A+uv^T)=\det(A)(1+v^TA^{-1}u)$. For any fixed $t_0$, for all but finitely many values of $x$, the matrix $xI-(1-t_0)\mathrm{diag}(c_1,\ldots,c_n)$ is invertible, and then 
    \begin{eqnarray*}
        P(x,t_0)&=&\det(xI-D_{t_0})=\det\left(xI-(1-t_0)\mathrm{diag}(c_1,\ldots,c_n)-t_0\overline{c}\cdot \mathbbm{1}^T\right)\\
        &=& Q(x,t_0)\left(1-t_0\mathbbm{1}^T\cdot\mathrm{diag}\left(\frac{1}{x-(1-t_0)c_1},\ldots,\frac{1}{x-(1-t_0)c_n}\right)\overline{c}\right) \\
        &=& Q(x,t_0)\left(1-\sum_i \frac{t_0c_i}{x-(1-t_0)c_i}\right) = Q(x,t_0) - t_0\frac{\partial Q}{\partial t}(x,t_0).
    \end{eqnarray*}
    But if the two polynomials are equal for almost every $x$, they must be equal for every $x$.

    \eqref{enu:formula with dx} Because $Q(x,t)$ is homogeneous in $x$ and $(1-t)$, it is easily verified that $nQ(x,t)=x\frac{\partial Q}{\partial x} - (1-t)\frac{\partial Q}{\partial t}$. Hence, 
    $$ \frac{\partial Q}{\partial t} = \frac{x}{1-t}\frac{\partial Q}{\partial x}-\frac{n}{1-t}Q.$$
    Plugging this expression in the identity $P=Q-t\frac{\partial Q}{\partial t}$ yields the second identity.
\end{proof}

Let $f,g\in\R[x]$ be two real-rooted polynomials of degree $n$ with roots $\alpha_n\le\ldots\le\alpha_1$ (of $f$) and $\beta_n\le\ldots\le\beta_1$ (of $g$). We say that $f$ and $g$ \textbf{interlace} with the roots of $f$ larger if 
\[
    \beta_n\le \alpha_n \le \ldots\le \beta_2\le\alpha_2\le \beta_1\le \alpha_1.
\]

\begin{lemma} \label{lem:interlacing and real-rootedness}
    Fix $t_0\in\R_{<1}$. Then $P(\under,t_0)$ is a monic, real-rooted degree-$n$ polynomial which interlaces $Q(\under,t_0)$. If $t_0\in(0,1)$, the roots of $P(\under,t_0)$ are larger. If $t_0<0$, the roots of $P(\under,t_0)$ are smaller.
\end{lemma}

\begin{proof}
    It is clear from the definition of $P$ that $P(\under,t_0)$ is monic of degree $n$. 
    The polynomial $Q(\under,t_0)$ is clearly monic degree-$n$ and real-rooted, with roots $0\le c_n(1-t_0)\le \ldots\le c_2(1-t_0)\le c_1(1-t_0)$. When $t_0=0$ the claim is obvious as $P(\under,0)=Q(\under,0)$ so assume $t_0\ne0$.    
    It is enough to prove the claim when $c_1>c_2>\ldots>c_n>0$: indeed, the map from the coefficients of a monic, degree-$n$ polynomial in $\C[x]$ to the unordered multiset of its $n$ roots is continuous (e.g., \cite[Thm.~V.4A]{whitney1972complex}). So assume that $c_1>c_2>\ldots>c_n>0$.
    
    The derivative $\frac{\partial}{\partial x}Q(\under,t_0)$ is positive at the largest root $c_1(1-t_0)$, negative at the second one $c_2(1-t_0)$ and so on alternately. Using Lemma \ref{lem:two formulas}\eqref{enu:formula with dx}, we conclude that 
    \[
        P\left(c_1(1-t_0),t_0\right)= - \frac{t_0}{1-t_0}\cdot c_1(1-t_0)\cdot\frac{\partial Q}{\partial x}(c_1(1-t_0),t_0)
        \begin{cases}
            <0 & \mathrm{if~} t_0\in(0,1) \\
            >0 & \mathrm{if~} t_0<0 
        \end{cases}.
    \]
    Similarly, 
    \[
        P\left(c_2(1-t_0),t_0\right)
        \begin{cases}
            >0 & \mathrm{if~} t_0\in(0,1) \\
            <0 & \mathrm{if~} t_0<0 
        \end{cases},
    \]
    and so on alternately. For any $t_0<1$ we obtain that $P(\under,t_0)$ must cross the real line between any two consecutive roots of $Q(\under,t_0)$. If $t_0\in(0,1)$, $P(\under,t_0)$ also crosses the real line once to the right of the largest root $c_1(1-t_0)$. If $t<0$, it must cross the real line once more to the left of the smallest root $c_n(1-t_0)$. This proves that $P(\under,t_0)$ is real-rooted and interlaces $Q(\under,t_0)$ as stated.
\end{proof}

\subsubsection*{Step II: Proving the theorem when $c_1>c_2>\ldots>c_n>0$}

\begin{lemma} \label{lem:largest root at least c_1 then h(t) monotone}
Assume that $c_1>c_2>\ldots>c_n>0$. For any $0\ne t<1$, denote by $h(t)$\marginpar{$h(t)$} the largest eigenvalue of $D_t$, namely, the largest root of $P(\under,t)$. Then the set $\{t\in(0,1)\,\mid\,h(t)>c_1\}$ is an open interval of the form $(t_c,1)$  and $h(t)$ is strictly increasing in it.\footnote{As for now, this interval is possibly empty. We will shortly show it is the entire $(0,1)$.} Likewise, the open set $\{t<0\,\mid\,h(t)>c_1\}$ is an open ray of the form $(-\infty,t_c)$, and $h(t)$ is strictly decreasing in it.\footnote{Again, the ray is possibly empty, but we will shortly show it is the entire ray $(-\infty,0)$.}
\end{lemma}
\begin{proof}
    Assume that $c_1>\ldots >c_n>0$. We begin with a claim about real-rootedness and interlacing of the polynomial $P(x_0,\under)\in\R[t]$ whenever $x_0>c_1$. So fix $x_0>c_1$. The polynomial  $$Q(x_0,\under)=\prod_i\left(c_it-(c_i-x_0)\right)$$ is degree-$n$ with positive leading coefficient $\prod c_i$. Its roots are
    \[
        1-\frac{x_0}{c_n}<\ldots<1-\frac{x_0}{c_2}<1-\frac{x_0}{c_1}<0.
    \]
    The derivative $\frac{\partial Q}{\partial t}(x_0,\under)$ is positive at the largest root $1-\frac{x_0}{c_1}$, negative at the second root and so on alternately. The formula $P = Q - t\frac{\partial Q}{\partial t}$ from Lemma \ref{lem:two formulas}\eqref{enu:formula with dt}, yields that $P(x_0,\under)$ is of degree $n$ in $t$ with \textit{negative} leading coefficient $(1-n)\prod c_i<0$. Note that $$P\left(x_0,1-\frac{x_0}{c_1}\right)=Q\left(x_0,1-\frac{x_0}{c_1}\right)-\left(1-\frac{x_0}{c_1}\right)\frac{\partial Q}{\partial t}\left(x_0,1-\frac{x_0}{c_1}\right)>0,$$ so $P(x_0,\under)$ is positive at the largest root $1-\frac{x_0}{c_1}$. Similarly, it is negative at the second root $1-\frac{x_0}{c_2}$, and so on alternately. Hence $P(x_0,\under)$ must have a root between any two consecutive roots of $Q(x_0,\under)$, and an additional root to the right of the largest root $1-\frac{x_0}{c_1}$. Hence, it is real-rooted and interlaces $Q(x_0,\under)$ with the roots of $P(x_0,\under)$ larger.
    \medskip

    We now proceed to proving the statement of the lemma. Consider first the set $T=\{t\in(0,1)\,\mid\,h(t)>c_1\}$. It is enough to prove the following local property: for every $t_0\in T$ there is an open neighborhood $U$ such that for every $t_0\ne s\in U$, if $s<t_0$ then $h(s)<h(t_0)$ and if $t_0<s$ then $h(t_0)<h(s)$. Indeed, by the continuity of $h$, if $T$ is not a single (possibly empty) interval of the form $(t_c,1)$ or if $h$ is not strictly increasing in $T$, then there must exist $a,b\in T$ with $a<b$ and $h(a)\ge h(b)$. Let $S=\{x\in[a,b] \,\mid\, h(x)>h(a)\}$. By the local property at $a$, $S\ne\emptyset$. Let $s=\sup(S)$. Clearly, $a<s$ and $h(s)=h(a)$, but the local property fails at $s$, a contradiction.

    So let $t_0\in T$, and set $x_0=h(t_0)>c_1$. By the continuity of $h$, there is an open neighborhood $U\subseteq T$ of $t_0$ in which $h(t)>c_1$. By definition, $P(x_0,t_0)=P(h(t_0),t_0)=0$, namely, $t_0$ is a root of $P(x_0,\under)$. By the analysis in the first part of this proof, as $0<t_0$, $t_0$ must be the largest root of $P(x_0,\under)$, hence $P(x_0,\under)$ is \textit{decreasing} at $t_0$. Moreover, the second largest root of $P(x_0,\under)$ is negative. Hence, for every $s\in U\subseteq(0,1)$, if $s<t_0$ then  $P(x_0,s)>0$, and if $t_0<s$ then $P(x_0,s)<0$.
    
    If $s<t_0$, by Lemma \ref{lem:interlacing and real-rootedness} and the fact that $(1-s)c_1<c_1<x_0$, $x_0$ lies to the right of the second root of $P(\under,s)$. But $P(\under,s)$ is negative between the two largest roots, so $x_0$ lies to the right of the largest root. We conclude that $h(s)<x_0=h(t_0)$. Similarly, if $t_0<s$, then $P(x_0,s)<0$ means that $x_0$ lies to the left of the maximal root of $P(\under,s)$, namely, $h(t_0)=x_0<h(s)$.
    \medskip

    The case $t<0$ is very similar. As above, it is enough to prove that for every $t_0\in T'=\{t<0\,\mid\,h(t)>c_1\}$ there is an open neighborhood $U$ such that for every $t_0\ne s\in U$, if $s<t_0$ then $h(s)>h(t_0)$ and if $t_0<s$ then $h(t_0)>h(s)$. Again, set $x_0=h(t_0)$. By Lemma \ref{lem:interlacing and real-rootedness}, $(1-t_0)c_2<x_0<(1-t_0)c_1$, so 
    \[
        1-\frac{x_0}{c_2}<t_0<1-\frac{x_0}{c_1}.
    \]
    As $t_0$ is a root of $P(x_0,\under)$, it must be its second largest root by the first part of this proof. In particular, $P(x_0,\under)$ is \textit{increasing} in $t_0$. Let $U\subseteq T'\subseteq(-\infty,0)$ be a neighborhood of $t_0$ in which $P(x_0,\under)$ is increasing and $(1-s)c_2<x_0<(1-s)c_1$ for all $s\in U$. In particular, for every $s\in U$, if $s<t_0$ then $P(x_0,s)<0$, and if $t_0<s$ then $P(x_0,s)>0$.

    By Lemma \ref{lem:interlacing and real-rootedness}, the largest root $h(s)$ of $P(\under,s)$ lies in the interval $((1-s)c_2,(1-s)c_1)$, and $P(\under,s)$ is positive to the right of this root in this inerval and negative to the left. If $s<t_0$ is in $U$, as $P(x_0,s)<0$ we conclude that $h(s)>x_0=h(t_0)$. Likewise, if $t_0<s$, then $P(x_0,s)>0$ and we conclude that means that $h(t_0)=x_0>h(s)$.
\end{proof}

We can now complete the proof of Theorem \ref{thm:n-1s} when $c_1>c_2>\ldots>c_n>0$.

\begin{prop}
    Assume that $c_1>c_2>\ldots>c_n>0$. Then as $t\in\R_{<1}$ approaches $0$ from either side, the largest eigenvalue of $D_t$ strictly decreases.
\end{prop}

\begin{proof}
    By Lemma \ref{lem:largest root at least c_1 then h(t) monotone}, it is enough to show that $h(t)$, the largest root of $P(\under,t)$ (or the largest eigenvalue of $D_t$), satisfies $h(t)>c_1$ for all $t\ne0$ in a neighborhood of $0$. 
    Note that $h(0)=c_1$, as $P(\under,0)=Q(\under,0)=\prod_i(x-c_i)$. Recall that $P(\under,t)$ is monic, so it is enough to show that $P(c_1,t)<0$ in a punctured neighborhood of $0$. Indeed, we show that $t=0$ is a local maxima of $P(c_1,\under)$ by demonstrating that that $\frac{\partial P}{\partial t}(c_1,0)=0$ and $\frac{\partial^2 P}{\partial t^2}(c_1,0)<0$. Using again the formula $P = Q - t\frac{\partial Q}{\partial t}$ from Lemma \ref{lem:two formulas}\eqref{enu:formula with dt},
    $$ \frac{\partial P}{\partial t} = \frac{\partial Q}{\partial t} - \frac{\partial Q}{\partial t} - t \frac{\partial^2 Q}{\partial t^2} = -t \frac{\partial^2 Q}{\partial t^2}$$ which vanishes at $t=0$.
    Now, $$\frac{\partial^2 P}{\partial t^2} = -\frac{\partial^2 Q}{\partial t^2} -t\frac{\partial^3 Q}{\partial t^3}, $$ so $\frac{\partial^2 P}{\partial t^2}(c_1,0)=-\frac{\partial^2 Q}{\partial t^2}(c_1,0)$, and it remains to show that $\frac{\partial^2 Q}{\partial t^2}(c_1,0)>0$. 
    We compute
    \[
        \frac{\partial^2 Q}{\partial t^2}(c_1,0) = \sum_{i\ne j}c_ic_j\prod_{k\ne i,j} (c_kt+x-c_k)\big|_{x=c_1,t=0} =2\sum_{j=2}^n c_1c_j\prod_{k\ne1,j}(c_1-c_k),
    \]
    and all the terms here are positive.
\end{proof}

\subsubsection*{Step III: Completing the proof in the general case}

\begin{proof}[Proof of Theorem \ref{thm:n-1s}] 
    We now fix $c_1\ge c_2 \ge \ldots \ge c_n\ge0$ with $c_2>0$. We need to show that for every $0\ne t\in\R_{<1}$, $h(t)$, the largest root of $P(\under,t)$, is strictly decreasing as $t$ approaches $0$ from either side.

    We first prove that $h$ weakly decreases as $t$ approaches $0$. We elaborate the proof in the interval $(0,1)$, but the argument is completely parallel when $t<0$. So let $0 < t < s < 1$. By the continuity of the set of roots of a degree-$n$ monic polynomial as a function of its coefficients (again, e.g., \cite[Thm.~V.4A]{whitney1972complex}), as we know that $\lambda_{\max}(D_t(c'_1,\ldots,c'_n))<\lambda_{\max}(D_s(c'_1,\ldots,c'_n))$ if there are strict inequalities $c'_1>c'_2>\ldots>c'_n>0$, we conclude that $h(t)=\lambda_{\max}(D_t(c_1,\ldots,c_n))\le\lambda_{\max}(D_s(c_1,\ldots,c_n))=h(s)$. 

    Finally, let us show that $h$ is \textbf{strictly} monotone in each of the intervals $(-\infty, 0)$ and $(0,1)$. If there are two points $t <s$ in the same interval such that $h(t)=h(s)$, then $h$ is constant in the entire interval $[t,s]$. Let $x_0=h(t)=h(s)$. As $P(x_0, \underline{~})$ is a polynomial, we conclude that it is constant. 
    
    Let $k\in[n]\setminus\{1\}$ satisfy $c_1\ge \ldots \ge c_k >c_{k+1}=\ldots=c_n=0$. By the definition of $Q$, $$Q(x_0,t)=x_0^{n-k}\prod_{i=1}^{k}\left (c_it-(c_i-x_0) \right).$$ By the identity $P=Q-t\frac{\partial Q}{\partial t}$ from Lemma \ref{lem:two formulas}\eqref{enu:formula with dt}, the coefficient of $t^k$ in $P(x_0,\under)$ is $(1-k)x_0^{n-k}\prod_{i=1}^k c_i$, which is non-zero as $k\ge2$ (note that $x_0\ne0$ as we know from Lemma \ref{lem:interlacing and real-rootedness} that $0<c_2(1-t)\le x_0$). Hence $P(x_0,\under)$ cannot be a constant polynomial and $h$ must be strictly monotone.    
\end{proof}

\subsection{A conjecture and some remarks}

We have just proved, then, that whenever the weights in the hypergraph $\Gamma=([n],w)$ are supported on sets of size $\ge n-1$, the smallest non-trivial eigenvalue in the $\U(n)$-spectrum of $\Gamma$ is associated with $\kmp_2(\Gamma)$, and in terms of irreps -- with $\rho_{(1,0,\ldots,0,-1)}$ or with $\rho_{(2,0,\ldots,0,-2)}$. Corollary \ref{cor:smallest-evalue for n-1 mean-field} showed that when all $(n-1)$-subsets have equal weights, the smallest non-trivial eigenvalue is associated with $\rho_{(2,0,\ldots,0,-2)}$. The following is an example where this eigenvalue is associated with $\rho_{(1,0,\ldots,0,-1)}$, thus showing that, indeed, both irreps are needed.

\begin{exm} \label{exm:n-1s with smallest evalue in (1,...,-1)}
     Let $n=3$ and consider the (hyper-) graph $\Gamma=([3],w)$ supported on the subsets $\{1,2\}$ and $\{2,3\}$ with weight 1 on each. A short computation shows that the spectrum of $\L(\Gamma,\kmp_1)$ is $\{0,\frac12,\frac32\}$, so $\lambda_{\min}(\Gamma,\rho_{(1,0,\ldots,0,-1)})=\frac12$. The spectrum of $\L(\Gamma,\kmp_2)|_{\pure(3,2)}$ is $\{\frac23,\frac43,2\}$, so $\lambda_{\min}(\Gamma,\rho_{(2,0,\ldots,0,-2)})=\frac23$. Hence, in this example, 
     \[
        \lambda_{\min}(\Gamma, \rho_{(1,0,\ldots,0,-1)}) < \lambda_{\min}(\Gamma,\rho_{(2,0,\ldots,0,-2)}).
     \]
\end{exm}
    
Our proof of Theorem \ref{thm:n-1s} gives something slightly stronger than what Theorem \ref{thm:n-1s} states:

\begin{cor} \label{cor: order on even KMPs and on odd KMPs}
    Let $\Gamma=([n],w)$ be a hypergraph with non-negative weights supported on subsets of size $n-1$. For every $k\in\Z_{\ge1}$ denote \marginpar{$\omega_k$}$$\omega_k\defi\lambda_{\min}(\Gamma,\rho_{(k,0,\ldots,0,-k)})=\lambda_{\min}\left(\L(\Gamma,\kmp_k)|_{\pure(n,k)}\right).$$
    Then 
    \[
        \omega_1\le\omega_3\le\omega_5 \le\ldots~~~~~~~\mathrm{and}~~~~~~~\omega_2\le\omega_4\le\omega_6 \le\ldots
    \]
    with strict inequalities if $n\ge3$ and $\Gamma$ is connected. 
\end{cor}

Computer simulations support the conjecture that this corollary is true in general. Note that like Conjecture \ref{conj:pure KMP}, this is a conjecture about discrete KMP over hypergraphs, which can be stated without any reference to unitary groups. As above, given $\Gamma$, denote by $\omega_k=\omega_k(\Gamma)$ the smallest \textit{new} eigenvalue of $\L(\Gamma,\kmp_k)$: new in the sense that it realized within the invariant subspace $\pure(n,k)$ and not from $\L(\Gamma,\kmp_{k-1})$ via the embedding $\MS(n,k-1)\hookrightarrow\MS(n,k)$ from Proposition \ref{prop:embedding MS(n,k) in MS(n,k+1) properties}.

\begin{conj} \label{conj:order of omega_k}
    Let $\Gamma=([n],w)$ be a hypergraph with non-negative weights and for every $k\in\Z_{\ge1}$ let $\omega_k=\omega_k(\Gamma)$ denote the smallest new eigenvalue of $\L(\Gamma,\kmp_k)$. Then 
    \[
        \omega_1\le\omega_3\le\omega_5 \le\ldots~~~~~~~\mathrm{and}~~~~~~~\omega_2\le\omega_4\le\omega_6 \le\ldots
    \]
\end{conj}

Finally, our analysis also yields the following result about the spectral gap, which will be used in $\S$\ref{subsec:spectral gap for connected hypergraphs} to prove that every connected hypergraph admits a spectral gap.

\begin{cor} \label{cor:spectral gap for connected n-1s}
    Let $\Gamma=([n],w)$ be a hypergraph with non-negative weights supported on subsets of size $n-1$. If $\Gamma$ is connected then there is a (positive) spectral gap in its $\U(n)$-spectrum.
\end{cor}
\begin{proof}
    We already know the smallest non-trivial eigenvalue in the $\U(n)$-spectrum of $\Gamma$ is either $\lambda_{\min}(\Gamma,\rho_{(1,0,\ldots,0,-1)})$ or $\lambda_{\min}(\Gamma,\rho_{(2,0,\ldots,0,-2)})$. We need to show these two values are strictly positive whenever $\Gamma$ is connected. Equivalently, we need to show that $h(t)$, the largest eigenvalue of $D_t$, satisfies $h(t)<\sum_x c_x$ for $t_1=-\frac{1}{n-1}$ and for $t_2=\frac{2}{n(n-1)}$. 
    
    In the current setting, $\Gamma$ is connected if and only if $n\ge3$ and $c_2>0$ (we assume, as above, that $c_1\ge c_2 \ge\ldots\ge c_n\ge0$). In this case $-1<t_1<0<t_2<1$. The $2\times2$ upper-right block of $D_t$ is 
    \[
        \begin{pmatrix}
            c_1 & c_1t\\
            c_2t & c_2
        \end{pmatrix},
    \]
    whose determinant is $c_1c_2(1-t^2)$. Hence $D_{t_1}$ and $D_{t_2}$ both have at least two non-zero eigenvalues. By Lemma \ref{lem:L of reps - basic properties of spectrum}, the spectrum of $A_t$, and therefore of $D_t=(\sum_xc_x)I_n-A_t$, is contained in $[0,\sum_xc_x]$ for $t=t_1,t_2$. Hence $h(t_1),h(t_2)<\sum_xc_x$.
\end{proof}

\section{The mean-field case and spectral gap for connected hypergraphs} \label{sec:mean-field}

In this section we prove Theorem \ref{thm:mean-field} concerning the ``mean-field'' case: hypergraphs $\Gamma$ where the weight of a subset depends only on its size. We also prove that any connected hypergraph admits a positive spectral gap. These two results are put in the same section as their proofs are similar. 

\subsection{Proof of Theorem \ref{thm:mean-field}: the mean-field case}
Throughout this subsection we assume $n\ge2$. Let $c_0,\ldots,c_n\ge0$ and let $\Gamma=([n],w)$ be a weighted hypergraph defined by $w_B=c_{|B|}$ for all $B\subseteq[n]$. Recall that Theorem \ref{thm:mean-field} states that the smallest non-trivial eigenvalue in the $\U(n)$-spectrum of $\Gamma$ is obtained in $\rho_{(2,0,\ldots,0,-2)}$ and is equal to 
\begin{equation} \label{eq:gap in mean-field}
    \sum_{\ell=0}^n c_\ell \frac{n+1}{\ell+1}\binom{n-2}{\ell-2}.    
\end{equation}

First we show that the value in \eqref{eq:gap in mean-field} is indeed an eigenvalue of $\L(\Gamma,\rho_{(2,0,\ldots,0,-2)})$ or, equivalently (by Theorem \ref{thm:(k,-k) dominates non-balanced parts} and Proposition \ref{prop:embedding MS(n,k) in MS(n,k+1) properties}), of $\L(\Gamma,\kmp_2)|_{\pure(n,2)}$. Denote by $\gnl=([n],w)$\marginpar{$\gnl$} the hypergraph with $w(B)=1$ when $|B|=\ell$ and $w(B)=0$ otherwise. For $S\subseteq[n]$, we also denote by $\Gamma^S_\ell=([n],w)$\marginpar{$\Gamma^S_\ell$} the hypergraph with $w(B)=1$ when $B\subseteq S$ and $|B|=\ell$, and $w(B)=0$ otherwise.
So the hypergraph from Theorem \ref{thm:mean-field} can be thought of as $\sum c_\ell\gnl$.

\begin{lemma} \label{lem:mean-field eigenvalue of f_3}
    In the mean-field case, the value in \eqref{eq:gap in mean-field} is an eigenvalue of  $\L(\Gamma,\kmp_2)|_{\pure(n,2)}$.
\end{lemma}

\begin{proof}
    Recall our notation of the indicator vector $\delta_\ii\in\MS(n,k)$ from $\S$\ref{subsec:embed KMP_k in KMP_k+1}. Consider the vector 
    \[
        v_0=\sum_{1\le i< j \le n}\delta_{\{i,j\}}-\sum_{i\in[n]}\frac{n-1}{2}\delta_{\{i,i\}} \in \MS(n,2).
    \]
    It is easy to check that $v_0$ is orthogonal to $\Psi_1(\delta_{\{i\}})$ for all $i\in[n]$ and therefore $v_0\in\pure(n,2)$ (see Definition \ref{def:embed MS(n,k) in MS(n,k+1)}). The one-dimensional space spanned by $v_0$ is the intersection of $\pure(n,2)$ and the subspace of $\MS(n,2)$ consisting of vectors invariant under the natural action of $\sym(n)$ (namely, vectors such that all multisets of the form $\{i,i\}$ have the same coefficient and so do all multisets of the form $\{i,j\}$ with $i\ne j$). For every $\ell=0,1,\ldots,n$, the operator $\L(\gnl,\kmp_2)$ preserves both subspaces of $\MS(n,2)$ and therefore their intersection. Hence $v_0$ is an eigenvector of $\L(\gnl,\kmp_2)$ for all $\ell$. It remains to show that for any given $\ell$, the corresponding eigenvalue is $\frac{n+1}{\ell+1}\binom{n-2}{\ell-2}$. Note that the eigenvalue is precisely the coefficient of $\delta_{\{1,2\}}$ in $\L(\gnl,\kmp_2).v_0$. 

    Let $B\subseteq[n]$ be of size $\ell$. The coefficients of $\delta_{\{1,2\}}$ in $\N_B.v_0$ is
    \[
        \begin{cases}
            \frac{1}{\ms{\ell}{2}}\left(\binom{\ell}{2}\cdot1-\ell\cdot \frac{n-1}{2}\right)& \mathrm{if}~ 1,2\in B\\
            1 & \mathrm{otherwise}.
        \end{cases}
    \]
    (Note that even if $B$ contains exactly one of $1$ or $2$, $\N_B$ does not alter the coefficient of $\delta_{\{1,2\}}$: it is an average of $\ell$ ones, which is equal to one.) As $\L(\gnl,\kmp_2)=\sum_{B\subseteq[n]\colon|B|=\ell}(I-\N_B)$, the coefficient of $\delta_{\{1,2\}}$ in $\L(\gnl,\kmp_2).v_0$ is
    \begin{eqnarray*}
        && \binom{n}{\ell}-\left[\binom{n-2}{\ell-2}\cdot \frac{\binom{\ell}{2}-\ell\cdot\frac{n-1}{2}}{\ms{\ell}{2}}+\left(\binom{n}{\ell}-\binom{n-2}{\ell-2}\right)\cdot 1\right] \\
        &=& \binom{n-2}{\ell-2}\left( 1-\frac{(\ell-1)-(n-1)}{\ell+1} \right) \\
        &=& \binom{n-2}{\ell-2}\cdot \frac{n+1}{\ell+1}.
    \end{eqnarray*}    
\end{proof}

Corollary \ref{cor:smallest-evalue for n-1 mean-field} established the statement of Theorem \ref{thm:mean-field} in the special case when $c_0=c_1=\ldots=c_{n-2}=0$. In fact, this special case is the most involved part of the proof: the general case follows from it by a rather short argument. The same argument appears, although in the disguise of a very different language, in the proof of \cite[Thm.~1.8]{bristiel2024entropy}, establishing the mean-field case of Caputo's conjecture in $\sym(n)$, as well as in older proofs of an analogous result for graphs in $\mathrm{SO}(n)$, analysing a model emanating from Kac's master equation \cite{maslen2003eigenvalues,carlen2003determination} (see $\S$\ref{subsec:related works} above).

\begin{proof}[Proof of Theorem \ref{thm:mean-field}]
    Given Lemma \ref{lem:mean-field eigenvalue of f_3}, it is enough to prove that \eqref{eq:gap in mean-field} is a lower bound for the non-trivial spectrum of $\Gamma$ in the mean-field case, and for this it is enough to prove that $\frac{n+1}{\ell+1}\binom{n-2}{\ell-2}$ is a lower bound for the non-trivial spectrum of $\gnl$.
    
    Recall our notation of the Laplacian $\L(\Gamma)$ in the regular representation and of inequalities of operators in the regular representation from $\S$\ref{subsec:spectrum is real non-negative and bounded}. 
    As explained in $\S$\ref{subsec:spectrum is real non-negative and bounded}, $P_{[n]}$ is the projection onto the trivial representation, so $\L(\gn_n)=\L_{[n]}=I-P_{[n]}$ is the projection onto the non-trivial irreps and its entire non-trivial spectrum consists of ones. This proves the $\ell=n$ case. It also shows that what we need to prove is equivalent to the operator inequality 
    \begin{equation}\label{eq:inequality of operators to prove}
        \L(\gnl) \ge \frac{n+1}{\ell+1}\binom{n-2}{\ell-2}\L(\gn_n),
    \end{equation}
    as both sides give 0 in the trivial representation, and the right hand side is $\frac{n+1}{\ell+1}\binom{n-2}{\ell-2}$ in any non-trivial irrep.

    Corollary \ref{cor:smallest-evalue for n-1 mean-field} yields the claim when $1\le\ell=n-1$, namely, it shows that
    \begin{equation} \label{eq:n-1 case in meanfield proof}
        \L(\gn_{n-1}) \ge \frac{(n+1)(n-2)}{n}\L(\gn_n).
    \end{equation}
    For every fixed $\ell\ge1$, we prove the general case by induction on $n$. The base cases of $n=\ell$ and $n=\ell+1$ were established. Assume that $n\ge\ell+2$. Then
    \begin{eqnarray*}
        \L(\gnl) &=& \frac{1}{n-\ell}\sum_{x\in[n]}\L\left(\Gamma_\ell^{[n]\setminus\{x\}}\right) \\
        &\stackrel{\text{by induction}}{\ge} & \frac{1}{n-\ell} \sum_{x\in[n]}\frac{n}{\ell+1}\binom{n-3}{\ell-2}\L\left(\Gamma^{[n]\setminus\{x\}}_{n-1}\right) \\
        &=& \frac{n}{(n-\ell)(\ell+1)}\binom{n-3}{\ell-2}\L(\gn_{n-1}) \\
        &\stackrel{\eqref{eq:n-1 case in meanfield proof}}{\ge}& \frac{n}{(n-\ell)(\ell+1)}\binom{n-3}{\ell-2} \cdot \frac{(n+1)(n-2)}{n}\L(\gn_n) \\
        &=& \binom{n-2}{\ell-2}\frac{n+1}{\ell+1}\L(\gn_n).
    \end{eqnarray*}
    In the first inequality we relied on the induction hypothesis and the observation from Lemma \ref{lem:inequalities remain true with branching} that operator inequalities of hypergraph-Laplacians remain true when the ambient vertex set is enlarged.
    
    Finally, when $\ell=0$, both sides of \eqref{eq:inequality of operators to prove} are trivially zero.
\end{proof}

\subsection{A spectral gap for any connected hypergraph} \label{subsec:spectral gap for connected hypergraphs}

Recall from page \pageref{page:connected definition} the (intuitive) definition of a connected hypergraph. We mentioned that the fact that a connected hypergraph has a (positive) spectral gap in its $\U(n)$-spectrum follows from \cite{benoist2016spectral}. We give here an alternative proof which goes along similar lines to the proof of Theorem \ref{thm:mean-field} above.

\begin{prop} \label{prop:spectral gap for connected hypergraphs}
Let $\Gamma=([n],w)$ be a connected hypergraph with non-negative weights. Then the $\U(n)$-spectrum of $\Gamma$ has a spectral gap.    
\end{prop}

\begin{proof}
    As explained in the proof of Theorem \ref{thm:mean-field}, a spectral gap for $\Gamma$ is the same as proving that $\L(\Gamma)\ge\varepsilon\L_{[n]}$ for some $\varepsilon>0$.
    
    First, for every $B\subseteq[n]$, the Laplacian $\L_B$ is a non-negative operator. Hence $\L(\Gamma)=\sum_Bw_B\L_B\ge w_{[n]}\L_{[n]}$, and we are done if $w_{[n]}>0$. Note that this argument establishes the proposition whenever $n\le2$. So assume from now on that $n\ge3$ and $w_{[n]}=0$. 

    We continue by induction on $n$. Assume that $n\ge3$ and that the statement holds for $n-1$. 
    The graph $G=([n],z)$ induced by $\Gamma$ is the ordinary graph obtained by replacing every hyperedge $B$ with all $\binom{|B|}{2}$ edges supported on subsets of size 2 in $B$ with the same weight (in particular, hyperedges of size $\le1$ are ignored). Namely, it is an ordinary weighted simple graph on the vertex set $[n]$, and for every $1\le i< j\le n$, the weight of the edge $\{i,j\}$ is $z_{\{i,j\}}=\sum_{B\ni i,j}w_B$.

    Clearly, $\Gamma$ is connected if and only if its induced graph $G$ is connected. Every finite connected graph on $\ge2$ vertices has at least two distinct vertices which are non-separating\footnote{A vertex of a connected graph is non-separating if the removal of this vertex together with all incident edges leaves the graph connected.} (this can be seen, for example, by considering a spanning tree and noticing it has at least two leaves). Let $x,y\in[n]$ be two distinct non-separating vertices of $G$. 

    Let $\Gamma_x=([n]\setminus\{x\},w^x)$ denote the weighted hypergraph defined as follows. For every $B\subseteq[n]\setminus\{x\}$ define $w^x_B=w_B+w_{B\sqcup\{x\}}$. From Lemma \ref{lem:inequality between subset} we deduce that $\L(\Gamma_x)\le \L(\Gamma)$. The graph induced from $\Gamma_x$ on $[n]\setminus\{x\}$ is connected as it is identical to the induced subgraph of $G$ with vertex set $[n]\setminus\{x\}$. Hence $\Gamma_x$ is a connected hypergraph on $[n]\setminus\{x\}$.
    
    Define analogously $\Gamma_y=([n]\setminus\{y\},w^y)$, which satisfies $\L(\Gamma_y)\le \L(\Gamma)$, and is connected on the vertex set $[n]\setminus\{y\}$. We deduce that 
    \[
        \L(\Gamma)\ge \frac12\left(\L(\Gamma_x)+\L(\Gamma_y)\right).
    \]
    By the induction hypothesis, there exist $\varepsilon_x,\varepsilon_y>0$ so that $\L(\Gamma_x)>\varepsilon_x\L_{[n]\setminus\{x\}}$ and $\L(\Gamma_y)>\varepsilon_y\L_{[n]\setminus\{y\}}$. Finally, Corollary \ref{cor:spectral gap for connected n-1s} says that $\varepsilon_x\L_{[n]\setminus\{x\}}+\varepsilon_y\L_{[n]\setminus\{y\}}$ admits a spectral gap in $\U(n)$, namely, there exists some $\varepsilon>0$ with $\varepsilon_x\L_{[n]\setminus\{x\}}+\varepsilon_y\L_{[n]\setminus\{y\}}\ge \varepsilon\L_{[n]}$. We are done as
    \[
        \L(\Gamma)\ge \frac12\left(\L(\Gamma_x)+\L(\Gamma_y)\right) \ge \frac12(\varepsilon_x\L_{[n]\setminus\{x\}}+\varepsilon_y\L_{[n]\setminus\{y\}})\ge\frac{\varepsilon}{2}\L_{[n]}.
    \]
\end{proof}


\section{The $\sym(n)$-spectrum is contained in the $\U(n)$-spectrum} \label{sec:sym(n) spectrum contained}

In the current section we prove Theorem \ref{thm:evalues of Sn} stating that for any weighted hypergraph $\Gamma=([n],w)$, the $\sym(n)$-spectrum of $\Gamma$ is contained in its $\U(n)$-spectrum. 

Recall that $\rho_\std\colon\U(n)\to\gl(V)$ marks the $n$-dimensional standard representation of $\U(n)$. We discussed above the tensor power $\R_{k,0}=V^{\otimes k}$ and the symmetric power $\Sym^k(\rho_\std)=\Sym^k(V)$ and described concrete realizations of the two in $\S$\ref{subsec:discrete KMP as tor-inv subspace}. We now also need the exterior power $\bigwedge^k(\rho_\std)=\bigwedge^k(V)$ for every $k\in[n]$. It can be constructed as follows. Its canonical basis is 
\[
    \left\{e_{i_1}\wedge e_{i_2}\wedge\ldots\wedge e_{i_k}\,\mid\, 1\le i_1< i_2 <\ldots<i_k \le n \right\}.
\]
The action of $A\in\U(n)$ on a basis element is given by
\[
    A.(e_{i_1}\wedge e_{i_2}\wedge\ldots\wedge e_{i_k})=\sum_{j_1,\ldots,j_k\in[n]}A_{j_1,i_1}\cdots A_{j_k,i_k} \cdot e_{j_1}\wedge\ldots\wedge e_{j_k},
\]
where for every $\sigma\in\sym(k)$ we identify $e_{j_{\sigma(1)}}\wedge\ldots\wedge e_{j_{\sigma(k)}} = \sign (\sigma)\cdot e_{j_1}\wedge\ldots\wedge e_{j_k}$. In particular, if two of the $j_i$'s are identical, then $e_{j_1}\wedge\ldots\wedge e_{j_k}=0$.

For $k\in[n]$, consider the $\U(n)$ representation \marginpar{$Y_k$} 
\[
    Y_k=\left(\bigwedge^k(\rho_\std)\right)\otimes(\rho_\std^*)^{\otimes k}=\left(\bigwedge^k(V)\right)\otimes(V^*)^{\otimes k}.
\]
Parallel to the notation from $\S$\ref{subsec:discrete KMP as tor-inv subspace}, we can mark the standard basis elements of $Y_k$ as
\[
    \left\{ \left(e_{i_1}\wedge\ldots\wedge e_{i_k}\right)\otimes\left( e^{j_1}\otimes\ldots \otimes e^{j_k} \right)\,\middle|\, \begin{gathered}
        1\le i_1< i_2 <\ldots<i_k \le n \\
        j_1,\ldots,j_k\in[n]
    \end{gathered} \right\}.
\]
The torus-invariant subspace of $Y_k$ is where the $\sym(n)$-spectrum of a hypergraph can be found within its $\U(n)$-spectrum. First, we identify this subspace.

\begin{prop} \label{prop:tor-inv subspace of exterior tensor R_0,k}
    Let $k\in[n]$. The torus-invariant subspace of $Y_k$ has a linear basis consisting of the vectors
    \[
        \left\{ v_{i_1,\ldots,i_k}\defi\left(e_{i_1}\wedge\ldots\wedge e_{i_k}\right)\otimes\left( e^{i_1}\otimes\ldots \otimes e^{i_k} \right)\,\middle|\, 
        i_1,i_2,\ldots,i_k\in[n]~\mathrm{all~distinct}  \right\}.
    \]
\end{prop}
\begin{proof}
    Recall that $T_n$ marks the diagonal subgroup of $\U(n)$, and consider the action of $P_{T_n}$, the projection on the torus-invariant subspace, on the basis element $(e_{i_1}\wedge\ldots\wedge e_{i_k})\otimes( e^{j_1}\otimes\ldots \otimes e^{j_k} )$ of $Y_k$ with $i_1<\ldots<i_k$ and $j_1,\ldots,j_k\in[n]$. In the resulting vector, the coefficient of $(e_{i'_1}\wedge\ldots\wedge e_{i'_k})\otimes( e^{j'_1}\otimes\ldots \otimes e^{j'_k} )$ (again, $i'_1<\ldots<i'_k$ and $j'_1,\ldots,j'_k\in[n]$) is
        \[
        \mathbb{E}_{A\in T_n}\left[\left(\sum_{\sigma\in\sym(k)}\sign(\sigma) A_{i'_{\sigma(1)},i_1}\cdots A_{i'_{\sigma(k)},i_k} \right)
        \cdot \overline{A_{j'_1,j_1}}\cdots \overline{A_{j'_k,j_k}} \right].        
    \]
    Recall that $A\in T_n$ is diagonal, so the first sum does not vanish only if $i'_t=i_t$ for all $t\in[k]$ and then the only non-vanishing summand is when $\sigma=\mathrm{id}$. Likewise, the second sum does not vanish only if $j'_t=j_t$ for all $t\in[k]$. So the action of $P_{T_n}$ on $(e_{i_1}\wedge\ldots\wedge e_{i_k})\otimes( e^{j_1}\otimes\ldots \otimes e^{j_k})$ gives the same basis element with coefficient 
    \[
        \mathbb{E}_{A\in T_n}\left[A_{i_1,i_1}\cdots A_{i_k,i_k} \cdot
        \overline{A_{j_1,j_1}}\cdots \overline{A_{j_k,j_k}} \right],
    \]
    which does not vanish, and is equal to one, precisely when $j_1,\ldots,j_k$ are an arbitrary rearrangement of $i_1,\ldots,i_k$.
\end{proof}

We now show that the action of $\L(\Gamma,Y_K)$ on $\torinv(Y_k)$ is precisely the action $\L(\Gamma,Z_k)$, where \marginpar{$Z_k$}$Z_k$ is the permutation-representation of $\sym(n)$ acting on $k$-tuples of distinct elements from $[n]$. This is via the identification of the basis element $v_{i_1,\ldots,i_k}$ of $\torinv(Y_K)$ and the basis element $(i_1,\ldots,i_k)$ of $Z_k$. 
Of course, showing this is true for any weighted hypergraph is equivalent to showing that this is true for any subset $B\subseteq[n]$ separately.

\begin{lemma} \label{lem:action on torinv of Y_k}
    Let $k\in[n]$ and $B\subseteq[n]$. The action of the projection $P_B$ in $\torinv(Y_k)$ is identical to the action of the projection $P_B$ in the $\sym(n)$-representation $Z_k$, namely, it is given by 
    \[
        P_B(v_{i_1,\ldots,i_k})=\frac{1}{|\sym(B)|}\sum_{\sigma\in\sym(B)}v_{\sigma(i_1),\ldots,\sigma(i_k)}.
    \]
\end{lemma}

\begin{proof}
    Denote $b=|B|$. For simplicity of notation, we assume without loss of generality that the indices within $i_1,\ldots,i_k$ that belong to $B$ constitute a prefix of length $0\le r\le k$, namely, $i_1,\ldots,i_r\in B$ while $i_{r+1},\ldots,i_k\notin B$. As $i_1,\ldots,i_k$ are all distinct, we must have $b\ge r$. Then, 
    \begin{eqnarray*}
        P_B.v_{i_1,\ldots,i_k} &=& \int_{A\in \U_B}A.v_{i_1,\ldots,i_k}  \\
        &=& \sum_{\substack{i'_1,\ldots,i'_k\\j'_1,\ldots,j'_k}\in [n]} \int_{A\in \U_B}A_{i'_1i_1}\cdots A_{i'_ki_k}\cdot\overline{A_{j'_1i_1}\cdots A_{j'_ki_k}} \cdot \left(e_{i'_1}\wedge\cdots\wedge e_{i'_k}\otimes e^{j'_1}\otimes \ldots\otimes e^{j'_k}\right). \\
        &=& \sum_{\substack{i'_1,\ldots,i'_r\\j'_r,\ldots,j'_r}\in B} \int_{A\in \U_B}A_{i'_1i_1}\cdots A_{i'_ri_r}\cdot\overline{A_{j'_1i_1}\cdots A_{j'_ri_r}} \cdot \\
        &&~~~~~~~~~~~~~~~\left(e_{i'_1}\wedge\cdots\wedge e_{i'_r}\wedge e_{i_{r+1}}\cdots\wedge e_{i_k}\otimes e^{j'_1}\otimes \ldots\otimes e^{j'_r}\otimes e^{i_{r+1}}\otimes \ldots\otimes e^{i_k}\right),
\end{eqnarray*}
where in the last inequality we restricted to indices $i'_1,\ldots,i'_k,j'_1,\ldots,j'_k$ whose associated summands are not obviously vanishing for $A\in\U_B$. In fact, because of the properties of $\bigwedge^k(V)$, we must also have that $i'_1,\ldots,i'_r$ are all distinct. Then, by Lemma \ref{lem:unbalanced integrals vanish}, $j'_1,\ldots,j'_r$ must be a rearrangement of $i'_1,\ldots,i'_r$. We get that for every \textit{distinct} $j'_1,\ldots,j'_r\in B$, the coefficient of $e_{j'_1}\wedge\cdots\wedge e_{j'_r}\wedge e_{i_{r+1}}\cdots\wedge e_{i_k}\otimes e^{j'_1}\otimes \ldots\otimes e^{j'_r}\otimes e^{i_{r+1}}\otimes \ldots\otimes e^{i_k}$ in $P_B.v_{i_1,\ldots,i_k}$ is
\begin{eqnarray*}
    && \sum_{\sigma\in S_r}\sign(\sigma)\int_{A\in \U_B} A_{j'_{\sigma(1)}i_1}\cdots A_{j'_{\sigma(r)}i_r} \cdot \overline{A_{j'_1,i_1} \cdots A_{j'_r,i_r}}\\
    &\stackrel{\mathrm{Theorem~}\ref{thm:weingarten}}{=}& \sum_{\sigma\in S_r}\sign(\sigma) \wg_{r,b}(\sigma)\\
    &\stackrel{\mathrm{Lemma~}\ref{lem:sum and signed sum of weingarten}}{=}&\frac{1}{b(b-1)\cdots(b-r+1)}=\frac{(b-r)!}{b!}.
\end{eqnarray*}

On the other hand, consider the action of the projection $P_B$ in the $\sym(n)$-representation $Z_k$ on the basis element $(i_1,\ldots,i_k)$. The image is a linear combination over the elements $(j'_1,\ldots,j'_r,i_{r+1},\ldots,i_k)$ with $j'_1,\ldots,j'_r\in B$ distinct. The coefficient of $(j'_1,\ldots,j'_r,i_{r+1},\ldots,i_k)$ is
\begin{eqnarray*}
    && \frac{1}{\sym(B)}\#\left\{\sigma\in\sym(B)\,\mid\, \sigma(i_1)=j'_1,\ldots,\sigma(i_r)=j'_r\right\} = \frac{(b-r)!}{b!}.
\end{eqnarray*}
This proves the statement of the lemma.
\end{proof}

We can now easily prove a detailed version of Theorem \ref{thm:evalues of Sn}.

\begin{theorem} \label{thm:evalues of Sn - elaborated}
    Let $\Gamma=([n],w)$ be a weighted hypergraph. Let $\nu\vdash n$ and $\pi_\nu\in\irr(\sym(n))$ be the corresponding irreducible representation of $\sym(n)$. Assume that $\nu$ has at most $k\in[n]$ blocks outside the first row (namely, $\nu_1\ge n-k$). Then the eigenvalues of $\L(\Gamma,\pi_\nu)$ are contained, as a multiset, in the multiset of eigenvalues of $\L(\Gamma,Y_k)$, and thus also in that of $\L(\Gamma,R_{k,k})$.

    In particular, the entire $\sym(n)$-spectrum of $\Gamma$ is contained in $\L(\Gamma,Y_{n-1})$ or in $\L(\Gamma,R_{n-1,n-1})$, as well as in $\L(\Gamma,Y_n)$ or in $\L(\Gamma,R_{n,n})$.    
\end{theorem}

\begin{proof}
    This follows from Lemma \ref{lem:action on torinv of Y_k} together with the fact that the $\sym(n)$-representation $Z_k$ decomposes as a direct sum of all irreps $\pi_\nu$ with $\nu\vdash n$ having at most $k$ blocks outside the first row (where each of this irreps has a non-zero multiplicity, given by a corresponding Kostka number) -- see, e.g., \cite[Cor.~4.39]{fulton2013representation}. 
\end{proof}

\section{Possible and impossible extensions} \label{sec:Possible and impossible extensions}
Our study, results and conjectures suggest natural extensions. Some of them were ruled out by simulations we conducted. Others are still plausible. We list a few of them.

\begin{itemize}
    \item \textbf{A more direct analog of transpositions:} The original Aldous' conjecture in $\sym(n)$ dealt with transpositions. In the $n$-dimensional standard representation of $\sym(n)$, these are permutation matrices $P$ such that $\mathrm{rank}(I-P)=1$. One may wonder if some kind of an Aldous phenomenon holds when one considers symmetric probability distributions on $\U(n)$ supported on matrices $A\in\U(n)$ with $\mathrm{rank}(I-A)=1$. These are matrices with spectrum $\{\lambda,1,\ldots,1\}$ for some $1\ne\lambda\in S^1$. However, simulations we conducted found no deterministic Aldous-type phenomenon is this regime.

    \item \textbf{Stabilizers of general linear subspaces of $\C^n$:} Our hypergraph measures on $\U(n)$ consist of arbitrary positive measures on the set of subsets $B\subseteq[n]$, or, equivalently, on linear subspaces $W\le\C^n$ which are \textit{parallel} to the axes. Once such a subspace is chosen, we take a Haar-random element from the pointwise-stabilizer of $W^\perp$. Is it possible that Conjecture \ref{conj:main conj} holds for arbitrary linear subspaces, which are not necessarily parallel to the axes? Simulations we conducted rule out this option.

    \item \textbf{Non-Haar distributions on the subgroups $\U_B$:} In our hypergraph measures on $\U(n)$, we take an arbitrary non-negative measure on the subsets $B\subseteq[n]$, but once some $B\subseteq[n]$ is chosen, always the Haar measure on $\U_B$. It is plausible that one may consider more general measures. These measures should probably be restricted in certain senses (e.g., symmetric in the sense of $g$ vs.~$g^{-1}$), and possibly parallel for different subsets $B$ of the same size (e.g., if $\sigma\in\sym(n)$ and $B'=\sigma.B$, then $\mu_{B'}(E)=\mu_B(\sigma^{-1}E\sigma)$). For example, in the works \cite{maslen2003eigenvalues,carlen2003determination}, the authors analyze the complete graph on $n$ vertices (with constant weight 1 on every edge), and consider the resulting measure on $\mathrm{SO}(n)$. However, the measures they consider on $\mathrm{SO}_{i,j}\cong\mathrm{SO}(2)\cong S^1$, are quite a wide range of the possible measures on the unit circle. In particular, some parts of their proofs, like the part we generalized in our proof of Theorem \ref{thm:mean-field}, work seamlessly in this more general regime. We did not pursue this direction and cannot rule it out.
\end{itemize}

\section{Open problems} \label{sec:open problems}
We gather some open problems that naturally arise from this paper.

\paragraph{Aldous in $\U(n)$:} In our view, the most appealing open question this paper gives rise to is Conjecture \ref{conj:main conj}, or its equivalent Conjecture \ref{conj:main KMP style}, which claim that the phenomenon we proved in the mean-field case of for hypergraphs supported on subsets of size $\ge n-1$, holds for an arbitrary hypergraph with non-negative weights. 

Is there a more detailed phenomenon about more irreps that are always ``spectrally-dominated'' by other irreps (as in Theorem \ref{thm:(k,-k) dominates non-balanced parts})? In particular, is it possible that for any $\rho\in\irr(\U(n))$, $\rho\ne\triv,\rho_{(2,0,\ldots,0,-2)}$, we have  $$\lambda_{\min}(\Gamma,\rho_{(1,0,\ldots,0,-1)})\le \lambda_{\min}(\Gamma,\rho)?$$
If true, then together with Theorem \ref{thm:evalues of Sn}, this would yield Caputo's Conjecture \ref{conj:Caputo}.

\paragraph{Discrete KMP on hypergraphs:} Along the paper we mentioned two conjectures that can be stated completely in terms of the discrete KMP processes on the hypergraphs $\Gamma$ without any relation to $\U(n)$-representation. The first, Conjecture \ref{conj:pure KMP}, asks if the smallest possible non-trivial eigenvalue of $\L(\Gamma,\kmp_k)$ for any $k$ is always obtained when $k=2$. The second, Conjecture \ref{conj:order of omega_k}, asks if the order we established on the smallest new eigenvalue of $\L(\Gamma,\kmp_k)$ for different $k$'s in the case of hypergraphs supported on sets of size $\ge n-1$, holds for all hypergraphs.

\paragraph{An Octopus inequality:} The proof of the original Aldous' conjecture in $\cite{caputo2010proof}$ relies on a mysterious inequality in the group algebra $\C[\sym(n)]$, which is coined the ``octopus inequality'' in that paper (see \cite{alon2025aldous-caputo} for some generalizations). Is there an analogous inequality in $\U(n)$? If so, is it useful for making progress towards Conjecture \ref{conj:main conj}? 

\paragraph{A much more general regime?} The paper \cite{chen2025incompressibility} proves a spectral gap phenomenon, similar in spirit to our results and conjectures here, but when small Haar-unitaries are embedded diagonally by $A\mapsto A\times A\times\ldots\times A\in \U(n)$, where $A\times\ldots\times A$ is a block-diagonal element in $\U(n)$. Specifically, they hint that a similar phenomenon to Conjecture \ref{conj:main conj} holds in their model (see Page 12 in the arXiv version of \cite{chen2025incompressibility}). Is there a model for measures on $\U(n)$ which incorporates both our hypergraph measures and the measures from \cite{chen2025incompressibility}) in which the Aldous phenomenon holds?

\paragraph{Aldous phenomena in other groups:} Hypergraph measures on $\U(n)$ and $\sym(n)$ have natural analogs in many other sequences of matrix groups, such as $\gl_n(q)$ (matrix groups over finite fields) or $\mathrm{O}(n)$ (orthogonal groups). Is there an analogous Aldous phenomenon in these cases? We mention that Levhari and the second author have studied hypergraph measures on wreath products of the form $G\wr \sym(n)$, where $G$ is an arbitrary finite group, and found in \cite{levhari2026wreath} that an Aldous phenomenon there follows from the one in $\sym(n)$.

\bibliographystyle{alpha}
\bibliography{unitary-aldous-bib}

\newcommand{\etalchar}[1]{$^{#1}$}
\begin{thebibliography}{ACD{\etalchar{+}}20}

\bibitem[ACD{\etalchar{+}}20]{aldous2020life}
D.~Aldous, P.~Caputo, R.~Durrett, A.~E. Holroyd, P.~U. Jung, and A.~L. Puha.
\newblock The life and mathematical legacy of {T}homas {M}. {L}iggett.
\newblock {\em Notices Amer. Math. Soc.}, 68(1), 2020.

\bibitem[AG26]{AlonGhosh25}
G.~Alon and S.~Ghosh.
\newblock Spectral gap for the signed interchange process with arbitrary sets.
\newblock preprint arXiv:2510.04244, 2026+.

\bibitem[AK13]{aldousorder}
G.~Alon and G.~Kozma.
\newblock Ordering the representations of {$S_n$} using the interchange
  process.
\newblock {\em Canad. Math. Bull.}, 56(1):13--30, 2013.

\bibitem[AKP25]{alon2025aldous-caputo}
G.~Alon, G.~Kozma, and D.~Puder.
\newblock On the {A}ldous-{C}aputo spectral gap conjecture for hypergraphs.
\newblock {\em Math. Proc. Cambridge Philos. Soc.}, 179(2):259--298, 2025.

\bibitem[BC24]{bristiel2024entropy}
A.~Bristiel and P.~Caputo.
\newblock Entropy inequalities for random walks and permutations.
\newblock {\em Ann. Ins. Henri Poincar{\'e} (B) Probab. Stat.}, 60(1):54--81,
  2024.

\bibitem[BCH{\etalchar{+}}94]{benkart1994tensor}
G.~Benkart, M.~Chakrabarti, T.~Halverson, R.~Leduc, C.~Y. Lee, and J.~Stroomer.
\newblock Tensor product representations of general linear groups and their
  connections with {B}rauer algebras.
\newblock {\em J. Algebra}, 166(3):529--567, 1994.

\bibitem[BdS16]{benoist2016spectral}
Y.~Benoist and N.~de~Saxc{\'e}.
\newblock A spectral gap theorem in simple {L}ie groups.
\newblock {\em Invent. Math.}, 205(2):337--361, 2016.

\bibitem[BG08]{bourgain2008spectral}
J.~Bourgain and A.~Gamburd.
\newblock On the spectral gap for finitely-generated subgroups of {$SU(2)$}.
\newblock {\em Invent. Math.}, 171(1):83--121, 2008.

\bibitem[BG12]{bourgain2012spectral}
J.~Bourgain and A.~Gamburd.
\newblock A spectral gap theorem in {$SU(d)$}.
\newblock {\em J. Eur. Math. Soc. (JEMS)}, 14(5), 2012.

\bibitem[Bou17]{bourgain2017random}
J.~Bourgain.
\newblock On random walks in large compact {L}ie groups.
\newblock In B.~Klartag and E.~Milman, editors, {\em Geometric Aspects of
  Functional Analysis: Israel Seminar (GAFA) 2014--2016}, volume 2169, pages
  55--63. Springer, 2017.

\bibitem[Bum13]{bump2013lie}
D.~Bump.
\newblock {\em Lie Groups}, volume 225 of {\em Graduate texts in Mathematics}.
\newblock Springer Science \& Business Media, second edition, 2013.

\bibitem[CCL03]{carlen2003determination}
E.~A. Carlen, M.~C. Carvalho, and M.~Loss.
\newblock Determination of the spectral gap for {K}ac's master equation and
  related stochastic evolution.
\newblock {\em Acta Math.}, 191:1--54, 2003.

\bibitem[Ces16]{cesi2016few}
F.~Cesi.
\newblock A few remarks on the octopus inequality and {A}ldous’ spectral gap
  conjecture.
\newblock {\em Comm. Algebra}, 44(1):279--302, 2016.

\bibitem[Ces20]{cesi2020spectral}
F.~Cesi.
\newblock On the spectral gap of some {C}ayley graphs on the {W}eyl group
  {$W(Bn)$}.
\newblock {\em Linear Algebra Appl.}, 586:274--295, 2020.

\bibitem[CHH{\etalchar{+}}25]{chen2025incompressibility}
C.F. Chen, J.~Haah, J.~Haferkamp, Y.~Liu, T.~Metger, and X.~Tan.
\newblock Incompressibility and spectral gaps of random circuits.
\newblock In {\em 2025 IEEE 66th Annual Symposium on Foundations of Computer
  Science (FOCS)}, pages 1304--1312. IEEE, 2025.
\newblock Full version available in arXiv:2406.07478v3.

\bibitem[CLR10]{caputo2010proof}
P.~Caputo, T.~Liggett, and T.~Richthammer.
\newblock Proof of {A}ldous' spectral gap conjecture.
\newblock {\em J. Amer. Math. Soc.}, 23(3):831--851, 2010.

\bibitem[Col03]{collins2003moments}
B.~Collins.
\newblock Moments and cumulants of polynomial random variables on
  unitarygroups, the {I}tzykson-{Z}uber integral, and free probability.
\newblock {\em Int. Math. Res. Not.}, 2003(17):953--982, 2003.

\bibitem[C{\'S}06]{collins2006integration}
B.~Collins and P.~{\'S}niady.
\newblock Integration with respect to the {H}aar measure on unitary, orthogonal
  and symplectic group.
\newblock {\em Comm. Math. Phys.}, 264(3):773--795, 2006.

\bibitem[FH13]{fulton2013representation}
W.~Fulton and J.~Harris.
\newblock {\em Representation theory: a first course}, volume 129.
\newblock Springer Science \& Business Media, 2013.

\bibitem[Fol16]{folland2016course}
G.~B. Folland.
\newblock {\em A course in abstract harmonic analysis}.
\newblock CRC press, 2016.

\bibitem[Gho26]{Ghosh23}
Subhajit Ghosh.
\newblock Aldous-type spectral gap results for the complete monomial group.
\newblock {\em Ann. Inst. Henri Poincar\'{e} Probab. Stat.}, 2026+.
\newblock to appear; also at arXiv:2309.12154.

\bibitem[Har08]{Harville2008MatrixAlgebra}
D.~A. Harville.
\newblock {\em Matrix algebra from a statistician’s perspective}.
\newblock Springer, 2008.

\bibitem[HLT25]{haah2025efficient}
J.~Haah, Y.~Liu, and X.~Tan.
\newblock Efficient approximate unitary designs from random {P}auli rotations.
\newblock {\em Comm. Math. Phys.}, 406(12):1--24, 2025.

\bibitem[Kac56]{kac1956foundations}
M.~Kac.
\newblock Foundations of kinetic theory.
\newblock In J.~Neyman, editor, {\em Proceedings of The third Berkeley
  symposium on mathematical statistics and probability}, volume~3, pages
  171--197, 1956.

\bibitem[KMP82]{kipnis1982heat}
C.~Kipnis, C.~Marchioro, and E.~Presutti.
\newblock Heat flow in an exactly solvable model.
\newblock {\em J. Stat. Phys.}, 27(1):65--74, 1982.

\bibitem[Koi89]{koike1989decomposition}
K.~Koike.
\newblock On the decomposition of tensor products of the representations of the
  classical groups: by means of the universal characters.
\newblock {\em Adv. Math.}, 74(1):57--86, 1989.

\bibitem[KQS25]{kim2025spectral}
S.~Kim, M.~Quattropani, and F.~Sau.
\newblock Spectral gap of the {KMP} and other stochastic exchange models on
  arbitrary graphs.
\newblock preprint arXiv:2505.02400, 2025.

\bibitem[KS24]{kim2024spectral}
S.~Kim and F.~Sau.
\newblock Spectral gap of the symmetric inclusion process.
\newblock {\em Ann. Appl. Probab.}, 34(5):4899--4920, 2024.

\bibitem[KW26]{Kanegae04032026}
K.~Kanegae and H.~Wachi.
\newblock The analogue of {A}ldous’ spectral gap conjecture for the
  generalized exclusion process.
\newblock {\em Comm. Algebra}, 54(3):1256--1272, 2026.

\bibitem[LP]{levhari2026wreath}
N.~Levhari and D.~Puder.
\newblock Aldous-type spectral gaps in generalized symmetric groups.
\newblock In preparation.

\bibitem[Mac98]{macdonald1998symmetric}
I.~G. Macdonald.
\newblock {\em Symmetric functions and {H}all polynomials}.
\newblock Oxford university press, 1998.

\bibitem[Mas03]{maslen2003eigenvalues}
D.~K. Maslen.
\newblock The eigenvalues of {K}ac's master equation.
\newblock {\em Math. Z.}, 243(2):291--331, 2003.

\bibitem[Pir10]{piras2010generalizations}
D.~Piras.
\newblock {\em Generalizations of {A}ldous’ Spectral Gap Conjecture}.
\newblock PhD thesis, Tesi di Laurea, Universit{\'a} degli Studi Roma Tre,
  2010.

\bibitem[PP20]{parzanchevski2020aldous}
O.~Parzanchevski and D.~Puder.
\newblock Aldous’s spectral gap conjecture for normal sets.
\newblock {\em Trans. Amer. Math. Soc.}, 373(10):7067--7086, 2020.

\bibitem[PS23]{puder2023stable}
D.~Puder and Y.~Shomroni.
\newblock Stable invariants of words from random matrices.
\newblock preprint arXiv:2311.17733v3, 2023.

\bibitem[QS23]{quattropani2023mixing}
M.~Quattropani and F.~Sau.
\newblock Mixing of the averaging process and its discrete dual on
  finite-dimensional geometries.
\newblock {\em Ann. Appl. Probab.}, 33(2):1136--1171, 2023.

\bibitem[Sam80]{samuel1980integral}
Stuart Samuel.
\newblock U(n) integrals, 1/{N}, and the {D}e {W}it–’t {H}ooft anomalies.
\newblock {\em J. Math. Phys.}, 21(12):2695--2703, 12 1980.

\bibitem[Whi72]{whitney1972complex}
H.~Whitney.
\newblock {\em Complex analytic varieties}.
\newblock Addison Wesley, 1972.

\end{thebibliography}

~\\
Gil Alon\\
Department of Mathematics and Computer Science, The Open University of Israel, 1 University Road, Raanana 4353701, Israel\\
\noindent\texttt{gilal@openu.ac.il}\\
\\
Doron Puder\\
School of Mathematical Sciences, Tel Aviv University, Tel Aviv, 6997801, Israel\\
and IAS Princeton, School of Mathematics, 1 Einstein Drive, Princeton NJ 08540, USA\\
\texttt{doronpuder@gmail.com}~\\

\end{document}